\documentclass[a4paper, 10pt]{article}
\usepackage[utf8]{inputenc}
\usepackage{a4wide}
\usepackage{amsmath, amssymb, amsthm, amsfonts}
\usepackage[pdftex]{graphicx}
\usepackage[all, poly]{xy}
\usepackage{mathtools}
\usepackage{layout}
\usepackage{verbatim}
\usepackage{bbold}
\numberwithin{equation}{section}
\usepackage{fancyhdr}
\usepackage{comment}
\usepackage{setspace}
\usepackage{mathrsfs}
\usepackage{eufrak}
\usepackage{bbm}
\usepackage{ stmaryrd }
\DeclareMathAlphabet{\mathpzc}{OT1}{pzc}{m}{it}
\usepackage{tikz}
\usepackage{wasysym}

\newcommand{\foo}[1]{%
\begin{tikzpicture}[#1]%
\draw (0,0) -- (1ex,1.8ex);%
\draw (1ex,1.8ex) -- (2ex,0);%
\draw (0,0) -- (2ex,0);%
\draw (0.5ex,0) -- (1.25ex,1.3ex);%
\end{tikzpicture}%
}

\newcommand{\fooo}[1]{%
\begin{tikzpicture}[#1]%
\draw (0,0) -- (0.7ex,1.26ex);%
\draw (0.7ex,1.26ex) -- (1.4ex,0);%
\draw (0,0) -- (1.4ex,0);%
\draw (0.25ex,0) -- (0.875ex,0.91ex);%
\end{tikzpicture}%
}
\newcommand{\tot}[1]{%
\begin{tikzpicture}[#1]%
\draw (0.1,0) arc (-65:245:0.13cm);
\draw (0.1,0) -- (0.23,0);
\draw (-0.13,0) -- (0.00,0);
\draw (0,0) -- (0,0.23);
\end{tikzpicture}%
}
\newcommand{\pullback}[0]{{}_{d_1}\underset{\X_0}{\kern-\scriptspace{\times}}{}_{d_0}}
\def\barroman#1{\sbox0{#1}\dimen0=\dimexpr\wd0+1pt\relax
  \makebox[\dimen0]{\rlap{\vrule width\dimen0 height 0.06ex depth 0.06ex}%
    \rlap{\vrule width\dimen0 height\dimexpr\ht0+0.03ex\relax 
            depth\dimexpr-\ht0+0.09ex\relax}%
    \kern.5pt#1\kern.5pt}}

\DeclareMathOperator{\ca}{\mathcal{C}}
\DeclareMathOperator{\tr}{\text{\normalfont Tr}}
\DeclareMathOperator{\C}{\mathbb{C}}
\DeclareMathOperator{\p}{\mathcal{P}}
\DeclareMathOperator{\Lie}{\mathcal{L}}
\DeclareMathOperator{\m}{\mathcal{M}}
\DeclareMathOperator{\R}{\mathbb{R}}
\DeclareMathOperator{\X}{\mathpzc{X}}
\DeclareMathOperator{\Y}{\mathpzc{Y}}

\DeclareMathOperator{\fg}{\mathfrak{g}}

\DeclareMathOperator{\n}{\mathcal{N}}

\DeclareMathOperator{\fX}{\mathfrak{X}}
\DeclareMathOperator{\fa}{\mathfrak{a}}
\DeclareMathOperator{\fY}{\mathfrak{Y}}

\DeclareMathOperator{\N}{\mathbb{N}}
\DeclareMathOperator{\Z}{\mathbb{Z}}

\numberwithin{equation}{section}
\theoremstyle{plain} 
\newtheorem{theorem}[subsubsection]{Theorem}

\newtheorem{lemma}[subsubsection]{Lemma}
\newtheorem{corollary}[subsubsection]{Corollary}

\theoremstyle{definition}
\newtheorem{definition}[subsubsection]{Definition}
\newtheorem{example}[subsubsection]{Example}

\theoremstyle{remark} 
\newtheorem{remark}[subsubsection]{Remark}
\newtheorem{construction}[subsubsection]{Construction}

\title{Classification Theory}
\author{A. H. G. Milor}

\begin{document}
\setlength{\headheight}{28pt}

\pagestyle{fancyplain}

\lhead{\fancyplain{}{\thepage}}
\chead{}
\rhead{\fancyplain{}{\textit{\leftmark}}}
\lfoot{}
\cfoot{}
\rfoot{}
\onehalfspacing

\begin{titlepage}
\begin{center}

\textsc{\LARGE About extending the Chern-Weil classification to the simplicial principal bundles}

\vspace{2cm}

\includegraphics[width=0.3\textwidth]{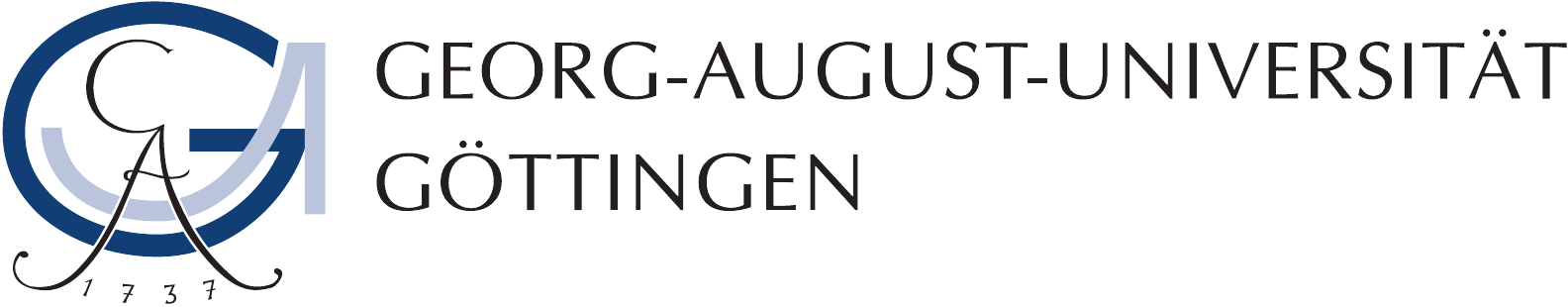}

\vspace{1.5cm}

\Large Master thesis in mathematics \\
\Large submited to the faculty of mathematics and informatic\\
\Large of the Georg-August-university Göttingen\\
the December 17, 2022

\vspace{1cm}
\small{of}\\

\large{Abel Henri Guillaume Milor}
\vspace{1cm}

\small{First advisor}\\

\large{Pr. Chenchang Zhu}

\vspace{1cm}

\small{Second advisor}\\

\large{Dr. Miquel Cueca Ten}

\end{center}
\end{titlepage}

\tableofcontents
\begin{center}
    This paper contains 8 self-made pictures
\end{center}
\newpage
\begin{abstract}
    After introducing the simplicial manifolds, such as the different ways of defining the differential forms on them, we summarized a canonical way of calculating the characteristic classes of a $G$-principal bundle by computing them on the classifying bundle $EG\longrightarrow BG$. Finally, we calculated the first Pontryagin class on the classifying bundle of the Lie matrix groups and showed that for certain of them, the computed form is equal to the symplectic form on $BG$ given by some authors (Cueca, Zhu\cite{zhu} and Weinstein\cite{weinstein}), up to a constant coefficient. 
\end{abstract}
\section{Introduction}
In geometry, the \textit{simplex} is the simplest polytope in a certain finite real space, i.e. the one generated by $n+1$ different vertices in $\R^n$. In $\R$, it is simply a segment, then a triangle in $\R^2$, a tetrahedron in $\R^3$... This is a very basic structure with a lot of applications in mathematics and particularly in geometry, like the triangulation, or the singular (co)homology. \\
That is the reason why S. Eilenberg and A. Zilber introduced the concept of \textit{simplicial set}, in order to describe as a category a topological space that can be decomposed in simplices. It is a collection of sets $\{\X_i\}_{i\in \N_0}$ where each element of $\X_n$ represents an $n$-simplex. Since the faces of each simplex are simplices of lower degree, we have a set of descending maps $d_i: \X_n\longrightarrow \X_{n-1}$ showing which $n-1$-simplices are the faces of which $n$-simplex. In the same way, we have ascending maps showing how we can degenerate each simplex. The original topological space is really encoded by this structure since we can reconstruct it with the \textit{geometric realization}, which is a glueing of the simplices with respect to these ascending and descending maps.\\
Such a construction has a categorical interest when we consider the \textit{nerve} of a (small) category: The composition of $n$ morphisms, with all the possible subcompositions, can be associated with an $n$-simplex. This allows a new approach to the categorical world with simplicial sets and we can construct a corresponding topological space using the geometric realization. If such a structure can be constructed with sets, we can do the same with other kinds of objects. Particularly, if $\text{Obj}(C)$ and $\text{Morph}(C)$ have a manifold structure, we might sometimes have a simplicial manifold structure on the nerve. It is for example the case if $C$ is a Lie group or a \v{C}ech groupoid. \\
As we introduce the concept of simplicial manifolds, some natural questions shall arise: Is it an extension of the concept of manifolds? It can be easily thought of in that way, as we will see it. In this regard, another question appears: How to extend the concept of differential forms on it? Unfortunately, there is no canonical way for extending them. We can define the differential forms in different ways which are not equivalent to each other: we will present here two of them. The first one is considering the topological space resulting from the geometric realization as a glueing of the manifolds that are the simplices. In this case, the idea is thus to define differential forms on them, such that they vary smoothly and can be glued along the simplices. This leads to a complex analogous to the De Rham complex. The second one is inspired by the \v{C}ech double complex: we obtain this double cochain by considering the De Rham complex for each $\m_i$, with a horizontal derivative given by the descending maps. Then, the differential forms are defined as the elements of the total complex. These two definitions are nevertheless related: the second one can be obtained by integrating along the simplices of the first one. This map is a quasi-isomorphism: although these definitions are not equivalent, they have to give rise to the same cohomology.\\
But if the simplicial structure can be constructed with manifolds, we can also use it with $G$-principal bundles for a given Lie group $G$. We obtain in this way the definition of simplicial $G$-principal bundles and we can also extend some constructions on the $G$-principal bundles. In particular, S. Chern and A. Weil constructed a homomorphism which, for a certain invariant homogeneous polynomial on $\fg$ (the Lie algebra of $G$), and using a connection of the bundle, gives a certain cohomology class of the base manifold. They also proved that this class only depends on the isomorphy class of the principal bundle, and can thus be used for classifying the bundles. This homomorphism such as these characteristic classes can be extended to the simplicial principal bundles. \\
Among the simplicial principal bundles, there is a particular one which is denoted $EG\longrightarrow BG$ and is called the classifying bundle. This simplicial principal bundle has the property that, knowing the characteristic class corresponding to a certain invariant homogeneous polynomial on it, there is a canonical way to calculate the characteristic class of any $G$-principal bundle for the same polynomial. Shulman, in his thesis \cite{shulman}, proposed a way to compute the characteristic classes on the classifying bundle by introducing a canonical connection on it. This gives a general canonical way for computing the characteristic classes that was missing before since the choice of the connection is in general not canonical. \\
Surely, the most useful and common Lie groups are the Lie matrix groups, which are all the Lie subgroups of $GL(n)$. Furthermore, very common polynomials for computing the characteristic classes of matrix Lie groups are the Pontryagin and the Chern polynomials, which are related to each other. That is the reason why we want to present in the last part of this paper a calculation, using the connection presented by Shulman, of the first Pontryagin class of the matrix Lie groups, such as the Spin group. \\
Finally, M. Cueca and C. Zhu in \cite{zhu}, along with A. Weinstein in \cite{weinstein} also worked on the base $BG$ of the classifying bundle, also called the classifying space. They described a canonical way to define a symplectic form on this space. We want to finish this paper by proving that in certain cases, this symplectic form is equal to the characteristic class of the first Pontryagin class on the classifying bundle, up to a constant coefficient depending on the group. \\
The main aim of this paper is to reach this last equality, but also to make a presentation of all the topics mentioned in this introduction. The paper was thought to be understandable to a reader having the usual knowledge of differential geometry. 

\section*{Acknowledgement}
I would like to thank Chenchang Zhu and Miquel Cueca for having given to me the possibility to do this master thesis, and for having led me to discover these new mathematical notions.\\
Thanks also to Federica Pasqualone for having advised me on the book \textit{Curvature and characteristic classes} of Johan Dupont.\\
Thanks also to my parents for their continuous presence beside me. \\
I also want to give special thanks to Moataz Dawor, for all the time spent working together, although on different topics. 
\newpage
\section{Simplicial sets, simplicial manifold and nerves}
The main objects we will consider in this paper are the simplicial manifolds and the simplicial sets. They can be thought of as a collection of simplices of different dimensions, satisfying some nice arrangement. We will give in this first section a brief overview of these objects with some interesting properties and examples.
\subsection{Simplicial objects}
Each simplicial object is a categorical construction, built with the same structure, which is implied by the simplex category. 
\begin{definition}
We call the \underline{simplex category}, denoted $\Delta$, the category having the sets $[n]:=\{0, 1, 2, \dots , n\}$ as objects and the order-preserving maps as morphisms. 
\end{definition}
Some morphisms of $\Delta$ are of particular interest. We have: 
$$\sigma_i^n:[n+1]\longrightarrow [n], i\mapsto 
\begin{dcases}
j ,& j\leq i\\
j-1 ,& j>i
\end{dcases}\qquad \forall n\in\N_0, \; i\in [0,n]$$ 
$$\delta_i^n:[n-1]\longrightarrow [n], i\mapsto 
\begin{dcases}
j ,& j< i\\
j+1 ,& j\geq i
\end{dcases}\qquad \forall n\in\N, \; i\in [0,n]$$
\begin{remark}
If the $n$ is clear with the context, we simply write $\sigma_i$ and $\delta_i$.
\end{remark}

These morphisms are elementary for the simplex category $\Delta$ in the sense that we can construct all the morphisms of the category from them.
\begin{lemma}
All the morphisms of $\Delta$ can be written as a composition of the following sort: 
$$\delta_{i_k}\circ \dots \circ \delta_{i_1}\circ \sigma_{j_1}\circ \dots\circ \sigma_{j_l}$$
for $i_1\leq  \dots \leq i_k$ and $j_1\leq\dots \leq j_l$. 
\end{lemma}
\begin{proof}
Let $\mu:[n]\longrightarrow[m]$ be an order-preserving morphism, i.e. for all $i\in [n]$ we have $\mu(i)\leq \mu(i+1)$. Then, consider:
$$\hat{\sigma}:=\sigma_{j_1}\circ \dots\circ \sigma_{j_l}\qquad \forall\, j_x\in [n]\; \text{such that}\; \mu(j_x)= \mu(j_x+1)$$
with $j_1\leq\dots \leq j_l$. Thus $\mu(j_x)= \mu(j_x+1)$ implies that $\hat{\sigma}(j_x)= \hat{\sigma}(j_x+1)$. Since $\mu$ is order-preserving, from $\mu(i)= \mu(j)$ for $i<j$ we obtain $\mu(i)= \mu(k)= \mu(j)$ for all $i\leq k\leq j$. We can conclude that $\mu(i)=\mu(j)$ implies $\hat{\sigma}(i)=\hat{\sigma}(j)$. Contrariwise, $\hat{\sigma}(i)=\hat{\sigma}(i+1)$ gives $\mu(i)=\mu(i+1)$ and since $\hat{\sigma}$ is preserving the order, we get that $\hat{\sigma}(i)=\hat{\sigma}(j)$ implies $\mu(i)=\mu(j)$.\\
There is a canonical map $\hat{\delta}: [n-l]\longrightarrow [m]$, such that $\hat{\delta}\circ \hat{\sigma}= \mu$, and clearly, $\hat{\delta}$ is injective. Since both $\hat{\sigma}$ and $\mu$ are order-preserving and $\hat{\sigma}$ is surjective, so is $\hat{\delta}$. And an injective order-preserving map can be uniquely defined by the elements of the target set which do not belong to the image. Thus, define:
$$\hat{\delta}:= \delta_{i_k}\circ \dots \circ \delta_{i_1}\qquad \forall\, i_x\in [m]\; \text{such that}\; i_x\notin \text{Im}(\mu)$$
for $i_1\leq  \dots \leq i_k$. One can easily check that this definition fulfils $\hat{\delta}\circ \hat{\sigma}= \mu$. So we proved our statement. 
\end{proof}
We can now introduce the functorial definition of a simplicial object:
\begin{definition}
Let $\ca$ be a category. Then a \underline{simplicial object} in $\ca$ is a contravariant functor $F: \Delta^{opp} \longrightarrow \ca$. The simplicial object is generally denoted by $\X_{\bullet}$ and the image of $[n]$ is denoted by $\X_{n}$. We denote $d_i$, called the \underline{face maps}, and $s_i$, called the \underline{degeneracy maps}, the respective images of $\delta_i$ and $\sigma_i$. \\
If we have a covariant functor instead, i.e. $F: \Delta \longrightarrow \ca$, we call it a \underline{cosimplicial object} in $\ca$. 
\end{definition}
This definition might however be unsuitable for a lot of calculations and considerations we are used to do with simplicial objects. That's why the following equivalent definition can be more useful in certain contexts:
\begin{lemma}\label{simplicial equations}
Let $\left\{\X_n\right\}_{n\in \N_0}$ be a collection of objects in a certain category $\ca$, with morphisms $d_i^n: \X_n \longrightarrow \X_{n-1}\; \forall n>0, i\in [0,n]$ and $s_i^n: \X_n \longrightarrow \X_{n+1}\; \forall n, i\in [0,n]$. Then $\left\{\X_n\right\}$ is (the image of) a simplicial object iff the following equations are fulfilled:
\begin{equation}\label{eq: simplicial equation}
\begin{aligned}
    d_i\circ d_j =& d_{j-1}\circ d_i \qquad \text{if}\;\; i<j\\
    s_i\circ s_j =& s_{j}\circ s_{i-1} \qquad \text{if}\;\; i>j\\
    d_i\circ s_j=&
    \begin{dcases}
    s_{j-1}\circ d_i &\qquad \text{if}\;\; i<j\\
    \text{\normalfont id} &\qquad \text{if} \;\; i\in \{j, j+1\}\\
    s_{j}\circ d_{i-1} &\qquad \text{if}\;\; i>j+1\\
    \end{dcases}
\end{aligned}
\end{equation}
They are called the \textit{\underline{simplicial equations}}.
\end{lemma}
\begin{proof}
Let us consider the morphisms in $\Delta$. Because of the definition of $\delta_i$ and $\sigma_i$, it can be easily checked that:
\begin{equation}\label{eq:cosimplicial equation}
\begin{aligned}
    \delta_j \circ \delta_i =& \delta_i\circ \delta_{j-1}\qquad \text{if}\;\; i<j\\
    \sigma_j\circ  \sigma_i =& \sigma_{i-1}\circ \sigma_j \qquad \text{if}\;\; i>j\\
    \sigma_j\circ \delta_i =&
    \begin{dcases}
    \delta_i\circ \sigma_{j-1} &\qquad \text{if}\;\; i<j\\
    \text{id} &\qquad \text{if} \;\; i\in \{j, j+1\}\\
    \delta_{i-1}\circ \sigma_j &\qquad \text{if}\;\; i>j+1\\
    \end{dcases}
\end{aligned}
\end{equation}
Then, applying the contravariant functor of the simplicial object, we deduce that the image of the functor will fulfil the equations of the statement, which are dual to the equations (\ref{eq: simplicial equation}). \\
Let us consider the other direction: suppose that $\X_n, d_i, s_i$ have the simplicial equations given in the statement. We can easily construct the contravariant functor $F: \Delta^{\text{opp}}\longrightarrow \ca$, given by $F([n])= \X_n$, $F(\delta_i^k)= d_i^k$ and $F(\sigma_i^k)= s_i^k$ and:
\begin{align*}
    \text{id}_{[n]}&\mapsto \text{id}_{\X_n}\\
    \sigma_{j_1}\circ \dots\circ \sigma_{j_l} &\mapsto s_{j_l}\circ \dots \circ s_{j_1}\\
    \delta_{i_k}\circ \dots \circ \delta_{i_1}&\mapsto d_{i_1}\circ \dots\circ d_{i_k}\\
    \delta_{i_k}\circ \dots \circ \delta_{i_1}\circ \sigma_{j_1}\circ \dots\circ \sigma_{j_l}  &\mapsto s_{j_l}\circ \dots \circ s_{j_1}\circ d_{i_1}\circ \dots\circ d_{i_k}
\end{align*}
for $i_1\leq \dots \leq i_k$ and $j_1\leq\dots \leq j_l$. We need to prove that it is a functor, i.e. that the composition is preserved. Let $\phi$ and $\psi$ be two maps of $\Delta$ such that $\phi\circ \psi$ is well-defined. We can write them so:
$$\phi= \delta_{i_k}\circ \dots \circ \delta_{i_1}\circ \sigma_{j_1}\circ \dots\circ \sigma_{j_l}\qquad \psi= \delta_{p_r}\circ \dots \circ \delta_{p_1}\circ \sigma_{q_1}\circ \dots\circ \sigma_{q_s}$$
Then the combination of both will be a certain sequence of $\sigma_i$ and $\delta_i$. But using the third commutation rule of $\sigma$ and $\delta$ in $\Delta$ in (\ref{eq:cosimplicial equation}) we can move all the $\sigma$'s on the right side of the composition and all the $\delta$'s on the left side. Remark that again because of the third rule, the number of $\sigma$ and $\delta$ might decrease since $\sigma_i\circ \delta_i= \text{id}=\sigma_i\circ \delta_{i+1}$, and we will suppose this happens $n$ times. Furthermore, using the first rule, we can order the $\delta$ in decreasing order and using the second one, we can order all the $\sigma$ in increasing order. We get something like: 
\begin{equation}\label{ordered maps}
\phi\circ \psi= \delta_{a_{r+k-n}}\circ \dots \circ \delta_{a_{1}}\circ \sigma_{b_{1}}\circ \dots\circ \sigma_{b_{l+s-n}}
\end{equation}
with $a_1\leq \dots \leq a_{r+k-n}$ such as $b_1\leq \dots \leq b_{s+l-n}$. Thus, the image of $\phi\circ \psi$ through $F$ is given by: 
$$F(\phi\circ\psi)=s_{b_{l+s-n}}\circ \dots\circ s_{b_{1}}\circ  d_{a_{1}}\circ \dots \circ d_{a_{r+k-n}}$$
Then, using the properties of $\X_n, d_i, s_i$, which are dual to the commutation properties of $\Delta$, we can reverse the commutations we did for obtaining $(\ref{ordered maps})$, and we get:
\begin{align*}
    F(\phi\circ\psi)&= s_{q_s}\circ\dots\circ s_{q_1}\circ d_{p_1}\circ\dots\circ d_{p_r} \circ  s_{j_l}\circ \dots \circ s_{j_1}\circ d_{i_1}\circ \dots \circ d_{i_k}\\
    &= F(\psi)\circ F(\phi)
\end{align*}
Thus the contravariant functor $F$ is well-defined, and $(\X_n, d_i, s_i)$ represents correctly a simplicial object in $\ca$.
\end{proof}
\begin{corollary}\label{cosimplicial equation}
    Let $\left\{\X_n\right\}_{n\in \N_0}$ be a collection of objects in a certain category $\ca$, with morphisms $d_i^n: \X_n \longrightarrow \X_{n+1}\; \forall n, i\in [0,n+1]$ and $s_i^n: \X_n \longrightarrow \X_{n-1}\; \forall n>0, i\in [0,n-1]$. Then $\left\{\X_n\right\}$ is (the image of) a cosimplicial object iff the following equations are fulfilled:
\begin{equation}\label{eq: real cosimplicial equations}
\begin{aligned}
    d_j \circ d_i =& d_i\circ d_{j-1}\qquad \text{if}\;\; i<j\\
    s_j\circ  s_i =& s_{i-1}\circ s_j \qquad \text{if}\;\; i>j\\
    s_j\circ d_i =&
    \begin{dcases}
    d_i\circ s_{j-1} &\qquad \text{ if}\;\; i<j\\
    \text{\normalfont id} &\qquad \text{if} \;\; i\in \{j, j+1\}\\
    d_{i-1}\circ s_j &\qquad \text{if}\;\; i>j+1\\
    \end{dcases}
\end{aligned}
\end{equation}
They are called the \textit{\underline{cosimplicial equations}}.
\end{corollary}
\begin{proof}
    Analogous to the proof of lemma \ref{simplicial equations}.
\end{proof}
\begin{remark}
Using this definition, we can conclude that a natural transformation $\alpha$ between two simplicial objects $\X_{\bullet}$ and $\Y_{\bullet}$ of $\ca$ can be described as morphisms $\alpha_n : \X_n\longrightarrow \Y_n$ in $\ca$ such that $\alpha_{n-1}\circ d_i^{\X}= d_i^{\Y}\circ \alpha_n$ and $\alpha_{n+1}\circ s_i^{\X}= s_i^{\Y}\circ \alpha_n$ for all $n$ and $i$, where $d_i^{\X}, s_i^{\X}$ and $d_i^{\Y}, s_i^{\Y}$ are the face and degeneracy maps of respectively $\X_{\bullet}$ and $\Y_{\bullet}$. We call such a natural transformation a \underline{simplicial morphism}. 
\end{remark}
\begin{definition} Intuitively, we have:
\begin{itemize}
    \item A \underline{simplicial set} is a simplicial object in the category $\text {Set}$ of sets.
    \item A \underline{simplicial manifold} is a simplicial object in the category $\text {Mfd}$ of manifolds.
\end{itemize}
\end{definition}
We can deduce the following properties for the face and degeneracy maps:
\begin{lemma}
Let $\X_{\bullet}$ be a simplicial set. Then, the face maps $d_i$ are surjective and the degeneracy maps $s_i$ are injective for all $i$. If $\X_{\bullet}$ is a simplicial manifold, $s_i$ is furthermore an immersion. 
\end{lemma}
\begin{proof}
We prove first these statements for $\X_{\bullet}$ a simplicial set. \\We know from Lemma \ref{simplicial equations} that $d_i\circ s_i= \text{id}$. Thus $d_i(\text{im}(s_i))=\X_n$, implying the surjectivity of $d_i$, such as the injectivity of $s_i$.\\
Consider now a simplicial manifold, with smooth face and degeneracy maps. Taking the derivative, we have:
$$T_{s_i(p)}(d_i)\circ T_p(s_i)=T_p(\text{id})= \text{id}_{T_p \X_n}\qquad \qquad \forall p\in \X_n$$
Thus $T_p(s_i)$ is injective for all $p$, making of $s_i$ an immersion. 
\end{proof}

\subsection{Examples of simplicial sets and manifolds}
\begin{example}\label{simplicial simplex}
Let us see a formal example. For this, define the simplicial set $\Delta[n]$ with $\Delta[n]_k$ being the set of order-preserving maps from $[k]$ to $[n]$, where the  face and degeneracy maps, for $f:[k]\longrightarrow [n]$, are given by:
$$d_i(f): x\mapsto
\begin{dcases}
f(x) ,& x<i \\
f(x+1) ,& x\geq i
\end{dcases} \qquad s_i(f): x\mapsto
\begin{dcases}
f(x) ,& x\leq i \\
f(x-1) ,& x > i
\end{dcases}$$
This simplicial set is called the \underline{Yoneda $n$-simplex}, denoted $\Delta[n]$. The origin of this name comes from the fact that by applying the Yoneda lemma, we have for each simplicial set $\X_{\bullet}$ the property:
$$\text{Hom}_{nt}(\Delta[n], \X_{\bullet})\cong \X_n$$
where $\text{Hom}_{nt}$ is the set of natural transformations. 
\end{example}
\begin{example}\label{horn}
Another example of a simplicial set that we commonly find in the literature is the \underline{Horn}, denoted $\Lambda[n, p]$, for $0\leq p \leq n$. It is a subset of $\Delta[n]$, where 
$$\Lambda[n, p]_k= \left\{f\in \Delta[n]_k \middle| \{0, \dots, p-1, p+1, \dots , n\}\nsubseteq \{f(0), \dots , f(k)\}\right\}$$
The face and degeneracy maps are the same as the ones of $\Delta[n]$, restricted to $\Lambda[n, p]_k$. 
\end{example}
\begin{example}
Let us see here a trivial example of a simplicial manifold, called the \underline{\textit{constant simplicial}} \underline{\textit{ manifold}}. Let $\m$ be a manifold, then we define $\underline{\m}_{\bullet}$ as:
$$\underline{\m}_n=\m\qquad d_i=\text{id}_{\m}\qquad s_i=\text{id}_{\m}\qquad \forall \; n, i, j$$
With this example, we do now see in which sense the simplicial manifolds are an extension of the world of manifolds. 
\end{example}
\begin{example}\label{ex geo 1}
Let us see here a visual example of a simplicial manifold. We will not define it rigorously but give a picture of what a simplicial manifold could be like. Let us consider the rotation of a triangle along a cylinder, as represented in the following image:\\
\begin{center}
   \includegraphics[width=0.3\textwidth]{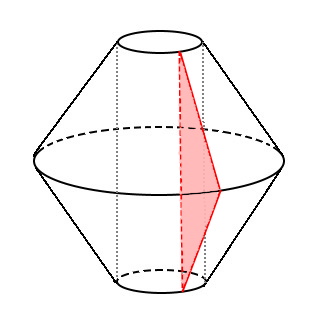} 
\end{center}
This can be seen as a set of triangles: each position of the triangle during the rotation is a new triangle (see in red). For each triangle, there are its corresponding vertices and edges. Let $\X_0$ be the set of vertices of all the triangles, $\X_1$ the set of all the edges, and $\X_2$ the set of all the triangles. All these sets include the degenerate edges or triangles. Generally, let $\X_n$ be the set of all $n$-simplices (they are but all degenerate for $n>2$). The face maps send to the faces of the simplex, and the degeneracy maps ``degenerate" the simplex by splitting one of its vertices. If we want to represent the manifolds $\X_0, \X_1$ and $\X_2$, we get the following picture, where the dotted lines represent the degenerate elements: \\
\begin{center}
     \includegraphics[width=1\textwidth]{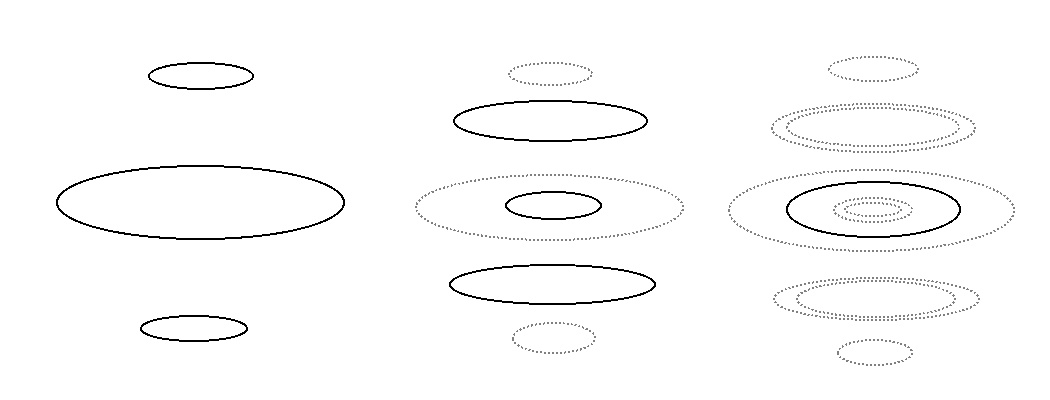} 
\end{center}
$$\qquad \X_0\quad\qquad\qquad\qquad \qquad \qquad \qquad \X_1\qquad\qquad\qquad\qquad\qquad \qquad \X_2$$
\end{example}

\subsection{Nerves}
In this subsection, we will focus on another example of a simplicial set/ simplicial manifold, which is of particular importance in the following developments of this paper. 
\begin{definition}
Let $\ca$ be a small category. We can then define the simplicial set $\n\ca_{\bullet}$, called the \underline{nerve of $\ca$}, as: 
$$\n\ca_0= Obj(\ca)\quad \n\ca_{n}=\left\{(f_1, \dots, f_n)\in Morph(\ca)^n\middle|f_n\circ\dots\circ f_1\; \; \text{is a valid composition}\right\} \; \; \forall n>0$$

The degeneracy maps, for $n=1$, are defined as follows: $d_0(f)$ is the target of $f$, and $d_1(f)$ is its source. For $n\geq 2$, they are defined as:
\begin{align*}
    d_0((f_1, f_2, \dots, f_n))=(f_2, \dots, f_n) \;,&\qquad d_n((f_1, \dots, f_{n-1}, f_n))=(f_1, \dots, f_{n-1})\\
    d_i((f_1,\dots, f_i, f_{i+1},\dots , f_n))&= (f_1, \dots, f_{i+1}\circ f_{i},\dots , f_n)\;\; \forall i\neq 0, n
\end{align*}
The degeneracy map, for $n=0$, is defined as $s_0(x)= \text{id}_x$, and for $n\geq 1$, is defined as:
$$s_i((f_1,\dots, f_i, f_{i+1},\dots , f_n))= (f_1,\dots, f_i,\text{id}, f_{i+1},\dots , f_n)$$
\end{definition}
\begin{remark}
In the literature, we can find some other definition of a nerve: Especially, in \cite{zhu}, the nerve is defined the reverse of our definition: 
$$\n\ca_{n}=\left\{(f_n, \dots, f_1)\in Morph(\ca)^n\middle|f_n\circ\dots\circ f_1\; \; \text{is a valid composition}\right\} \; \; \forall n>0$$
Because of this, all the face and degeneracy maps does also act reversely: for $n=1$, $d_0(f)$ is the source and $d_1(f)$ is the target... As we will see it in the last section, this can produces changes in formulas, and need thus to be checked conscientiously. 
\end{remark}
\begin{remark}\label{ex geo 2}
We can have a very geometric picture of the nerve of a category $\ca$. Indeed, we can represent the sequence $(f_1, \dots, f_n)$ as a sequence of directed edges connecting $n+1$ vertices (the sources and targets of the $f_i$'s), and we can add the edges representing each possible composition of $f_i$'s. We obtain then an $n$-simplex, as shown in the picture below. The face map $d_i$, with the composition of $f_i$ and $f_{i+1}$, allows us to ``skip" the $n$-th vertex (we start counting at $0$) and we obtain a face of the $n$-simplex. The degeneracy map $s_i$ is splitting the $i$-th vertex with the map id, giving us a degenerate $n+1$-simplex. All these transformations are represented below:
\begin{center}
   \includegraphics[width=1\textwidth]{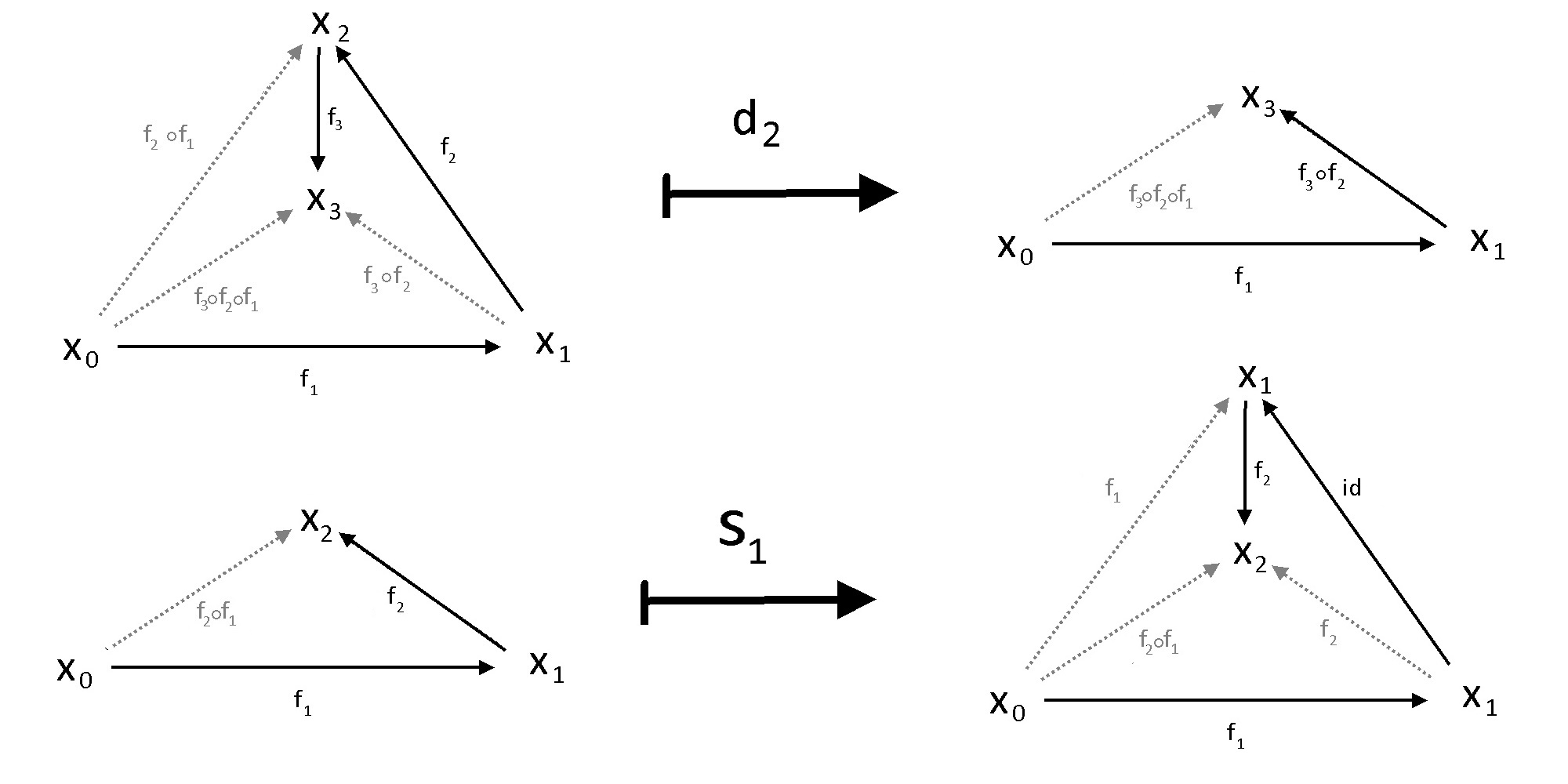}
\end{center}
\end{remark}
The reason why the nerves are particularly interesting is that they contain the whole information of the category. We can thus apprehend most parts of the category theory using specific simplicial sets. We get for example the following characterizations:
\begin{lemma}
Let $\X_{\bullet}$ be a simplicial set. Then, $\X_{\bullet}$ is isomorphic to the nerve of a certain category $\ca$ iff it satisfies the following condition (called the \underline{Segal condition}):
$$\X_{n}\cong \underbrace{\X_1 \pullback \dots \pullback \X_1}_{n\text{-times}}$$
\end{lemma}
\begin{proof}
See \cite{segal cond}. 
\end{proof}

\begin{lemma}
Let $\X_{\bullet}$ be a simplicial set. Then, $\X_{\bullet}$ is isomorphic to the nerve of a groupoid iff for all $n\geq 2$ and for all $0\leq i\leq n$, the following maps are bijections:
$$\text{\normalfont Hom}_{nt}(\Delta[n], \X_{\bullet})\longrightarrow \text{\normalfont Hom}_{nt}(\Lambda[n, i], \X_{\bullet})$$
\end{lemma}
\begin{proof}
See introduction of \cite{zhu2}.
\end{proof}
\begin{remark}
In the specific case of a Lie group $G$, since all the morphisms are composable with the other ones, it follows that the nerve will have the form $\n G_0=\{\ast\}$, $\n G_n=G^n$. Since $G$ has a manifold structure, it follows that $\n G_{\bullet}$ is a simplicial manifold. (one easily checks that the face and degeneracy maps are smooth: it follows from the smoothness of the composition on a Lie group). 
\end{remark}
\subsection{Realization of a simplicial set}
In the example \ref{ex geo 1} and in the remark \ref{ex geo 2}, we associated several times the structure of the simplicial set to a geometrical structure, made of simplices glued together. In this subsection, we shall first extend our intuitive geometrical comprehension, and then introduce a formal construction corresponding to it, called the realization. Indeed, the intuition of example \ref{ex geo 1} and remark \ref{ex geo 2} can be generalized:
\begin{remark}\label{way of seeing simp set}
A nice way to figure out the simplicial sets is to represent each element of $\X_n$ by a copy of the $n$-simplex. The face map $d_i$ is sending to the copy of the $n-1$-simplex corresponding to the face of the $n$-simplex opposed to the $i-th$ vertex while the degeneracy map $s_i$ is sending to the (degenerate) $n+1$-simplex where we duplicate the $i$-th vertex (giving us a degenerated edge). Thus, degeneracy and face maps imply a numbering of the vertices and one can remark that the order of this numbering is preserved by reducing to a face of the simplex, using the commuting the face maps, as seen in Lemma \ref{simplicial equations}. Often, it might be useful to ignore the degenerate faces, since they can be deduced from the non-degenerate faces and because it simplifies the picture. 
\end{remark}
\begin{example}
Using this remark, some of the constructions presented before make more sense. Indeed, the Yoneda $n$-simplex $\Delta[n]$ presented in Example \ref{simplicial simplex} can now be considered as the usual $n$-simplex. Analogously, the horn $\Lambda[n,p]$, presented in Example \ref{horn}, can be understood as the $n$-simplex without its interior and without the interior of its $p$-th face (face opposed to the $p$-th vertex). 
\end{example}
These considerations should have improved the comprehension of the reader unused to these concepts. Let us now consider a more precise and rigorous construction, following this idea. We need first a clear definition of the $n$-simplex. 

\begin{definition}
We define the \underline{n-simplex} as being the n-dimensional subset of $\R^{n+1}$ given by:
$$\left\{(t_0, \dots, t_n)\middle| \sum_{i=0}^{n}t_i = 1 \;\; \text{and}\;\; t_i\geq 0\; \forall i
\right\}$$
We denote it $\foo{}^n$.
\end{definition}
\begin{center}
    \includegraphics[width=0.4 \textwidth]{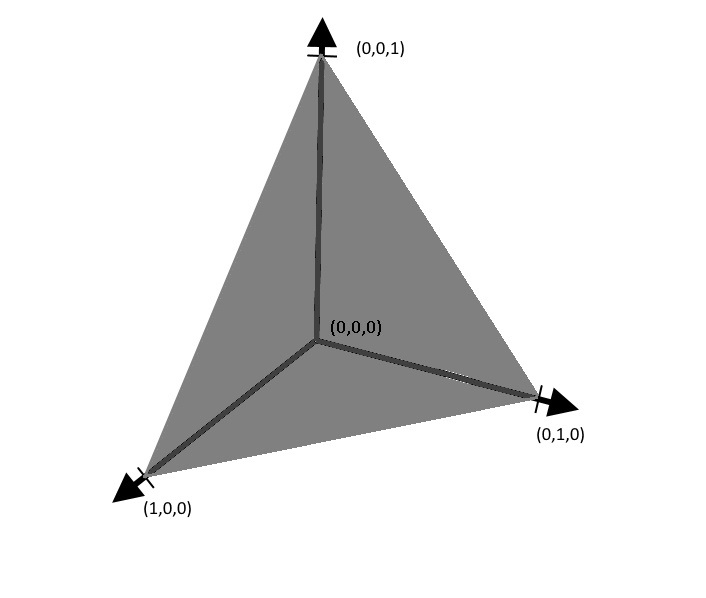}
\end{center}
\begin{remark}
We call the coordinates implied by the $t_i$'s on $\foo{}^n$ the \underline{barycentric coordinates}. Intuitively, $t_i$ is near 1 when the point is near the $i$-th vertex and near 0 when it is near the opposite face. Remark that these coordinates are not linearly independent.
\end{remark} 

\begin{construction}
Consider the sets $\left\{\foo{}^n\right\}_{n\in\N_0}$. Then, there is a collection of descending maps given by $\eta_i$ and of ascending maps given by $\tau_i$, where:
\begin{equation}\label{eq: maps on the simplex}
    \begin{aligned}
    \tau_i: \foo{}^n\longrightarrow \foo{}^{n+1},\; i\in [0, n+1]:&\qquad  (t_0, \dots, t_{i-1}, t_{i},t_{i+1},\dots , t_n) \mapsto(t_0, \dots, t_{i-1},0,  t_i , t_{i+1},\dots , t_n)\\
    \eta_i: \foo{}^n\longrightarrow \foo{}^{n-1} ,\; i\in [0, n-1]:& \qquad  (t_0, \dots, t_{i-1}, t_{i},t_{i+1},\dots , t_n)\mapsto (t_0, \dots, t_{i-1}, t_{i} + t_{i+1},\dots , t_n)
\end{aligned}
\end{equation}

\end{construction}
\begin{remark}
Let the $i$-th vertex of a simplex be $e_{i+1}$, element of the usual basis of $\R^{n+1}$. Intuitively, the map $\tau_i$ is the inclusion on the $i$-th face of a higher simplex, while $\eta_i$ is a projection on the face opposed to the $(i+1)$-th vertex. For example, $\eta_1$ on $\foo{}^2$ is the projection obtained by contracting the edge $[e_2, e_3]$:
\begin{center}
    \includegraphics[width=0.3 \textwidth]{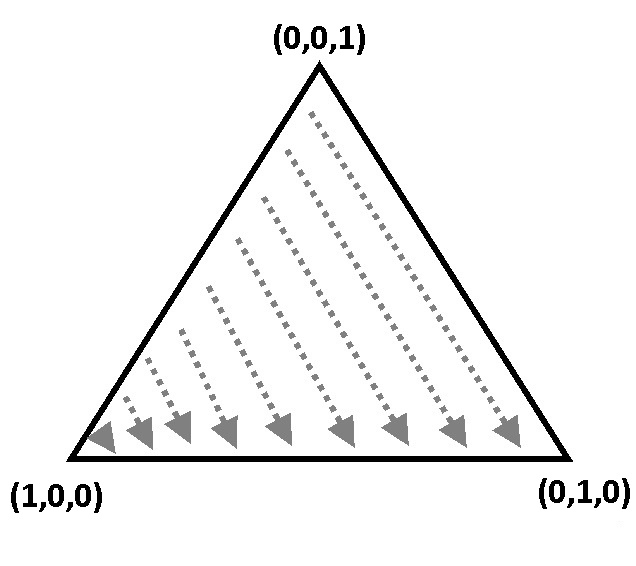}
\end{center}
It can be easily checked that $\left\{\foo{}^n\right\}_{n\in\N_0}$ with the maps (\ref{eq: maps on the simplex}) is a cosimplicial set (and even a cosimplicial manifold). 
\end{remark}
We can now construct the following structure:
\begin{definition}\label{pseudorealization}
Let us consider $\X_{\bullet}$ to be a simplicial set, with face maps $d_i$ and degeneracy maps $s_i$. We obtain the \underline{pseudo-realization}, denoted $\hat{\X}_{\bullet}$, which is defined as:
$$\left\{\hat{\X}_n:= \foo{}^n\times \X_n\right\}_{n\in \N_0} $$

\end{definition}
Remark that the union $\bigsqcup_n \foo{}^n\times \X_n$ does not represent the simplicial sets in the way we wanted it in remark \ref{way of seeing simp set}. Indeed, the $i$-th face of $\foo{}^n$ associated to the element $x\in \X_n$ is different than the simplex $\foo{}^{n-1}$ associated to $d_i(x)\in \X_{n-1}$. Furthermore, the degenerate simplices have nothing to do with the original (non-degenerate) simplex: That's why we need to associate these different elements together by an equivalence relation. This is the purpose of the following definition:
\begin{definition}\label{gluing realization}
Let $\X_{\bullet}$ be a simplicial set. We can then construct the following set, given by:
$$\bigsqcup_n \hat{\X}_n/\sim\; = \bigsqcup_n  \foo{}^n\times \X_n /\sim$$
where the equivalence relation $\sim$ is given by:
\begin{equation}\label{eq: glue the face}
(\tau_i(y), x) \sim (y, d_i(x))
\end{equation}
with $x\in \X_n$ and $y\in \foo{}^{n-1}$ for all $i\in [0, n]$ and all $n>1$, and
\begin{equation}\label{eq: glue the degenerate}
(\eta_i(y), x)\sim (y, s_i(x))
\end{equation}
with $x\in \X_n$ and $y\in \foo{}^{n+1}$ for all $i\in [0, n]$ and all $n>0$.\\
This set is called the \underline{geometric realization} of $\X_{\bullet}$, and is denoted $|\X_{\bullet}|$. 
\end{definition}
\begin{remark}
The geometric realization has a canonical topology in the case where $\X_{\bullet}$ is a simplicial manifold, and thus generate a topological space. 
\end{remark}
\begin{remark}
This construction corresponds to the remark \ref{way of seeing simp set}. Indeed, the condition (\ref{eq: glue the face}) ensures to ``glue" each simplex with the simplices representing its faces, and (\ref{eq: glue the degenerate}) removes the degenerate faces. \\
However some authors prefer to consider the topological construction of \ref{gluing realization} only with the glueing of the faces (\ref{eq: glue the face}), and thus without the one of the degenerates (\ref{eq: glue the degenerate}). This construction is called the \underline{fat realization} of $\X_{\bullet}$ and is denoted $||\X_{\bullet}||$. Most of the constructions we will encounter with the geometric realization can be constructed analogously with the fat realization. 
\end{remark}
\begin{remark}
The pseudo-realization (Definition \ref{pseudorealization}) is a pre-construction of the geometric realization: each object on the geometric realization will be first defined on the pseudo-realization, and we will require that the object respects the glueing properties (\ref{eq: glue the face}) and (\ref{eq: glue the degenerate}).
\end{remark}

\section{De Rahm complex of simplicial manifolds}
The aim of this section is to construct two definitions for the differential forms on simplicial manifolds, extending the one of manifolds, and to introduce a cochain complex of differential forms on them. We shall see that the first definition might be understood as an extension of the \v{C}ech complex, while the second one is more like the usual De Rham complex, defined on the geometric realization. The both are not equivalent to each other but give rise to the same cohomology.  
\subsection{Construction of the cohomology}
In this part, we define the differential forms and their complex for a simplicial manifold. The reason why we define them in this way might not seem obvious to the unfamiliar reader, that is why we gave its relation with the \v{C}ech complex. The relation to the second definition, presented in subsection \ref{other part}, will also make this one more intuitive. 

\begin{definition}
We call a \underline{cosimplicial cochain complex} a cosimplicial object in the category of cochain complex (of vector spaces over $\R$ or $\C$). It is denoted by $V^{\bullet, \ast}$, where $V^{n, \ast}$ is the cochain complex being the image of $[n]\in \Delta$. 
\end{definition}
\begin{remark}
The De Rham functor is a contravariant functor $\Omega^{\ast}: \text{Mnf}^{opp}\longrightarrow \text{CochCompl}$. Thus, by composing it with a simplicial manifold $\m_{\bullet}$ (seen as a contravariant functor), we obtain a covariant functor from $\Delta$ to $\text{CochComp}$ (and thus a cosimplicial cochain complex). In this case, $V^{n, k}= \Omega^k(\m_n)$. We denote by ${s}^*_i: V^{n+1, \ast}\longrightarrow V^{n, \ast}$ and ${d}^*_i: V^{n, \ast}\longrightarrow V^{n+1, \ast}$ the cosimplicial structure maps, given by the pullbacks of the degeneracy and face maps of $\m_{\bullet}$. 
\end{remark}
By this consideration, we obtain canonically a grid $V^{\bullet, \ast}$. But in order to obtain a double chain complex, we will need to define a horizontal differential. We will also reduce subsets of $V^{\bullet, \ast}$, in order to allow us to ignore the degenerate faces. 
\begin{definition}\label{definition of normalized complex}
Let $V^{\bullet, *}$ be a cosimplicial cochain complex with $d_j^*$ the coface maps and $s_j^*$ the codegeneracy maps. We define the \underline{normalized complex} of $V^{\bullet, \ast}$ as the double complex $C^{\ast, \ast}(V)$, described in the following way:
$$C^{p,q}(V)= 
\begin{dcases}
V^{0,q} \qquad \qquad\qquad\qquad\;\;\,\qquad \quad\qquad \text{if}\;\; p=0\\
\bigcap_{i=0}^{p-1}\text{Ker}\left({s}^*_i: V^{p, q}\longrightarrow V^{p-1, q}\right)\; \; \qquad\text{if} \;\; p>0
\end{dcases}$$
where the vertical differential is $d: C^{p, q}(V)\longrightarrow C^{p, q+1}(V)$, the differential of the cochain complex $V^{p, *}$ (not to be confused with the face map $d_i$). The horizontal differential is given by $\delta: C^{p, q}(V)\longrightarrow C^{p+1, q}(V)$, defined as:
$$\delta =\sum^{p+1}_{j=0}(-1)^{j}({d_j}^*)_{|C^{p,q}(V)}$$
An element of $V^{\bullet, \ast}$ is said to be \underline{normalized} if it belongs to $C^{\ast, \ast}(V)$.
\end{definition}
\begin{lemma}
This definition is correct, i.e. that it fulfils the properties of a double complex.
\end{lemma}
\begin{proof}
It is clear that $(C^{n , \ast}(V), d)$ is a cochain complex for all $n$, since $d$ is simply the differential of $V^{n, *}$. \\
To check that $(C^{\ast, n}(V), \delta)$ is a chain complex, it is enough to prove that $\delta$ is well defined and that $\delta\circ\delta=0$.\\
Let $x\in \bigcap_{i=0}^{p-1}\text{Ker}\left({s}^*_i: V^{p, n}\longrightarrow V^{p-1, n}\right)$. Then, $\delta(x)=\sum^{p+1}_{j=0}(-1)^{j}{d_j}^*(x) $. Recall that the coface and codegeneracy maps are cochain maps, and thus are linear. Then for any $i$:
\begin{align*}
    {s}^*_i\circ \delta(x)&= \sum_{j=0}^{i-1}\left( (-1)^j s_i^{\ast}\circ d_j^{\ast}(x)\right) + (-1)^is_i^{\ast}\circ d_i^{\ast}(x) -(-1)^is_i^{\ast}\circ d_{i+1}^{\ast}(x) + \sum_{j=i+2}^{p+1} (-1)^j s_i^{\ast}\circ d_j^{\ast}(x)\\
    &\underset{(1)}{=} \sum_{j=0}^{i-1} (-1)^j d_j^{\ast}\circ s_{i-1}^{\ast}(x) + (-1)^i(x) -(-1)^i(x) + \sum_{j=i+2}^{p+1} (-1)^j d_{j-1}^{\ast}\circ s_i^{\ast}(x)\\
   &\underset{(2)}{=} 0
\end{align*}
where the equation $(1)$ holds because of the cosimplicial equations (see Corollary \ref{cosimplicial equation}) and the equation $(2)$ holds because $x\in \bigcap_{i=0}^{p-1}\text{Ker}\left({s}^*_i\right)$. So we proved that $\delta(x)\in \bigcap_{i=0}^{p}\text{Ker}\left({s}^*_i\right) $, and thus the well-definiteness of $\delta$ follows. \\
The next thing to show is that $\delta^2=0$. We have:
\begin{align*} 
    \delta\circ \delta (x)&= \sum_{i=0}^{p+2}\sum_{j=0}^{p+1}(-1)^{i+j} d_{i}^*\circ d_{j}^*(x)\\
    &= \sum_{i=0}^{p+2}\sum_{j=0}^{i-1}(-1)^{i+j} d_{i}^*\circ d_{j}^* + \sum_{i=0}^{p+2}\sum_{j=i}^{p+1}(-1)^{i+j} d_{i}^*\circ d_{j}^*(x) \\
     &\underset{(3)}{=} \sum_{i=1}^{p+2}\sum_{j=0}^{i-1}(-1)^{i+j} d_{j}^*\circ d_{i-1}^*(x) + \sum_{i=0}^{p+1}\sum_{j=i}^{p+1}(-1)^{i+j} d_{i}^*\circ d_{j}^*(x)\\
     &\underset{(4)}{=} \sum_{j=0}^{p+1}\sum_{i=0}^{j}(-1)^{i+j-1} d_{i}^*\circ d_{j}^*(x) +\sum_{i=0}^{p+1}\sum_{j=i}^{p+1}(-1)^{i+j} d_{i}^*\circ d_{j}^*(x)\\
    &\underset{(5)}{=} \sum_{i=0}^{p+1}\sum_{j=i}^{p+1}(-1)^{i+j-1} d_{i}^*\circ d_{j}^*(x) + \sum_{i=0}^{p+1}\sum_{j=i}^{p+1}(-1)^{i+j} d_{i}^*\circ d_{j}^*(x)\\
    &=0
\end{align*}
where $(3)$ is true because of the cosimplicial equations (Corollary \ref{cosimplicial equation}), $(4)$ is obtained by changing the index ($i-1 \mapsto j,\; j\mapsto i$), and $(5)$ is obtained because the first sum of each line is equal to the sum $\sum_{0\leq i\leq j\leq p+1}(\dots)$.
Finally, we need to check that $\delta\circ d= d\circ \delta$. It follows easily from the fact that the coface is a cochain map, and thus commutes with the differential:
$$\delta\circ d= \sum^{p+1}_{j=0}(-1)^{j}d_j^{\ast}\circ d= \sum^{p+1}_{j=0}(-1)^{j}d\circ d_j^{\ast}= d\circ \delta $$
\end{proof}
\begin{remark}
Intuitively, we can apply this construction on the cosimplicial cochain complex implied by the De Rham complex on a simplicial manifold. In this case, $d$ is the De Rham exterior derivative, $d_i^*$ and $s_i^*$ are the pullbacks of the face and degeneracy maps. For simplification, we denote the normalized double complex obtained in this way by $C^{*,*}(\m_{\bullet})$ for $\m_{\bullet}$ a simplicial manifold. 
\end{remark}
To understand better the definition \ref{definition of normalized complex}, let us introduce this construction:
\begin{definition}\label{cech groupoid}
Let $\m$ be a manifold and $\{U_i\}$ an open covering of $\m$. We obtain the following category:
$$Obj(\check{C})= \bigsqcup_i U_i\qquad Morph(\check{C})= \bigsqcup_{i, j}U_i\cap U_j$$
where the source and the target of $(x_i, x_j)$ are respectively $x_i$ and $x_j$, and $(x_j, x_k)\circ (x_i, x_j)=(x_i, x_k)$. This is easily checked to be a groupoid, called the \v{C}ech groupoid.\\ Intuitively, the morphisms go from the representation of $x$ on $U_i$ to the one on $U_j$ as illustrated below:
\end{definition}
\begin{center}
    \includegraphics[width=0.3 \textwidth]{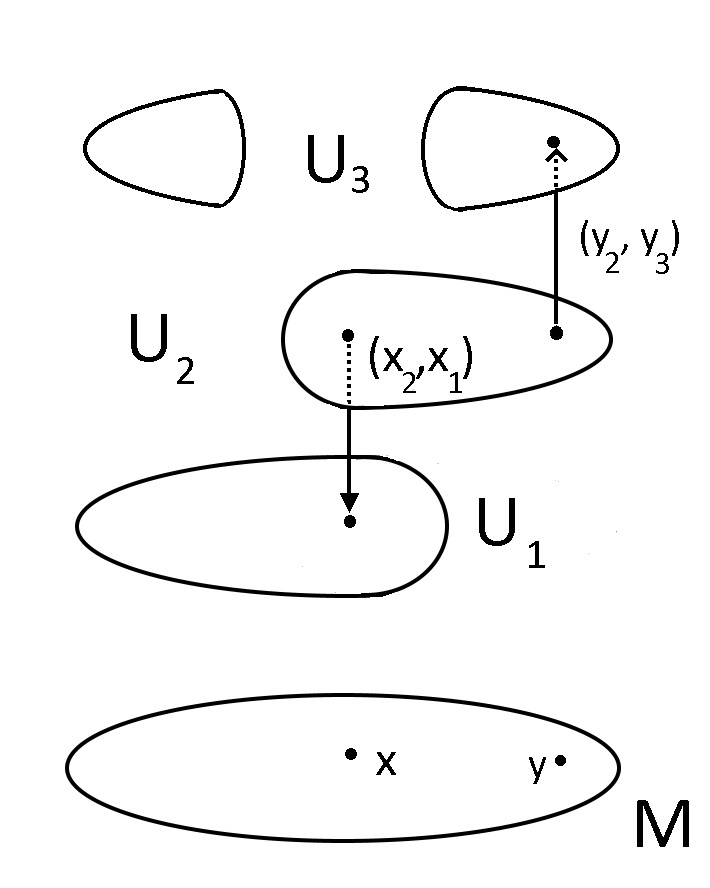}
\end{center}
\begin{construction}\label{cech nerve}
We can now construct the nerve of the \v{C}ech groupoid, which is given by:
$$\n\check{C}_{n}= \bigsqcup_{i_0, \dots, i_n} U_{i_0}\cap \dots \cap U_{i_n}$$
where, for $x\in U_{i_0}\cap \dots \cap U_{i_n}$, the face maps are given by $d_j(x)=x \in U_{i_0}\cap \dots\cap\hat{U}_{i_j}\cap\dots  \cap U_{i_n}$ and $s_j(x)=x\in U_{i_0}\cap \dots U_{i_j}\cap U_{i_j}\dots \cap U_{i_n} $. This nerve is called the \v{C}ech nerve and it can be checked to be a simplicial manifold. 
\end{construction}
Then we come to the following observation:
\begin{remark}\label{extension of cech complex}
The construction of the normalized complex can be understood as a generalization of the \v{C}ech complex. Indeed, let us consider the normalized complex of the \v{C}ech nerve. The condition $\bigcap_{i=0}^{p-1}\text{Ker}\left({s}_i\right)$ can be understood as ignoring the differential forms on $U_{i_1}\cap \dots \cap U_{i_n}$ if there is $i_j=i_k$ for $j\neq k$. We obtain a double chain complex given by: $$V^{p, q}= \Omega^{q}\left(\underset{\text{all different}}{\bigsqcup_{i_0, \dots, i_p}} U_{i_0}\cap \dots \cap U_{i_p} \right)$$
with $$\left(\delta(\phi)\right)_{i_0, \dots, i_p+1}=\sum^{p+1}_{j=0} (-1)^j d_j^{\ast} \phi_{i_0, \dots, \hat{i_j}, \dots, i_p+1} $$
where $\psi_{i_0, \dots, i_p}$ is the restriction of $\psi$ on $ U_{i_0}\cap \dots \cap U_{i_p}$ and $d_j$ is the face map, equal here to the inclusion. So we get the \v{C}ech double chain complex, and the normalized complex can thus be understood as a generalization of it. 
\end{remark}

Following this remark, we can, as for the \v{C}ech double chain complex, calculate the total complex of the normalized chain complex.
\begin{definition}
Given a double complex $C^{\ast, \ast}$, we can construct the \underline{\textit{total complex}} \tot{} defined as:
$$\tot{}^n=\bigoplus_{k+l=n}C^{k, l}$$
The differential, denoted $D$, is defined as $D= \partial_{hor}+(-1)^{\text{horizontal degree}}\partial_{vert}$. This means concretely that for $(\alpha_0, \alpha_1, \dots, \alpha_n)\in C^{0,n}\oplus C^{1, n-1}\oplus \dots \oplus C^{n, 0} =\tot{}^n$, we have:
$$D((\alpha_0,
\dots, \alpha_n))= (\partial_{vert}(\alpha_0), \partial_{hor}(\alpha_0)+(-1)\cdot\partial_{vert}(\alpha_1), \dots, \partial_{hor}(\alpha_{n-1})+(-1)^{n}\partial_{vert}(\alpha_{n}), \partial_{hor}(\alpha_n))$$
\end{definition}
As before, we can apply it to the simplicial manifolds by calculating the total complex of $C^{*,*}(\m_{\bullet})$. We denote it $\tot{}^{n}(\m_{\bullet})$. 
\begin{definition}\label{lalaoui}
The elements of $\tot{}^{n}(\m_{\bullet})$ are called the \underline{differential $n$ forms of $\m_{\bullet}$}. 
\end{definition}

In our specific case, we can then construct the total complex of our normalized double complex constructed from a simplicial manifold $\m_{\bullet}$. In this case, the elements will have the form:
$$(\alpha_0, \dots, \alpha_{p}, \dots, \alpha_n)\in \tot{}^n(\m_{\bullet})= \Omega^n(\m_0)\oplus \dots \oplus \Omega^{n-p}(\m_p)\oplus\dots\oplus \Omega^0(\m_n)$$
such that $s_i^*(\alpha_j)=0\;\; \forall \; i, j$.

This extends the differential forms of simplicial manifolds. Indeed:

\begin{remark}\label{simplicial form on usual manifolds}
In the case of the constant simplicial manifold $\underline{\m}_{\bullet}$, $\tot{}^{\ast}(\underline{\m}_{\bullet})= \Omega^{\ast}(\m)$. This follows easily from the fact that for all $n\geq1$, $\underline{\m}_n=\m$ and for any $\phi$ differential form on $\m$, $\phi= \text{id}^{\ast}(\phi)$, with $s_i:=\text{id}$. Thus, $C^{p, q}(\underline{\m}_{\bullet})=0$ as soon as $p\neq 0$, and $\tot{}^q(\underline{\m}_{\bullet})= C^{0, q}(\underline{\m}_{\bullet})= \Omega^q(\m)$. This point allows us to see in which sense the cohomology $\tot{}$ extends the usual de Rham complex. 
\end{remark}
\begin{definition}\label{pullback on total spac}
    For a morphism of simplicial manifolds, $f_{\bullet}:=\{f_n\}: \m_{\bullet}\longrightarrow \m_{\bullet}'$, and for $\phi\in \tot{}^k(\m_{\bullet}')$, we define the pullback as:
    $$f_{\bullet}^*\phi= (f_0^*\phi_0, \dots, f_{i}^*\phi_i, \dots, f_k^*\phi_k)\in \tot{}^k(\m_{\bullet})= \Omega^k(\m_0)\oplus \dots \oplus \Omega^{k-i}(\m_i)\oplus\dots\oplus \Omega^0(\m_k)$$
\end{definition}
Let us finally introduce a last structure:
\begin{definition}
Let $C$ be a category and $\X$ be an object of $C$. We call \underline{\textit{filtration}} of $\X$ any sequence of monomorphisms :
$$\dots \longrightarrow \X^{n+1}\longrightarrow \X^{n}\longrightarrow \dots \longrightarrow \X^1\longrightarrow \X$$
The object $\X$ is called the \underline{filtered object}. 
\end{definition}
\begin{definition}
Let $\m$ be a manifold and $\Omega^{\ast}(\m)$ be its De Rham cochain complex. Then, there exists a canonical filtration of $\Omega^{\ast}(\m)$ in the category of cochain complex sometimes called the ``\underline{filtration bête}" (french for ``silly filtration"), defined as:
$$\dots \longrightarrow F^i\Omega^*(\m) \longrightarrow F^{i-1}\Omega^*(\m)\longrightarrow\dots\longrightarrow F^1\Omega^*(\m) \longrightarrow \Omega^*(\m)$$
$$F^i\Omega^k(\m)=
\begin{dcases}
0, \qquad \qquad & k<i\\
\Omega^k(\m) , & k\geq i
\end{dcases}$$
and the monomorphisms of the filtration are simply given by the canonical inclusions, either\\ $\text{id}: \Omega^k(\m)\longrightarrow \Omega^k(\m)$ or $0\hookrightarrow \Omega^k(\m)$. 
\end{definition}
Clearly, $F^i\Omega^{\ast}(\m)$ is again a cochain complex, and the monomorphisms of the filtration are cochain morphisms. 
We can apply this ``filtration bête" on the total complex $\tot{}^{\ast}$ constructed over a simplicial manifold $\m_{\bullet}$. We obtain then:
$$F^i\tot{}^n(\m_{\bullet})=\bigoplus_{k+l=n}F^iC^{l,k}(\m_{\bullet}) $$
and where an element of $F^i\tot{}^n(\m_{\bullet})$ has the form:
$$(\alpha_0, \dots, \alpha_{p}, \dots, \alpha_{n-i})\in \tot{}^n(\m_{\bullet})= \Omega^n(\m_0)\oplus \dots \oplus \Omega^{n-p}(\m_p)\oplus\dots\oplus \Omega^i(\m_{n-i})$$
such that $s_i^*(\alpha_j)=0\;\; \forall \; i, j$.

\subsection{Relation with the realization}\label{other part}
The construction made in the preceding part may be unintuitive, despite its relation to the \v{C}ech complex. We want to see now a second way for defining the differential forms on simplicial manifolds, and although the two definitions are not equivalent, we want to show that they are strongly related.  \\
Before we start, we need some calculating considerations, so let us start with the following definition/recall:
\begin{definition}[Integration along the fiber]\label{Integration along the fibers}
Let $\pi:\p\longrightarrow\m$ be a fibre bundle with fibre $F$ of dimension $l$ being compact and orientable. Let $\phi$ be a differential $k$-form on $\p$, with $k\geq l$. We define:
\begin{itemize}
    \item For given $u_1, \dots, u_{k-l}\in T_p\m$, $p\in\m$, the differential $l$-form $\Bar{\phi}_{u_1, \dots, u_{k-l}}$ on $F$, given by:
    $$\Bar{\phi}_{u_1, \dots, u_{k-l}}(v_1, \dots, v_l)= \phi(v_1, \dots, v_l, \Tilde{u}_1, \dots, \Tilde{u}_{k-l})$$
    with $\Tilde{u}_i\in (\pi_*)^{-1}(u_i)$ any lift of $u_i$. If ${u_1, \dots, u_{k-l}}$ are clear with the context, or if they are let as variables, we simply write $\bar{\phi}$. 
    \item We define the \underline{integration along the fibers} for $\phi$ as the $k-l$ differential form on $\m$ defined as:
    $$\left(\int_F\phi\right)_p(u_1, \dots, u_{k-l})= \int_{\pi^{-1}(p)}\Bar{\phi}_{u_1, \dots, u_{k-l}}, \qquad \text{with $\pi^{-1}(p)\cong F$}$$
\end{itemize}
\end{definition}
\begin{remark}
    In the case where the integral is not well-defined, that is that $l>k$, we set the result of the integral to be $0$.
\end{remark}
The following lemma is necessary to understand why the definition \ref{Integration along the fibers} is well-defined:
\begin{lemma}\label{linear dependence}
The definition of $\bar{\phi}$ does not depend on the choice of the lifts $\Tilde{u}_i\in \pi^{-1}(u_i)$. 
\end{lemma}
\begin{proof}
Let $l$ be the dimension of the fibre. Indeed:
\begin{itemize}
    \item If $v_1, \dots, v_l$ are not linearly independent, then w.l.o.g. $\exists\lambda_i$ such that $v_1= \sum_{i=2}^l \lambda_i\cdot v_i$. Thus:
    $$\bar{\phi}(v_1, \dots, v_l)= \sum_{i=2}^l\lambda_i\cdot \underbrace{\bar{\phi}(v_i, v_2, \dots, v_l)}_{=0}=0$$
    \item If $v_1, \dots, v_l$ are linearly independent, any vertical vector can be written as a linear combination of them. Thus for $\Tilde{u}_i^1$ and $\Tilde{u}_i^2$ two lifts of $u_i$, $\exists \lambda_j$ such that $\Tilde{u}_i^2- \Tilde{u}_i^1= \sum_{j=1}^l \lambda_j\cdot v_j$. Thus:
    \begin{align*}
        \phi(v_1, \dots, v_l,& \Tilde{u}_1, \dots, \Tilde{u}_i^1,\dots,  \Tilde{u}_{k-l})\\= &\phi(v_1, \dots, v_l, \Tilde{u}_1, \dots, \Tilde{u}_i^1,\dots,  \Tilde{u}_{k-l})+\sum_{j=1}^l \lambda_j \cdot \underbrace{\phi(v_1, \dots, v_l, \Tilde{u}_1, \dots, -v_j,\dots,  \Tilde{u}_{k-l})}_{=0\; \text{since $v_j$ comes twice}}\\ =& \phi(v_1, \dots, v_l, \Tilde{u}_1, \dots, \Tilde{u}_i^2,\dots,  \Tilde{u}_{k-l})
    \end{align*}
\end{itemize}
\end{proof}
\begin{construction}
Let us consider $\m_{\bullet}$ a simplicial manifold. Then $\foo{}^n\times \m_n$ is a manifold (with boundary) for each $n$. It means thus that each element of the pseudo-realization of $\m_{\bullet}$ has a canonical manifold structure. Remark but that it does not necessarily descend to a manifold structure on $|\m_{\bullet}|$, the geometric realization of $\m_{\bullet}$.\\
But since $|\m_{\bullet}|$ results from the glueing of differential manifolds, we can intuitively extend the definition of differential forms, in the sense that we define a differential form on each $\foo{}^n\times \m_n$, and we expect it to agree accordingly to the glueing used for $|\m_{\bullet}|$. 
\end{construction}
This intuitive construction can be set more formally under the following definition:
\begin{definition}\label{pseudo-real form}
Let $\m_{\bullet}$ be a simplicial manifold and $\hat{\m}_{\bullet}$ be its pseudo-realization. Then, we define a \underline{differential $p$-form on $|\m_{\bullet}|$} as being a collection of $p$-forms $\phi=\{\phi_i\}_{i\in \N_0}$, with $\phi_i$ differential $p$-form on $\foo{}^i\times \X_i$, and such that the following equations are fulfilled:
\begin{align*}
    (\tau_j\times \text{id})^{\ast}\phi_i&= (\text{id}\times d_j)^{\ast} \phi_{i-1}\qquad \text{on $\foo{}^{i-1}\times \X_i$}\\
    (\eta_j\times \text{id})^{\ast}\phi_{i-1}&= (\text{id}\times s_j)^{\ast} \phi_{i}\qquad\;\;\;\; \text{on $\foo{}^{i}\times \X_{i-1}$}
\end{align*}
for all the $i$ and $j$ for which it is well-defined. We denote $\Omega^k(|\m_{\bullet}|)$ the set of the differential forms on $|\m_{\bullet}|$ of degree $k$.
\end{definition}

Since both the exterior product and the exterior derivative commute with the pullback, we get the following well-defined extended concepts:
\begin{definition}\label{ext prod and deri for pseudo simplicial}
Let $\m_{\bullet}$ be a simplicial manifold and let $\phi=\{\phi_i\}$ and $\psi=\{\psi_i\}$ be differential forms on $|\m_{\bullet}|$. We can then define:
\begin{itemize}
    \item $\phi\wedge\psi:= \{\phi_i\wedge\psi_i\}_{i\in \N_0}$
    \item $d\phi:=\{d\phi_i\}_{i\in \N_0}$
\end{itemize}
The well-definiteness follows from:
$$(\tau_j\times \text{id})^*(\phi_i\wedge \psi_i)= (\tau_j\times \text{id})^* \phi_i\wedge (\tau_j\times \text{id})^*\psi_i= ( \text{id}\times d_j)^*\phi_{i-1}\wedge ( \text{id}\times d_j)^*\psi_{i-1}= ( \text{id}\times d_j)^*(\phi_{i-1}\wedge \psi_{i-1})$$
It is then analogous for $(\eta_j\times \text{id})$. With the same justification, we see that the exterior derivative is also well-defined.
The usual properties of the exterior product and the exterior derivative follow immediately. 
\end{definition}
\begin{lemma}
Let $\m$ be a manifold, and let $\underline{\m}_{\bullet}$ be the corresponding constant simplicial manifold. Then, there is an isomorphism of cochain complexes, giving us:
$$\Omega^k(|\underline{\m}_{\bullet}|)\cong \Omega^k(\m)$$
\end{lemma}
\begin{proof}
To prove it, we have to prove that all the elements of $\Omega^k(|\underline{\m}_{\bullet}|)$ have the form $(\pi^n_{\m})^*\phi$, for $\pi_{\m}^n: \foo{}^n\times \m\longrightarrow \m$ the projection and $\phi\in \Omega^k(\m)$.\\
Take $\{\phi_i\}\in\Omega^k(|\underline{\m}_{\bullet}|) $. Then, since $\underline{\m}_0=\m$, $\phi_0\in \Omega^k(\m)$. Suppose now that for $\underline{\m}_n$, $\phi_n=(\pi^n_{\m})^*\phi_0$, i.e. $\phi_n((v_1^1, v_2^1), \dots, (v_1^k, v_2^k)) = \phi_0(v_2^1, \dots, v_2^k)$ for $(v_1^i, v_2^i)\in T_t\foo{}^n\times T_p\m\cong T_{(t,p)}\foo{}^n\times \m$. Then, we have on $\foo{}^{n+1}\times \m$ the form $(\eta_i\times \text{id})^*(\pi^n_{\m})^*\phi_0$ This form is equal to $(\pi^{n+1}_{\m})^*\phi_0$. Furthermore, by the definition of a differential form on $|\underline{\m}_{\bullet}|$, $(\pi^{n+1}_{\m})^*\phi_0= (\eta_i\times \text{id})^*\phi_n = (\text{id}\times s_i)^*\phi_{n+1}= \phi_{n+1}$ since $s_i=\text{id}$. So we proved inductively that for all $n$, $\phi_n=(\pi^{n}_{\m})^*\phi_0$n (the induction's start is $\phi_0= \text{id}^* \phi_0$). Thus, the differential form on $\Omega^k(|\underline{\m}_{\bullet}|)$ is uniquely defined by the differential form $\phi \in \Omega^k(\m)$. We get then two isomorphisms, inverse to each other:
$$\phi\in \Omega^k(\m)\mapsto \{(\pi^i_{\m})^*\phi\}\in\Omega^k(|\underline{\m}_{\bullet}|)  \qquad \{\phi_i\}\in \Omega^k(|\underline{\m}_{\bullet}|)\mapsto \phi_0\in \Omega^k(\m)$$
\end{proof}

\begin{remark}\label{form for constant simplicial}
We can see that the definition \ref{pseudo-real form} also extends the De Rham complex to the world of simplicial manifolds. It follows then that we have two extensions of the definition of differential forms to the world of simplicial manifolds: the one given in Definition \ref{lalaoui} and the one in Definition \ref{pseudo-real form}.
\end{remark}
\begin{remark}
Since the exterior derivative of forms defined in \ref{pseudo-real form} is given by the exterior derivative of each component, we can conclude that the cohomology classes of differential p-forms on $|\X_{\bullet}|$ are contained in the following set:
$$\left\{\rho_{\bullet}=\{\rho_i\}\in \prod_{i=0}^{\infty}H^p(\foo{}^i\times \m_i)\middle| (\tau_j\times \text{id})^{\ast}\rho_i= (\text{id}\times d_j)^{\ast} \rho_{i-1}, (\eta_j\times \text{id})^{\ast}\rho_{i-1}= (\text{id}\times s_j)^{\ast} \rho_{i}\; \forall i,j\right\}$$
In this notation, the cohomology class of $\phi_{\bullet}= \{\phi_i\}$ is $\rho_{\bullet}=\{\rho_i\}$, with $\rho_i$ the cohomology class of $\phi_i$ in $\m_i$. 
\end{remark}
\begin{definition}\label{pullback on geom reali}
Let $\m_{\bullet}$ and $\m_{\bullet}'$ be two simplicial manifolds and $f:\m_{\bullet}\longrightarrow \m_{\bullet}'$ be a simplicial morphism between them which decomposes in $f_n: \m_n\longrightarrow\m_n'$ at each level. Let furthermore $\phi_{\bullet}=\{\phi_i\}$ be a differential form on $|\m_{\bullet}'|$. Then, we define:
$$f^*\phi_{\bullet}:= \{(\text{id}_{\fooo{}^{i}}\times f_i)^*\phi_i\}$$
\end{definition}
\begin{lemma}
This definition is well-defined.
\end{lemma}
\begin{proof}
This comes from the fact that, since $f$ is a morphism of simplicial manifolds, $d_j\circ f_n= f_{n-1}\circ d_j$ and $s_j\circ f_n= f_{n+1}\circ s_i$. Thus:
\begin{align*}
    (\text{id}\times d_j) \circ (\text{id}\times f_n)=   (\text{id}\times f_{n-1})\circ (\text{id}\times d_j),\quad& (\text{id}\times s_j) \circ (\text{id}\times f_n)=   (\text{id}\times f_{n+1})\circ (\text{id}\times s_j)\\
    (\tau_j\times \text{id})\circ (\text{id}\times f_n)= (\text{id}\times f_n) \circ (\tau_j\times \text{id}),\quad&  (\eta_j\times \text{id})\circ (\text{id}\times f_n)= (\text{id}\times f_n) \circ (\eta_j\times \text{id})
\end{align*}
With the pullbacks, we obtain:
\begin{align*}
    (\tau_j\times \text{id})^*(\text{id}\times f_i)^*\phi_i = (\text{id}\times f_i)^*(\tau_j\times \text{id})^*\phi_i &= (\text{id}\times f_i)^*(\text{id}\times d_j)^*\phi_{i-1} \\&= (\text{id}\times d_j)^*(\text{id}\times f_{i-1})^*\phi_{i-1}\\
    (\eta_j\times \text{id})^*(\text{id}\times f_{i-1})^*\phi_{i-1} = (\text{id}\times f_{i-1})^*(\eta_j\times \text{id})^*\phi_{i-1} &= (\text{id}\times f_{i-1})^*(\text{id}\times s_j)^*\phi_{i} \\&= (\text{id}\times s_j)^*(\text{id}\times f_{i})^*\phi_{i}
\end{align*}
So we conclude that the definition is well-defined.
\end{proof}
\begin{construction}
With definition \ref{pullback on geom reali}, we can intuitively try to bring the differential form $\phi=\{\phi_i\}$ to a form on each $\m_n$ by integrating along $\foo{}^n$. 
\end{construction}
Concretely, we get the following lemma:
\begin{lemma}\label{well defined integral}
Let $\m_{\bullet}$ be a simplicial manifold and $\phi=\{\phi_i\}$ be a differential $p$-form on $|\m_{\bullet}|$. Then:
$$\left\{\int_{\fooo{}^i}\phi_i\right\}\in \tot{}^p(\m_{\bullet})$$
where the integration along $\foo{}^n$ is made using the orientation of the $n$-simplex given by $(t_1, \dots, t_n)$.  Let us denote this map by $S$. 
\end{lemma}
\begin{proof}
Because of the definition of the integration along the fibre, it is clear that: $$\left\{\int_{\fooo{}^i}\phi_i\right\}\in \Omega^p(\m_0)\otimes \Omega^{p-1}(\m_1)\otimes \dots \otimes \Omega^0(\m_p)$$
Furthermore, we know that $(\eta_j\times \text{id})^{\ast}\phi_{i-1}= (\text{id}\times s_j)^{\ast} \phi_{i}$. For $\phi$ a $p=r+t$ form on $|\m_{\bullet}|$ it implies that:
\begin{align*}
     s_i^{\ast}\left(\int_{\fooo{}^{r}}\phi_r\right) (u_1, \dots, u_t)&= \left(\int_{\fooo{}^{r}}\phi_r\right) ((s_i)_{\ast}u_1, \dots, (s_i)_{\ast}u_t)\\
     &= \int_{\fooo{}^r} \left(\bar{\phi}_r\right)_{(s_i)_{\ast}u_1, \dots, (s_i)_{\ast}u_t}
\end{align*}
where $\bar{\phi}$ has the same meaning as in Definition \ref{Integration along the fibers}. Remark then the following: for given vertical tangent vectors $v_1, \dots, v_r$ at a point on $\foo{}^r\times \m_{t}$, we have:
$$\left(\bar{\phi}_r\right)_{(s_i)_{\ast}u_1, \dots, (s_i)_{\ast}u_t}(v_1, \dots, v_r)= \phi_r (v_1, \dots, v_r, \widetilde{(s_i)_{\ast}u_1}, \dots, \widetilde{(s_i)_{\ast}u_t})$$
Since the bundle is trivial $\foo{}^r\times \m_{t}$, a lift $\widetilde{(s_i)_{\ast}u_j}$ of $(s_i)_{\ast}u_j$ can simply be given by: $$0\times (s_i)_{\ast}u_j= (\text{id}_{\fooo{}^{r}}\times s_i)_{\ast} (0\times u_j) $$
Since $0\times u_j $ is a lift of $u_j$, let us denote it by $\tilde{u_j}$. Furthermore, since $v_i$ is vertical, we have $(\text{id}_{\fooo{}^{r}}\times s_i)_{\ast} v_t= v_t$ and thus: 
\begin{align*}
    \left(\bar{\phi}_r\right)_{(s_i)_{\ast}u_1, \dots, (s_i)_{\ast}u_t}(v_1, \dots, v_r)&= (\text{id}_{\fooo{}^{r}}\times s_i)_{\ast} \phi_r(v_1, \dots, v_r, \tilde{u_1}, \dots, \tilde{u_t})\\
    &= (\eta_i\times \text{id}_{\m_{r-1}})^{\ast}\phi_{r-1} (v_1, \dots, v_r, \tilde{u_1}, \dots, \tilde{u_t})
\end{align*}
Since we set $\tilde{u_j}= 0\times u_j$, then $(\eta_i\times \text{id}_{\m_{r-1}})_{\ast} (0\times u_j) = 0\times u_j $. On the other hand, $(\eta_i\times \text{id}_{\m_{r-1}})_{\ast} v_i$ will be a vertical vector of $\foo{}^{r-1}\times \m_{t}$. But the fiber of $\foo{}^{r-1}\times \m_{t}$ is $\foo{}^{r-1}$, which has dimension $r-1$, and we have $r$ vertical tangent vectors as entries:
$$(\eta_i\times \text{id}_{\m_{r-1}})_{\ast} v_1, \dots, (\eta_i\times \text{id}_{\m_{r-1}})_{\ast} v_r$$
Thus, the entries of $\phi_{r-1}$ are linear dependent, and thus:
$$(\eta_i\times \text{id}_{\m_{r-1}})^{\ast}\phi_{r-1} (v_1, \dots, v_r, \tilde{u_1}, \dots, \tilde{u_t})=0$$
It implies then:
$$s_i^{\ast}\left(\int_{\fooo{}^{r}}\phi_r\right) (u_1, \dots, u_t)= \int_{\fooo{}^{r}}\left(\bar{\phi}_r\right)_{(s_i)_{\ast}u_1, \dots, (s_i)_{\ast}u_t} =0$$
Thus: $$\int_{\fooo{}^i}\phi_i\in \bigcap_i \text{Ker}(s_i)$$
We can thus conclude the proof of the lemma. 
\end{proof}
The map that we defined in lemma \ref{well defined integral} is more than a map between sets: the following lemma shows that it is even a map between chain complexes:
\begin{lemma}\label{chain complex map}
Let $\m_{\bullet}$ be a simplicial manifold, $\phi=\{\phi_i\}$ be a differential $p$-form on $|\m_{\bullet}|$ and $\left\{\int_{\fooo{}^i}\phi_i\right\}\in \tot{}^p(\m_{\bullet})$. Then:
$$D\left(\left\{\int_{\fooo{}^i}\phi_i\right\}\right)= \left\{\int_{\fooo{}^i}d\phi_i\right\}$$
\end{lemma}
To prove this lemma, we need to prove first another lemma which will allow us to apply Stokes' theorem on $\foo{}^n$:
\begin{lemma}\label{integrating on foo}
Consider the simplex $\foo{}^n$, and $\partial\foo{}^n$ the boundary of the simplex. Furthermore, let $\phi$ be a differential $(n-1)$-form on $\foo{}^n$. Then:
$$\int_{\partial \fooo{}^n}\phi = \sum_{i=0}^n(-1)^{i} \int_{\fooo{}^{n-1}}\tau_i^{\ast}\phi$$
with the orientation of $\foo{}^n$ being given by $(t_1, \dots, t_n)$. 
\end{lemma}
The idea of this lemma is that the inclusion given by $\tau_i$ implies canonically an orientation on each face, which is not necessarily the same as the orientation implied by $(t_1, \dots, t_n)$. For example, on $\foo{}^3$ it will give the following face orientations:
\begin{center}
    \includegraphics[width=0.3 \textwidth]{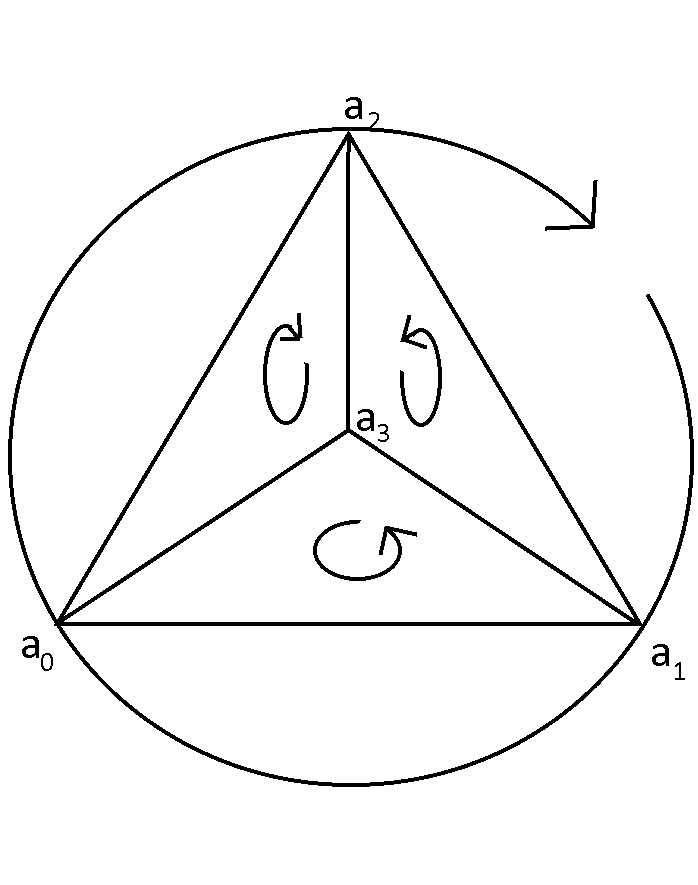}
\end{center}
where the outer circle shows the orientation of the outer face $(a_0, a_1, a_2)$. Let us now come to a precise proof following this intuition:
\begin{proof}
By the properties of the integral, it is clear that:
$$\int_{\partial \fooo{}^n}\phi = \sum_{i=0}^n \int_{\tau_i(\fooo{}^{n-1})}\phi$$
It follows:
$$\int_{\tau_i(\fooo{}^{n-1})}\phi = \pm \int_{\fooo{}^{n-1}} \tau_i^{\ast}\phi $$
since $\tau_i: \foo{}^{n-1}\longrightarrow \tau_i(\foo{}^{n-1})$ is an isomorphism, and the sign depends on the orientation. Indeed, the orientation of $\tau_i(\foo{}^{n-1})\cong \foo{}^{n-1}$ is implied by the orientation on $\foo{}^n$. 
So let us have specific attention to the orientation. In the following, we often denote $\tau_i(\foo{}^{n-1})$ by $\partial^i\foo{}^n$. \\
Let us recall that on the boundary of a manifold $\m$ of dimension $n$, the orientation is given in the following way: let $\gamma: U\subseteq \m \longrightarrow \R^{\geq 0} \times \R^{n-1} $ be an oriented boundary chart, that is an oriented chart such that:
$$\gamma(p)\in  \{0\}\times\R^{n-1}\Leftrightarrow p\in \partial\m \qquad \forall p\in U $$
Then, the orientation of $\partial U\subseteq \partial \m$ is given by the opposite orientation of $\{0\}\times\R^{n-1}\cong \R^{n-1}$. \\
On $\foo{}^n$, the i-th face $\partial^i\foo{}^{n-1}$ is the one delimited by the points $e_0, \dots, \hat{e}_i, \dots, e_n$ with $\{e_j\}$ being the usual basis of $\R^{n+1}$ and where the circumflex means that the element has been omitted. Suppose first that $i\geq 1$. Then, we consider the chart $\mu_i:  \foo{}^n \longrightarrow \R^{\geq 0}\times \R^{n-1}   ,\; (t_1, \dots, t_{n})\mapsto (t_i, t_1, \dots, t_{i-1}, t_{i+1}, \dots, t_{n})$. Clearly, this is a boundary map containing the $i$-th face, since a point is in $\{0\}\times \R^{n-1}$ iff $t_i=0$ (which is equivalent to be in the $i$-th face). The Jacobian matrix of this transformation is given by:

$$
\begin{pmatrix}
0 &\cdots&0&1&0&\cdots&0 \\
1&&&\vdots&&&\\
&\ddots&&\vdots&&& \\
&&1&0&&&\\
&&&0&1&&\\
&&&&&\ddots&\\
&&&&&&1
\end{pmatrix}
$$
We get easily that the determinant of this matrix is $(-1)^{i-1}$. Thus, the chart given by $\hat{\mu}_i: \foo{}^n\longrightarrow \R^{\geq 0}\times\R^{n-1}, \; (t_1, \dots, t_n)\mapsto (t_i, (-1)^{i-1}t_1,\dots, t_{i-1}, t_{i+1}, \dots, t_{n})$ is also a boundary chart, but which is necessarily well-oriented. Since the orientation on $\partial^i\foo{}^n$ is the opposite orientation to the one given on $\{0\}\times \R^{n-1}$, we get that the orientation on the $i$-th face is given by $((-1)^{i}t_1,\dots, t_{i-1}, t_{i+1}, \dots, t_{n})$. The orientation on $\foo{}^{n-1}$ is given by $(t_1, \dots, t_{n-1})$, and thus the orientation on $\partial^i\foo{}^n$ induced by $\tau_i: \foo{}^{n-1}\longrightarrow\partial^i\foo{}^n$ is $( t_1, \dots, t_{i-1}, t_{i+1}, \dots, t_{n})$. Thus, the two orientations differ by a factor $(-1)^i$.\\
In the case where $i=0$, we can consider the coordinates $(t_0, \dots, t_{n-1})$. Then, as before, we consider the chart $\mu_0: \foo{}^n \longrightarrow \R^{\geq 0}\times \R^{n-1}, \; (t_1, \dots, t_n) \mapsto (t_0, t_1, \dots, t_{n-1})$, with $t_0=1-\sum_{i=1}^{n}t_i$. As before, it is a boundary map containing the $0$-th face. If we consider now the Jacobian matrix of this transformation we get:

\begin{equation}\label{eq: matrix}
  \begin{pmatrix}
 -1 & -1 &\dots &-1\\
   1 & 0&&\\
 &\ddots &\ddots&\\
 &&1&0
\end{pmatrix}  
\end{equation}

We check easily that the determinant of this matrix is $(-1)^n$, and thus we can consider the chart $\hat{\mu}_0: \foo{}^n \longrightarrow \R^{\geq 0}\times \R^{n-1}, \; (t_1, \dots, t_n) \mapsto (t_0, (-1)^n t_1, \dots, t_{n-1})$ which is an oriented boundary chart. Thus, the canonical orientation of $\partial^0\foo{}^n$ is the opposite of the one on $\{0\}\times \R^{n-1}$, that is $((-1)^{n-1} t_1, \dots, t_{n-1})$. The orientation $(t_1, \dots, t_{n-1})$ of $\foo{}^{n-1}$ will be sent through $\tau_0$ to the orientation $(t_2, \dots, t_n)$ on $\partial^0\foo{}^n$. The Jacobian matrix of the transformation from $(t_2, \dots, t_n)$ to $((-1)^{n-1} t_1, \dots, t_{n-1})$ is the following, knowing that on $\partial^0\foo{}^n$ we have $t_1=1-\sum_{i=2}^n t_i$:

$$\begin{pmatrix}
 (-1)^{n-1}&(-1)^{n}&\dots &(-1)^{n}&(-1)^{n}\\
 1&0&&&\\
 &\ddots&\ddots&&\\
 &&1&0&\\
 &&&1&0
\end{pmatrix}$$

We can check that the determinant of this matrix is $1$ (since it is the same determinant as the matrix (\ref{eq: matrix}), but the matrix is now of dimension $n-1\times n-1$ with a $(-1)^{n-1}$ factor: it gives the determinant $(-1)^{n-1}\cdot (-1)^{n-1}= (-1)^{2n-2}=1$). Thus on $\partial^0\foo{}^n$, the orientation implied by $\tau_0$ and the one implied by $\foo{}^n$ are the same. 

In a conclusion:
$$\int_{\tau_i(\fooo{}^{n-1})}\phi = (-1)^{i} \int_{\fooo{}^{n-1}} \tau_i^{\ast}\phi$$
And the statement follows from it. 
\end{proof}
Let us come back to the proof of Lemma \ref{chain complex map}:
\begin{proof} \textit{(of Lemma \ref{chain complex map})}
Let $\{(U_i, \phi_i)\}$ be an atlas of $\X_n$ with a corresponding partition of unity $\alpha_i$. It gives rise canonically to a partition of unity on $\foo{}^n\times\X_n $ with $\alpha_i\circ \pi$, where $\pi: \foo{}^n\times\X_n\longrightarrow \X_n$ is the projection. If we consider a certain differential $k$-form $\theta$ on $\foo{}^n\times\X_n $, the restriction of it to $\foo{}^n\times U_i$ can be written as:
$$\sum_{\xi\in \N} a_{\xi} dy_{\xi(1)}\wedge \dots \wedge dy_{\xi(p)}\wedge dx_{\xi(p+1)}\wedge\dots \wedge dx_{\xi(p+q)}\qquad \text{with}\; p+q=k$$
where $\{y_i\}_{i\in [ 1, n]}$ are coordinates in $\foo{}^n$ (with the same orientation as $(t_1, \dots, t_n)$) and $\{x_i\}_{i\in [ 1, m]}$ are coordinates of $\X_n$. Since the sum and the integral such as the sum and the exterior derivative commute, it is enough to check the equality for $a_{\xi} dy_{\xi(1)}\wedge \dots \wedge dy_{\xi(p)}\wedge dx_{\xi(p+1)}\wedge\dots \wedge dx_{\xi(p+q)}\; \text{with}\; p+q=k$. Let us denote it $\theta_{\xi}$.
Then we have:
$$d\theta_{\xi}= \sum_{j=1}^n \frac{\partial a_{\xi}}{\partial y_j} dy_j\wedge d\bar{y}\wedge d\bar{x} + \sum_{j=1}^m \frac{\partial a_{\xi}}{\partial x_j} dx_j\wedge d\bar{y}\wedge d\bar{x}$$
with $d\bar{y}= dy_{\xi(1)}\wedge \dots \wedge dy_{\xi(p)}$ and $d\bar{x}=dx_{\xi(p+1)}\wedge\dots \wedge dx_{\xi(p+q)} $ Thus we have:
\begin{equation}\label{eq: first of proof}
   \int_{\fooo{}^n}d\theta_{\xi}= \int_{\fooo{}^n}\sum_{j=1}^n \frac{\partial a_{\xi}}{\partial y_j} dy_j\wedge d\bar{y}\wedge d\bar{x} + \int_{\fooo{}^n}\sum_{j=1}^m \frac{\partial a_{\xi}}{\partial x_j} dx_j\wedge d\bar{y}\wedge d\bar{x} 
\end{equation}
Let us consider the first part of the right side of (\ref{eq: first of proof}). Then, if we want to know the value of the integral at $w_{1}, \dots, w_{k+1-n}$ (tangent vectors on $\m_n$), it means that we will need to consider and integrate:
$$\frac{\partial a_{\xi}}{\partial y_j} dy_j\wedge d\bar{y}\wedge d\bar{x}(v_1, \dots, v_n, \tilde{w}_1, \dots, \tilde{w}_{k+1-n})$$
where $\tilde{w}_i=(0, w_i)\in T\foo{}^n\times T\m_n$ and where the $v_i$ are vertical vectors (that is in $T\foo{}^n$). Since $dy_j(v)=0$ for any horizontal vector, it follows, using the definition of the wedge product (Definition \ref{wedge products}), that it is equal to:
$$\frac{1}{k!}\sum_{i=1}^n(-1)^{i+1}\frac{\partial a_{\xi}}{\partial y_j} dy_j(v_i)\cdot \left(\underbrace{d\bar{y}\wedge d\bar{x}( v_{1}, \dots, \hat{v}_i,\dots,  v_{n}, \tilde{w}_1, \dots, \tilde{w}_{k+1-n})}_{=\bar{\theta}_{\xi}(v_{1}, \dots, \hat{v}_i,\dots,  v_{n}, \tilde{w}_1, \dots, \tilde{w}_{k+1-n})}\right)$$
This whole is then equal to $d\bar{\theta}_{\xi}$, with $\bar{\theta}_{\xi}$ as defined in Definition \ref{Integration along the fibers}, relatively to ${w}_1, \dots, {w}_{k+1-n}$. Then, because of Lemma \ref{integrating on foo}, we have:

\begin{align*}
    \int_{\fooo{}^n}\sum_{j=1}^n \frac{\partial a_{\xi}}{\partial y_j} dy_j\wedge d\bar{y}\wedge d\bar{x}(w_1, \dots, w_{k+1-n})&= \int_{\foo{}^n} d\bar{\theta}_{\xi} \underset{\text{Stokes}}{=} \int_{\partial\foo{}^n} \bar{\theta}_{\xi}\\
    &= \sum_{i=0}^n(-1)^{i} \int_{\fooo{}^{n-1}}\tau_i^{\ast}\bar{\theta}_{\xi}\\ &= \sum_{i=0}^n(-1)^{i} \int_{\fooo{}^{n-1}}(\tau_i\times \text{id}_{\m_n})^{\ast}\theta_{\xi}(w_1, \dots, w_{k+1-n})
\end{align*}
Let us now consider the second part of (\ref{eq: first of proof}). Remark that the integration will lead us to consider (and integrate):
$$\sum_{j=1}^m \frac{\partial a_{\xi}}{\partial x_j} dx_j\wedge d\bar{y}\wedge d\bar{x}(v_1, \dots, v_n, \tilde{w}_1, \dots, \tilde{w}_{k+1-n})$$
with the $v_i$ being vertical tangent vectors, and the $\tilde{w}_i$ being horizontal. Since $dx_i(v)=0$ for $v$ vertical, it follows that $d\bar{x}(v_1, \dots, v_{q})=0$ if any $v_i$ is vertical. By the same way, $d\bar{y}(v_1, \dots, v_{p})=0$ if any $v_i$ is horizontal. It follows then that, if $p\neq n$ (recall that $p$ is the number of $dy_j$ in $d\bar{y}$, see the first formula at the start of the proof), we have:
$$dx_j\wedge d\bar{y}\wedge d\bar{x}(v_1, \dots, v_n, \tilde{w}_1, \dots, \tilde{w}_{k+1-n})=0$$
In this case, we can also see that:
$$d\bar{y}\wedge d\bar{x}(v_1, \dots, v_{n}, \tilde{w}_1, \dots, \tilde{w}_{k-n})=0$$
and thus: 
$$\int_{\fooo{}^n} \theta_{\xi} = 0 $$
and finally:
$$d\int_{\fooo{}^n} \theta_{\xi} = 0 =  \int_{\fooo{}^n}\sum_{j=1}^m \frac{\partial a_{\xi}}{\partial x_j} dx_j\wedge d\bar{y}\wedge d\bar{x}$$
In the case where $p=n$, then, using the definition of the wedge product:
\begin{align*}
    dx_j\wedge d\bar{y}\wedge d\bar{x}(v_1, \dots, v_n, \tilde{w}_1, \dots, \tilde{w}_{k+1-n})&= (-1)^n d\bar{y}\wedge dx_j\wedge  d\bar{x}(v_1, \dots, v_n, \tilde{w}_1, \dots, \tilde{w}_{k+1-n})\\
    &=(-1)^n d\bar{y}(v_1, \dots, v_n)\cdot dx_j\wedge  d\bar{x}( \tilde{w}_1, \dots, \tilde{w}_{k+1-n})
\end{align*}
with the $v_i$ being vertical tangent vectors, and the others being horizontal. Then:
\begin{equation}\label{eq: second of the proof}
    \int_{\fooo{}^n}\sum_{j=1}^m \frac{\partial a_{\xi}}{\partial x_j} dx_j\wedge d\bar{y}\wedge d\bar{x}= (-1)^n \sum_{j=1}^m dx_j\wedge  d\bar{x} \cdot \int_{\fooo{}^n} \frac{\partial a_{\xi}}{\partial x_j} d\bar{y}
\end{equation}
Finally, we have:
$$\int_{\fooo{}^n} \frac{\partial a_{\xi}}{\partial x_j} d\bar{y} = \int_{I_1}\dots \int_{I_n} \frac{\partial a_{\xi}}{\partial x_j} dy_{\xi(1)}\wedge\dots \wedge dy_{\xi(n)} \underset{(1)}{=} \frac{\partial}{\partial x_j} \int_{I_1}\dots \int_{I_n}  a_{\xi} dy_{\xi(1)}\wedge\dots \wedge dy_{\xi(n)}= \frac{\partial}{\partial x_j} \int_{\fooo{}^n} a_{\xi}d\bar{y}$$
where the $I_i$ are subsets of $\R$, independent on $x_j$ and where the equality $(1)$ is true by the Leibniz integral rule. We get thus: 
\begin{align*}
    (-1)^n \sum_{j=1}^m dx_j\wedge  d\bar{x} \cdot \int_{\fooo{}^n} \frac{\partial a_{\xi}}{\partial x_j} d\bar{y} = (-1)^n \sum_{j=1}^m dx_j\wedge  d\bar{x} \frac{\partial}{\partial x_j} \int_{\fooo{}^n} a_{\xi}d\bar{y} =& d\left( (-1)^n  \left(\int_{\fooo{}^n} a_{\xi}d\bar{y}\right) \cdot d\bar{x}  \right)\\ =&(-1)^n d\left(   \int_{\fooo{}^n} \underbrace{a_{\xi}d\bar{y}\wedge d\bar{x}}_{=\theta_{\xi}}  \right)
\end{align*}
where the last equation has the same justification as for (\ref{eq: second of the proof}). Thus, having considered the case $p=n$ such as $p\neq n$, we have the following equation:
$$(-1)^n d\int_{\fooo{}^n} \theta_{\xi} = \int_{\fooo{}^n}\sum_{j=1}^m \frac{\partial a_{\xi}}{\partial x_j} dx_j\wedge d\bar{y}\wedge d\bar{x}$$
Compiling all the results, we have:
$$\int_{\fooo{}^n}d\theta_{\xi}= \sum_{i=0}^n(-1)^{i} \int_{\fooo{}^{n-1}}(\tau_i\times \text{id}_{\m_n})^{\ast}\theta_{\xi} + (-1)^{n}d \int_{\foo{}^n}\theta_{\xi}$$
This local equation can be extended globally using the partition of unity we introduced at the beginning. We get then:
$$\int_{\fooo{}^n}d\theta= \sum_{i=0}^n(-1)^{i} \int_{\fooo{}^{n-1}}(\tau_i\times \text{id}_{\m_n})^{\ast}\theta + (-1)^{n}d \int_{\foo{}^n}\theta$$
Now, if we consider $\{\phi_n\}$ a differential form on $|\m_{\bullet}|$, we have:
\begin{align*}
    \left\{\int_{\fooo{}^n}d\phi_n\right\}&=\left\{\sum_{i=0}^n(-1)^{i} \int_{\fooo{}^{n-1}}\underbrace{(\tau_i\times \text{id}_{\m_n})^{\ast}\phi_n}_{=(\text{id}_{\fooo{}^{n-1}}\times d_i)^*\phi_{n-1}} + (-1)^{n}d \int_{\fooo{}^n}\phi_n\right\}\\&= \left\{\sum_{i=0}^n(-1)^{i} d_i^* \int_{\fooo{}^{n-1}}\phi_{n-1} + (-1)^{n}d \int_{\fooo{}^n}\phi_n\right\}= D \left\{\int_{\fooo{}^n}\phi_n\right\}
\end{align*}

So the statement follows.
\end{proof}

\begin{theorem}
The map described in Lemma $\ref{chain complex map}$ is a quasi-isomorphism, that is an isomorphism of cohomologies. 
\end{theorem}
\begin{proof}
The proof can be made using chain homotopy: see \cite{artdupont} for the definition of the homotopy and the inverse maps. 
\end{proof}

To end this section, we want to make a short remark about the pullbacks:
\begin{lemma}\label{switch the pullback}
    Let $f_{\bullet}: \m_{\bullet}\longrightarrow\m_{\bullet}'$ be a morphism of simplicial manifolds, and $\phi=\{\phi_i\}$ a differential form on $|\m_{\bullet}'|$. Then, we have:
    $$\left\{\int_{\fooo{}^i} (\text{id}_{\fooo{}^i}\times f_i)^*\phi_i\right\}= f^*_{\bullet}\left\{\int_{\fooo{}^i}\phi_i\right\}$$
    So, the definitions \ref{pullback on geom reali} for differential forms on $|\m_{\bullet}'|$ and \ref{pullback on total spac} for $\tot{}^*(\m_{\bullet}')$ commute with the integration and agree in the common cohomology. 
\end{lemma}
\begin{proof}
    It follows easily from:
    \begin{align*}
        \left(\int_{\fooo{}^i}(\text{id}_{\fooo{}^i\times f_i})^*\phi_i\right)(v_1, \dots, v_{k-i})&= \int_{\fooo{}^i}\phi_i(\square, \dots, \square,f_i^*v_1, \dots, f_i^*v_{k-i} )\\ &=  \left(\int_{\fooo{}^i}\phi_i\right)(f_i^*v_1, \dots, f_i^*v_{k-i})
    \end{align*}
\end{proof}

\section{Chern-Weil theory and Classifying space}
In this section, we want to introduce the Chern-Weil theory, by describing the construction of the Chern-Weil homomorphism and explaining the classification of the principal bundles which is implied by it. We want furthermore to extend it to the simplicial principal bundles, to find a canonical way of calculating the classes implied by the Chern-Weil theory using a specific simplicial principal bundle called the classifying space.
\subsection{Lie algebras}
Before introducing the Chern-Weil Theory, let us recall some facts and constructions about the Lie algebras. 
\begin{definition}
Let $G$ be a Lie group. Then, we define the \underline{Lie algebra of $G$}, denoted as $\fg$, as the set of all left-invariant vector fields on $G$, that is the vector fields $\fX$ such that $(\Lie_{g})_{\ast}\fX= \fX$ for all $g\in G$, where $\Lie_g: G\longrightarrow G, \;\; h\mapsto g\cdot h$ is the left multiplication by $g\in G$. The Lie brackets of this Lie algebra are given by the Lie brackets of vector fields $[\fX, \fY]$.
\end{definition}
This set can be expressed in an easier formulation:
\begin{remark}\label{asso de MaurerCartan}
Since the vector fields are left-invariant, it is quite easy to see that for a given tangent vector $v$ at a point $g\in G$, there is exactly one element $\fX\in\fg$ such that $\fX_g=v$. It implies that the Lie algebra $\fg$ (as a vector space) is isomorphic to $T_g G$ for any $g\in G$, and especially to $T_e G$, with $e$ the neutral element of $G$. Often, in the literature such as in this paper, we will amalgamate $\fg$ and $T_g G$. 
\end{remark}
We introduce now two canonical maps that occur often when working with the Lie algebras:
\begin{definition}\label{adjoint representation}[Adjoint representation]
We can define on a Lie group $G$ the \underline{conjugation map}:
$$\psi(g): G\longrightarrow G, \qquad h\mapsto g\cdot h\cdot g^{-1}$$
We can then define the map $Ad(g): \fg\longrightarrow \fg$ as being the derivative of $\psi(g)$. This map $Ad: G\longrightarrow \text{Aut}(\fg)$ is called the \underline{adjoint representation} of $G$. 
\end{definition}
\begin{definition}\label{def of maurer cartan}
On each Lie group $G$ there is a canonical $\fg$-valued 1-form, called the \\*\underline{\textit{Maurer-Cartan form}}, and denoted in this paper with $\omega$. This form sends each tangent vector $v$ at $g$ to the corresponding element of $\fg$, as explained in Remark \ref{asso de MaurerCartan}. If we consider $\fg$ as $T_e G$, it is given by:
$$\omega (v)= (\Lie_{g^{-1}})_{\ast}v$$
\end{definition}
As said in the introduction, the following work will concentrate mostly on the category of $G$-principal bundles. Let us recall the following definition of such bundles such as the definition of morphisms between them:
\begin{definition}
Let $G$ be a Lie group. A \underline{$G$-principal bundle} is a surjective morphism $\p\longrightarrow \m$ such that $G$ acts smoothly on the right of $P$ and preserves the fibres, and such that the restriction of this action on the fibre is free and transitive .\\
It implies then that $\p$ has locally the form $\m\times G$.
\end{definition}
\begin{remark}
Since $\pi: \p \longrightarrow \m$ is a submersion, it means that if we have a map $g:\m'\longrightarrow\m$, the pullback of the diagram $\m'\longrightarrow\m\longleftarrow \p$ does always exist. 
\end{remark}
\begin{definition}
Let $\pi: \p\longrightarrow \m$ and $\pi': \p' \longrightarrow\m'$ be two $G$-principal bundles. We say that $(f_1, f_2)$ is a $G$-principal bundles morphism from $\pi$ to $\pi'$ if $f_1: \m \longrightarrow \m'$ is a differential map and $f_2: \p\longrightarrow
\p'$ is another  differential map which is also invariant under the $G$ right action, such that the following diagram commutes:
$$\xymatrix{
 \p \ar[r]^{f_2}\ar[d]_{\pi'} & \p'\ar[d]^{\pi'} \\
 \m\ar[r]_{f_1}& \m'
}$$
\end{definition}

We can introduce a so-called ``connection 1-form" on a $G$-principal bundle, which is expressing the notion of parallel on the bundle. 

\begin{definition}\label{definition of connection}
Let $\pi: \p\longrightarrow \m$ be a $G$-principal bundle. Then, we define a connection as a $\fg$-valued $1$-form $\theta$ on $\p$, fulfilling the following conditions:
\begin{itemize}
    \item $\theta\circ \varphi_x =\text{id}$ for all $x\in P$, where $\varphi_x: \fg= T_{e}G\longrightarrow T_x\p$ is the derivative of $g\mapsto x\cdot g$.
    \item $Ad_g(R_g^{\ast}(\theta))=\theta\;\; \forall  g\in G$, where $R_g:\p\longrightarrow\p$ is the automorphism given by the right multiplication with $g$. 
\end{itemize}
\end{definition}
Intuitively, at each point $x$ of $P$, $\theta$ is projecting the tangent vectors to the space of vertical tangent vectors of $x$, which can be associated with the tangent space of $G$ and thus to $\fg$. Thus, the kernel of $\theta$ can be thought of as horizontal vectors, and so $\theta$ defines a notion of parallel to $M$ at $x$. This intuition is presented in the picture below: 
\begin{center}
    \includegraphics[width=0.4 \textwidth]{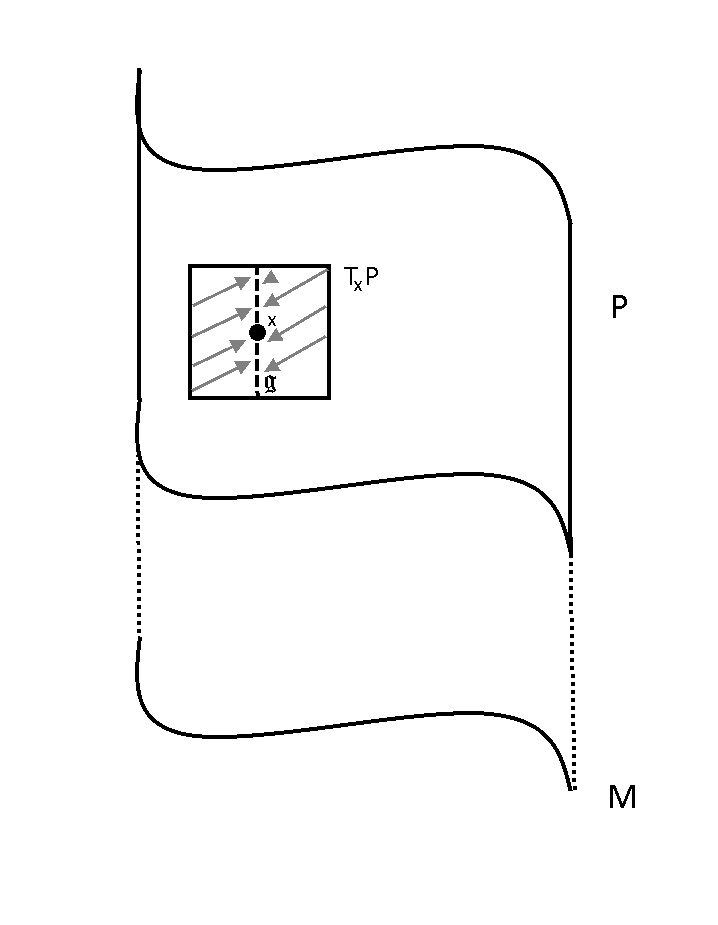}
\end{center}
\begin{lemma}\label{pullback of connection}
Let $\pi: \p\longrightarrow \m$ and $\pi': \p' \longrightarrow\m'$ be two $G$-principal bundles. Let furthermore $(f_1, f_2)$ be a $G$-principal bundles morphism from $\pi$ to $\pi'$. Let $\theta$ be a connection on $\p'$. Then, $f_2^{\ast}\theta$ is again a connection on $\p$.
\end{lemma}
\begin{proof}
Clearly, $f_2^*\theta$ is a $\fg$-valued 1-form on $\p$. Let us consider $\varphi_x:\fg \longrightarrow T_x\p$ on $\p$ as in Definition \ref{definition of connection}. Take $\overline{g(t)}\in \fg$, represented by the curve $g(t)$. Here, the over-lining of a curve means the tangent vector corresponding to it. Then, $\varphi_x(\overline{g(t)})= \overline{x\cdot g(t)}$. If we apply $f_2^*\theta$ on this result, we obtain:
$$\theta(\overline{f_2(x\cdot g(t))})=\theta(\overline{f_2(x)\cdot g(t)})$$
since $f$ is $G$-invariant. This last formula is equal to $\theta\circ\varphi'_{f_2(x)}(\overline{g(t)})= \overline{g(t)}$, where $\varphi'$ is as in Definition \ref{definition of connection} on $\p'$. Thus $f_2^*\theta \circ \varphi_x =\text{id}$.\\
Secondly, let $\overline{x(t)}$ be a tangent vector on $\p$ represented by the curve $x(t)$. Then we have:
\begin{align*}
    Ad_g\circ R_g^*(f_2^*\theta)(\overline{x(t)})&= Ad_g\theta(\overline{f_2(x(t)\cdot g)})= Ad_g\theta(\overline{f_2(x(t))\cdot g})\\&= Ad_g\circ R_g^*\theta(\overline{f_2(x(t))})= \theta(\overline{f_2(x(t))})= f_2^*\theta(\overline{x(t)})
\end{align*}
Thus $Ad_g\circ R_g^*(f_2^*\theta)=f_2^*\theta$ and it follows that $f_2^*\theta$ is a connection on $\p$. 
\end{proof}

Since we work with vector/Lie algebra-valued differential forms, it is necessary to adapt the concept of exterior product and exterior derivative, as follows:
\begin{definition}\label{wedge products}
Let $\m$ be a manifold. We introduce two new sorts of wedge products:
\begin{itemize}
    \item If $\phi$ is a $V$-valued $j$-form and $\psi$ a $W$-valued $k$-form, $V$ and $W$ both $\R$-(or $\C$-)vector spaces, the following is a generalization of the usual wedge product for vector-valued differential forms:
    $$(\phi\wedge \psi)(v_1, \dots , v_{j+k})= \frac{1}{j!\cdot k!}\sum_{\sigma\in \mathfrak{S}_{j+k}}\text{sign}(\sigma)\phi(v_{\sigma(1)}, \dots, v_{\sigma(j)})\otimes \psi(v_{\sigma(j+1)}, \dots, v_{\sigma(k+j)})$$
    Thus, we have a map $*\wedge*: \Omega^j(\m, V)\times \Omega^k(\m, W)\longrightarrow \Omega^{j+k}(\m, V\otimes W)$. Remark that the usual exterior product is the same as this one since $\R\otimes \R\cong \R$ and $\otimes$ can then be replaced by the multiplication.
    \item If $\phi$ is a $j$-form and $\psi$ a $k$-form, both $\fa$-valued where $\fa$ is a Lie algebra, we can define the following wedge product, denoted $[\phi\wedge\psi]$:
    $$[\phi\wedge\psi](v_1, \dots , v_{j+k})= \frac{1}{j!\cdot k!} \sum_{\sigma\in \mathfrak{S}_{j+k}}\text{sign}(\sigma)\left[\phi(v_{\sigma(1)}, \dots, v_{\sigma(j)}), \psi(v_{\sigma(j+1)}, \dots, v_{\sigma(k+j)})\right]$$
    This time, we have a map $[\ast\wedge\ast]: \Omega^j(\m, \fa)\times \Omega^k(\m, \fa)\longrightarrow \Omega^{j+k}(\m, \fa)$.
\end{itemize}
\end{definition}
Contrariwise to the usual exterior product, $[*\wedge*]$ fulfills the following rule:
\begin{lemma}\label{commutation of wedge product}
If $\phi$ is a $j$-form and $\psi$ a $k$-form, both $\fa$-valued where $\fa$ is a Lie algebra, we have:
$$[\phi\wedge\psi]= (-1)^{j\cdot k+1}[\psi\wedge\phi]$$
\end{lemma}
\begin{proof}
Let $\tau_0$ be the cycle permutation in $\mathfrak{S}_{j+k}$ such that $\tau_0(n)= n+ 1 \text{\;mod\;} j+k$, where we associate $0$ to $j+k$. We can see that its sign is $(-1)^{j+k-1}$: We permute $1$ and $2$, then $2$ and $3$, ... until permuting $n-1$ and $n$. Now, $\tau(n)= n+ k \text{\;mod\;} j+k = (\tau_0)^k(n)$. Thus, the sign of $\tau$ is $(-1)^{(j+k-1)\cdot k}= (-1)^{jk + k^2  - k}=(-1)^{jk}$ since $k^2-k$ is always even. Thus:
\begin{align*}
    [\phi\wedge\psi](v_1, \dots , v_{j+k})=& \frac{1}{j!\cdot k!} \sum_{\sigma\in \mathfrak{S}_{j+k}}\text{sign}(\sigma)\left[\phi(v_{\sigma(1)}, \dots, v_{\sigma(j)}), \psi(v_{\sigma(j+1)}, \dots, v_{\sigma(j+k)})\right]\\
    =& \frac{1}{j!\cdot k!} \sum_{\sigma\in \mathfrak{S}_{j+k}}\text{sign}(\sigma)\cdot (-1)^{jk} \left[\phi(v_{\sigma(k+1)}, \dots, v_{\sigma(j+k)}), \psi(v_{\sigma(1)}, \dots, v_{\sigma(k)})\right]\\
    =& \frac{1}{j!\cdot k!} \sum_{\sigma\in \mathfrak{S}_{j+k}}\text{sign}(\sigma)\cdot (-1)^{jk+1} \left[ \psi(v_{\sigma(1)}, \dots, v_{\sigma(k)}), \phi(v_{\sigma(k+1)}, \dots, v_{\sigma(j+k)}) \right]\\
    =&  (-1)^{j\cdot k+1}\cdot [\psi\wedge\phi](v_1, \dots , v_{j+k})
\end{align*}
\end{proof}

\begin{definition}
Let $\phi$ be a $V$-valued differential form on $\m$, with $V$ a vector space. Then, we define the exterior derivative $d$ as the usual exterior derivative being applied component-wise. It means, for a basis $\{e_1, \dots, e_n\}$ of $V$, $\phi$ can be written as: $\sum^n_{i=1}\phi_i\cdot e_i$, with $\phi_i$ being usual differential forms on $\m$. Thus:
$$d\phi:= \sum^n_{i=1}d\phi_i\cdot e_i$$
\end{definition}
\begin{remark}
This definition is independent of the choice of the basis $\{e_1, \dots, e_n\}$ of $V$, because of the linearity of $d$.
\end{remark}
\begin{lemma}
This new definition of the exterior derivative has the usual properties of the exterior derivative, that is:
\begin{enumerate}
    \item $d(\phi+\psi)= d\phi+ d\psi$
    \item $d(\phi\wedge\psi)= d(\phi)\wedge\psi + (-1)^{\txt{deg}(\phi)}\phi\wedge d\psi$
    \item $dd\phi =0 $
\end{enumerate}
\end{lemma}
\begin{proof}
The first and the third points are obvious and follow from the definition and the fact that the usual exterior differential also fulfils these equations. For the second one: Consider $\phi$ a $V=\text{Span}(e_1,\dots, e_m)$-valued differential $j$-form and $\psi$ a $W=\text{Span}(\hat{e}_1, \dots, \hat{e}_n)$-valued differential $k$-form. We can thus write them as $\phi=\sum_{i=1}^m\phi_i\cdot e_i$ and analogous for $\psi$. Recall that:
\begin{align*}
    &(\phi\wedge \psi)(v_1, \dots , v_{j+k})= \frac{1}{j!\cdot k!}\sum_{\sigma\in \mathfrak{S}_{j+k}}\text{sign}(\sigma)\phi(v_{\sigma(1)}, \dots, v_{\sigma(j)})\otimes \psi(v_{\sigma(j+1)}, \dots, v_{\sigma(k+j)})\\
    &= \sum_{(p,q)\in [1, m] \times[1, n] }\frac{1}{j!\cdot k!}\sum_{\sigma\in \mathfrak{S}_{j+k}}\text{sign}(\sigma)\phi_{p}(v_{\sigma(1)}, \dots, v_{\sigma(j)})\cdot \psi_{q} (v_{\sigma(j+1)}, \dots, v_{\sigma(k+j)})\cdot e_p\otimes \hat{e}_q\\
    &=\sum_{(p,q)\in [1, m] \times[1, n] } \phi_{p}\wedge\psi_{q} (v_1,\dots, v_{j+k})\cdot e_p\otimes \hat{e}_q
\end{align*}
Once we have that, the statement follows this way:
\begin{align*}
    d(\phi\wedge\psi)=& \sum_{(p,q)\in [1, m] \times[1, n] } d(\phi_{p}\wedge\psi_{q})\cdot e_p\otimes \hat{e}_q\\
    =& \sum_{(p,q)\in [1, m] \times[1, n] } \left( d\phi_p \wedge \psi_q + (-1)^{j}\phi_p\wedge d\psi_q\right) \cdot e_p\otimes \hat{e}_q\\
    =& d(\phi)\wedge\psi + (-1)^{\txt{deg}(\phi)}\phi\wedge d\psi
\end{align*}
\end{proof}

Finally, we introduce now the notion of curvature of a connection form:
\begin{definition}\label{definition of curvature}
For a given connection $\theta$, we define the corresponding curvature form $\kappa$ as:
$$\kappa= d\theta +\frac{1}{2} [\theta \wedge \theta]$$
In the case where the curvature form is equal to $0$, we say that the connection is flat. 
\end{definition}
\begin{remark}
In the literature, the curvature form is generally written $\Omega$. We write it here $\kappa$ to make a clearer difference with $\Omega^k$, the elements of the De Rham cochain complex. 
\end{remark}

\subsection{Chern-Weil theory}
In this section, we want to construct step by step the Chern-Weil homomorphism. We will not prove the properties of this map, but they can be found in \cite{dupont}. After that, we will show that we can extend this Chern-Weil construction to the simplicial $G$-principal bundles. \\
Before the construction of the morphism itself, let us introduce the following definition:
\begin{definition}
Let $V$ be a $\R$-(or $\C$-)vector space. We say that $f$ is an \underline{homogeneous polynomial} on $V$ of degree $k$ if $f: \prod^k V \longrightarrow\R$ (or $\C$) is multilinear and symmetric. The space of such maps is denoted $\R^k(V)$ (or $\C^k(V)$). 
\end{definition}
\begin{remark}\label{notation tensor poly}
This definition is equivalent to the set of maps $\Tilde{f}: S^k(V) \longrightarrow\R$ where $S^k(V)$ is the symmetric algebra. In this case, $\Tilde{f}$ is the linear map induced by $f$ and by the universal property of the tensor product. We will sometimes amalgamate $f$ and $\Tilde{f}$ in the following when the context makes it clear which one should be used.
\end{remark}
\begin{remark}\label{polynomials}
We call such maps ``homogeneous polynomials" because they are intuitively ``simulating" the multiplication of vectors. More formally, there is a bijection between the space $\R^k(V)$ and $\R[x_1, \dots, x_n]^k$ with $(e_1, \dots, e_n)$ basis of $V$, given by:
$$f\in \R(V)\mapsto \bar{f}\in \R[x_1, \dots, x_n]^k \qquad \text{with}\qquad \bar{f}(x_1, \dots, x_n)= f(v, \dots, v),\;\; v:=\sum^{n}_{i=1} e_i\cdot x_i$$
Its inverse map is called the polarization, and is given by:
$$\bar{f}\in \R[x_1, \dots, x_n]^k\mapsto f\in \R(V)\qquad \text{with}\qquad f(v_1, \dots, v_k)= \frac{1}{k!}\frac{\partial}{\partial \lambda_1}\dots \frac{\partial}{\partial \lambda_k}\bar{f}(\lambda_1\cdot v_1+\dots + \lambda_k\cdot v_k)$$
where $\bar{f}(w)$ for $w=(w_1, \dots, w_n)\in \R^n$ means $\bar{f}(w_1,\dots, w_n)$. \\
The fact that these two maps are inverse to each other is only computational and will not be discussed here. But let us remark that this implies a bijection between $f\in \R(V)$ and $v\mapsto f(v, \dots, v)$. Thus, it is enough to define the values of $f(v, \dots, v)$ to define $f$. 
\end{remark}
This sort of polynomials is still a little too general for us, and we want to restrict ourselves to the following ones:
\begin{definition}
Let $\fg$ be the Lie algebra of the Lie group $G$. We denote by $I^k(\fg)$ the set of homogeneous polynomials of degree $k$ on $\fg$ which are adjoint-invariant, that is:
$$f(Ad_{(g)}(v_1),  \dots, Ad_{(g)}(v_n))= f(v_1, \dots, v_n)\qquad \qquad \forall \; g\in G$$
Such polynomials are called the \underline{invariant polynomials} of $\fg$. 
\end{definition}
Let us now start with the construction of the Chern-Weil homomorphism: 
\begin{construction}\label{construction of CW-homo}
Let $\pi: \p \longrightarrow \m$ be a $G$-principal bundle. Let $\fg$ be the Lie algebra corresponding to $G$ and $I^k(\fg)$ be the set of invariant polynomials of degree $k$ on $\fg$. Let $\theta$ be a connection in $\p$ and $\kappa= d\theta+[\theta\wedge\theta] \in \Omega^2(\p, \fg)$ be its curvature form. Considering $\bigwedge^k \kappa\in \Omega^{2k}(\p, \bigotimes^k \fg)$, we obtain a homomorphism:
\begin{equation}\label{eq: start of cw}
    \xi(\p, \theta): I^k(\fg) \longrightarrow \Omega^{2k}(\p)\qquad \qquad f\mapsto {f}\circ \bigwedge^k \kappa
\end{equation}
\end{construction}
\begin{remark}\label{notation of polynomial}
In the rest of this paper, we will sometimes use the following notation: Let $V$ be a vector space, let $\phi_1, \dots, \phi_k$ be $V$-valued differential forms, not necessarily of the same degree, and let $f\in \R^k(V)$. Then, we write:
$$f(\phi_1, \dots, \phi_k):= {f}\circ \bigwedge_{i=1}^k \phi_i$$
So the map (\ref{eq: start of cw}) can be written as:
$$f\mapsto f(\kappa, \dots, \kappa)$$
\end{remark}
\begin{theorem}[Chern-Weil theorem, part 1]\label{chern weil theorem teil 1}
Let us consider the map $\xi(\p, \theta)$ of the construction \ref{construction of CW-homo}. Then, for all $f\in I^k(\fg)$, $\xi(f)= f\circ \bigwedge^k \kappa$ is in the image of $\pi^{\ast}: \Omega^{\ast}(\m)\longrightarrow \Omega^{\ast}(\p)$, i.e. that it is the lift of a certain $\phi_{CW}(f,\theta)\in \Omega^{2k}(\m)$ (with $CW$ for Chern-Weil). Furthermore, the form $\phi_{CW}(f, \theta)$ is closed.
\end{theorem}
Let us now finish the construction:
\begin{definition}[Chern-Weil homomorphism]
We define the following homomorphism:
$$\Phi_{CW}^{\p}: I^k(\fg)\longrightarrow H^{2k}(\m) \qquad \qquad f\mapsto [\phi_{CW}(f, \theta)]$$
where $H^{2k}(\m)$ is the $2k$-th de Rham cohomology class of $\m$ and $[\phi_{CW}(f, \theta)]$ the cohomology class of $\phi_{CW}(f, \theta)$. This homomorphism is called the \underline{\textit{Chern-Weil homomorphism}}, and $[\phi_{CW}(f, \theta)]$ the \underline{characteristic class} of $f$ and $\p$.
\end{definition}
Using the following theorem, we understand in which sense the Chern-Weil homomorphism allows a classification:
\begin{theorem}[Chern-Weil theorem, part 2]\label{chern weil theorem 2}
The Chern-Weil homomorphism has the following main properties:
\begin{itemize}
    \item The homomorphism is independent of the choice of the connection $\theta$ we used in the Construction \ref{construction of CW-homo}. 
    \item As a consequence of the first point, the map is depending on the total space $\p$ only. If $\p$ and $\p'$ are isomorphic as $G$-principal bundles over $\m$, they will produce the same characteristic class for a given $f$. 
    \item The map $\Phi_{CW}^{\p}: I^{\ast}(\fg)\longrightarrow H^{\ast}(\m)$ is an algebra homomorphism.
    \item The homomorphism $\Phi_{CW}^{\ast}: \{\text{$G$-principal bundles over $\m$}\}\longrightarrow \text{Hom}(I^{\ast}(\fg), H^{\ast}(\m)) $ commutes with the pullback. Let $g:\m'\longrightarrow \m$ be a morphism of manifolds, and $f\in I^{\ast}$, then: 
    $$g^{\ast}\Phi_{CW}^{\p}(f)= \Phi_{CW}^{g^{\ast}\p}(f)$$
\end{itemize}
\end{theorem}
\begin{proof}
    The proof of the two last theorems can be found in \cite{dupont}. Remark that J. Dupont used there Kobayashi's convention, i.e. the exterior differential of a form $\phi$ is equal to ours up to the coefficient $\frac{1}{\text{deg}(\phi)+1}$ and the exterior products we defined in Definition \ref{wedge products} have a coefficient $\frac{1}{j!+k!}$ instead of $\frac{1}{j!\cdot k!}$ for us. It means, depending on the convention, the Chern-Weil homomorphism can be checked to change by a coefficient (using Kobayashi's convention gives a further $\frac{1}{k\cdot 2^{k+1}}$ in comparison to us). Since the two definitions are proportional to each other, it does not change the properties of the homomorphism.
\end{proof}
\begin{remark}
This means that the Chern-Weil homomorphism allows us to make a classification of the $G$-principal bundles over $\m$. That is the reason why $\Phi_{CW}^{\p}(f)$ is called the characteristic class of $\p$ corresponding to $f$. 
\end{remark}

We can extend that to the simplicial manifolds. From now on, consider a simplicial $G$-principal bundle, that is:
\begin{definition}
We say that $(\pi_{\bullet}: \p_{\bullet} \longrightarrow\m_{\bullet}, d_i, s_i)$ is a simplicial $G$-principal bundle if it is a simplicial object in the category of $G$-principal bundles. Concretely, it means that each $\pi_n: \p_n \longrightarrow\m_n$ is a $G$-principal bundle, and each $d_i=((d_i)_1, (d_i)_2)$ and each $s_i=((s_i)_1, (s_i)_2)$ is a morphism of $G$-principal bundles, such that $(\p_{\bullet}, (d_i)_2, (s_i)_2)$ and $(\m_{\bullet}, (d_i)_1, (s_i)_1)$ are simplicial manifolds. 
\end{definition}

We can now define a connection form on a simplicial $G$-principal bundle. Before doing that, let us remark that for a $G$-principal bundle $\pi: \p\longrightarrow \m$, $\text{id}_{\fooo{}^n}\times\pi:\foo{}^n\times \p\longrightarrow\foo{}^n\times\m$ is canonically a principal bundle with the $G$-action given by $(t, p)\cdot g= (t, p\cdot g)$. We come then to:
\begin{definition}
Let $\pi_{\bullet}: \p_{\bullet}\longrightarrow \m_{\bullet}$ be a simplicial $G$-principal bundle. A connection on it is a $\fg$-valued differential form $\theta_{\bullet}=\{\theta_i\}_{i\in \N_0}$ on $|\p_{\bullet}|$ (as in Definition \ref{pseudo-real form}) such that $\theta_n$ is a connection on $\text{id}_{\foo{}^n}\times\pi_n:\foo{}^n\times \p_n\longrightarrow\foo{}^n\times\m_n$ for each $n$. 
\end{definition}

\begin{remark}
With the same idea as for the exterior derivative and the exterior product (see Definition \ref{ext prod and deri for pseudo simplicial}), the Lie algebras' exterior product $[\square\wedge\square]$ is well-defined for differential forms on $|\X_{\bullet}|$. Indeed, we have $f^*[\square\wedge\square]=[f^*\square\wedge f^*\square]$ for any $f$, and so the necessary gluing on $|\X_{\bullet}|$ can be fulfilled. This implies that the operation $d\theta+\frac{1}{2}[\theta\wedge\theta]$ is well-defined for differential forms on $|\X_{\bullet}|$, and that we can define the curvature of a connection form on $|\X_{\bullet}|$, as in Definition \ref{definition of curvature}.  
\end{remark}

We understand also easily that a differential form on $|\p_{\bullet}|$ can be composed with a polynomial of $I^k(\fg)$ without losing the well-definiteness. We obtain:
\begin{definition}[Simplicial Chern-Weil homomorphism] \label{simplicial chern weil}
Let $\theta_{\bullet}$ be a connection of $\pi_{\bullet}: \p_{\bullet}\longrightarrow\m_{\bullet}$. We define the homomorphism given by:
$$\hat{\Phi}^{\p_{\bullet}}_{CW}: I^k(\fg)\longrightarrow H^{2k}(\m_{\bullet})\qquad \qquad f\mapsto [\phi_{CW}(f, \theta_{\bullet})]:=\{[\phi_{CW}^i(f, \theta_i)]\}_i$$
where $\phi_{CW}^n(f, \theta_n)$ is the Chern-Weil closed form corresponding to $f$, $\pi_n: \p_n \longrightarrow \m_n$ and $\theta_n$ as described in Theorem \ref{chern weil theorem teil 1}. The brackets $[*]$ means the cohomology class.
\end{definition}
Remark that this is obtained in the same way that the usual Chern-Weil homomorphism: $\{[\phi_{CW}^i(f, \theta_i)]\}_i$ is the same as the cohomology of the descent of $f\circ \bigwedge^k \kappa_{\bullet}$, with $\kappa_{\bullet}$ the curvature of $\theta_{\bullet}$. We see that it does have the same properties as the usual Chern-Weil homomorphism:

\begin{corollary}
The cohomology class on $|\m_{\bullet}|$ defined in Definition \ref{simplicial chern weil} is independent of the choice of the connection $\theta_{\bullet}$. It only depends on the isomorphy class of $\p_{\bullet}$ for a given $\m_{\bullet}$. 
\end{corollary}
\begin{proof}
Follows directly from the Definition \ref{simplicial chern weil} and the theorem of Chern-Weil \ref{chern weil theorem 2}.
\end{proof}

\begin{lemma}\label{pullback of simplicial chern weil}
Let $\pi: \p_{\bullet}\longrightarrow \m_{\bullet}$ and $\pi': \p_{\bullet}' \longrightarrow\m_{\bullet}'$ be two simplicial $G$-principal bundles. Let $(f_1, f_2)$ be a simplicial $G$-principal bundles morphism from $\pi$ to $\pi'$. Then: $f_1^*\hat{\Phi}_{CW}^{\p_{\bullet}'}(f)= \hat{\Phi}_{CW}^{\p_{\bullet}}(f)$
\end{lemma}
\begin{proof}
Let $\theta_{\bullet}$ be a connection on $\p_{\bullet}'$, and $\kappa_{\bullet}$ its curvature. Then, if we write it precisely, we have: 
$$\hat{\Phi}_{CW}^{\p_{\bullet}'}(f)=[\phi_{CW}(f, \theta_{\bullet})] $$
Thus, we have:
$$f_1^*\hat{\Phi}_{CW}^{\p_{\bullet}'}(f)=f_1^*[\phi_{CW}(f,\theta_{\bullet})] = [f^*_1\phi_{CW}(f,\theta_{\bullet})]$$
Let $v_1, \dots, v_{2k}$ be tangent vectors of $\m_n$ at a certain point. Then, $\phi_{CW}^n(f,\theta_n) (v_1, \dots, v_{2k})= f\circ \bigwedge^k \kappa_n (w_1, \dots, w_{2k})$, where the $w_i$ are tangent vectors of $\p_n$ such that $(\pi_n)_*(w_i)=v_i$. Since $f_1\circ \pi= \pi'\circ f_2$, we obtain:
$$(f_1^n)_*v_i=(f_1^n)_*(\pi_n)_*(w_i)= (\pi_n')_*((f_2^n)_*w_i)$$
Thus, $[f^*_1\phi_{CW}(f, \theta_{\bullet})]$ is equal to the descent on $\m_{\bullet}$ of $[f^*_2 (f\circ \bigwedge^k \kappa_{\bullet})]$. The last one is equal to $[ f\circ \bigwedge^k f^*_2\kappa_{\bullet}]$. Finally, $$f^*_2\kappa_{\bullet}= f^*_2(d\theta_{\bullet}+\frac{1}{2}[\theta_{\bullet}\wedge\theta_{\bullet}])= df^*_2\theta_{\bullet}+\frac{1}{2}[f^*_2\theta_{\bullet}\wedge f^*_2\theta_{\bullet}]$$
Because of Lemma \ref{pullback of connection}, we conclude that $f^*_2\theta_{\bullet}$ is a connection of $\p_{\bullet}$. Since
$f_2^*\hat{\Phi}_{CW}^{\p_{\bullet}'}(f)$ is equal to the cohomology class of the descent of $$ f\circ \bigwedge^k \left(df^*_2\theta_{\bullet}+[f^*_2\theta_{\bullet}\wedge f^*_2\theta_{\bullet}]\right)$$
we deduce that:
$$f_2^*\hat{\Phi}_{CW}^{\p_{\bullet}'}(f)= \hat{\Phi}_{CW}^{\p_{\bullet}}(f)$$
\end{proof}
With this lemma, we extended the last property presented in Theorem \ref{chern weil theorem 2} to the simplicial principal bundles. 
\begin{remark}\label{simp and non simp char class}
This new definition for the simplicial principal bundles is an extension of the original one for the characteristic class. Indeed: having the $G$-principal bundle $\pi: \p\longrightarrow\m$, we can construct a simplicial $G$-principal bundle, which is $p: \underline{\p}_{\bullet}\longrightarrow\underline{\m}_{\bullet}$, with the corresponding constant simplicial manifolds and $p_n= \pi$. As said in Remark \ref{form for constant simplicial}, a connection on $\underline{\p}_{\bullet}$ is the same connection $\theta$ on each $\p_n=\p$. We can thus do the same reasoning with $f\circ \bigwedge^k\theta$, and for $\phi_{CW}(f, \theta)$. We conclude thus that the restriction of the characteristic class for $\p_{\bullet}$ and $f$ on $\m_0$ (or $\m_n$ for any $n$) is equal to the characteristic class for $\p$ and $f$ on $\m$. 
\end{remark}
\subsection{Classifying space}
\begin{construction}\label{simplicial principal bundle}
Let $G$ be a Lie group. There are two main possibilities to represent it as a category. The first one also denoted $G$, is the usual representation of groups in category theory, with only one object, the singleton $\ast$, and $\text{Morph}(G)=G$, with the composition given by the group multiplication $g\circ h= g\cdot h$. \\
The second way to represent a group $G$, denoted $\mathfrak{G}$, is given as follows: the objects of $\mathfrak{G}$ are the elements of $G$, and the morphisms are the elements of $G\times G$, where the source and the target of $(g_0, g_1)$ are respectively $g_0$ and $g_1$. The composition is given by $(g_1, g_2)\circ (g_0, g_1)= (g_0, g_2)$. \\
There is a natural functor $\gamma: \mathfrak{G}\longrightarrow G$, sending each object of $\mathfrak{G}$ to the unique object of $G$, and sending $(g_0, g_1)$ to $g_1\cdot g_0^{-1}$. \\
There is furthermore a free and transitive $G$-action on $\mathfrak{G}$. Indeed, for each $g\in G$, we can define a functor $F_g: \mathfrak{G}\longrightarrow\mathfrak{G}$ sending each object $h$ to $h\cdot g$ and each morphism $(h_0, h_1)$ to $(h_0\cdot g, h_1\cdot g)$. Remark that $\gamma\circ F_g= \gamma$ and that $G$ (as a group) acts freely and transitively in the fibre of each morphism of $G$ (as a category) and of the singleton $\ast$ (trivially).\\
The next step of the construction is to consider the nerves $\n G_{\bullet}$ and $\n\mathfrak{G}_{\bullet}$.
For the first one, the elements of $\n G_n$ can be represented as $(g_1, \dots, g_n)$, where each $g_i$ is a morphism in the category $G$. We obtain then $d_i(g_1, \dots, g_n)= (g_1, \dots,g_{i+1}\cdot  g_i, \dots, g_n)$ if $i\neq 0, n$. We have also $d_0(g_1, \dots, g_n)=(g_2, \dots, g_n)$ and $d_n(g_1, \dots, g_n)=(g_1, \dots, g_{n-1})$. For the degeneracy maps, we have $s_i(g_1, \dots, g_n)= (g_1, \dots, g_i, e, g_{i+1}, \dots, g_n)$, with $e$ the neutral element of $G$.\\
On the other hand, the elements of $\n\mathfrak{G}_n$ can be represented as $((g_0, g_1), (g_1, g_2), \dots, ( g_{n+1}, g_n))$. Then, the face and degeneracy maps are the same as the ones described for $\n G_n$. But it is much easier to represent it as $(g_0, g_1, g_2, \dots, g_n)$, that is, as the succession of sources/targets of the composition of morphisms. Under this form, the face maps are $d_i(g_0,\dots,  g_n)=(g_0,\dots,g_{i-1}, g_{i+1}, \dots,   g_n)$ and the degeneracy maps are given by $s_i(g_0,\dots,  g_n)= (g_0,\dots,g_i, g_i, \dots,  g_n)$.

The first nerve will thus have the form $\n G_0=\{\ast\}$, $\n G_n=G^n$, the second one will be $\n\mathfrak{G}_n=G^{n+1}$. The functor $\gamma$ now implies a collection of projections between manifolds:
$$\gamma_n: \n \mathfrak{G}_n \longrightarrow \n G_n\qquad \qquad (g_0, \dots, g_i, \dots ,g_n)\mapsto(g_1\cdot g_{0}^{-1}, \dots g_i\cdot g_{i-1}^{-1},\dots , g_n\cdot g_{n-1}^{-1})$$
for $n\geq 1$, and the trivial projection for $n=0$. \\
It is easy to check that each $\gamma_n: \n \mathfrak{G}_n \longrightarrow \n G_n$ is a $G$-principal bundle with the $G$ action inherited from the one on $\mathfrak{G}$: $(g_0, \dots, g_n)\cdot g= (g_0\cdot g,\dots,  g_n\cdot g)$. One can also check that we have obtained a simplicial $G$-principal bundle, that is that $\gamma_{n-1}\circ d_i= d_i\circ \gamma_n$ and $s_i\circ \gamma_n=\gamma_{n+1}\circ s_i $. 
\end{construction}
\begin{definition}
The simplicial $G$-principal bundle $\gamma:=\{\gamma_n\}: \n\mathfrak{G}\longrightarrow\n G$ is called the \underline{classifying bundle}, and $\n G$ in particular is called the \underline{classifying space}. In some literature, $\n\mathfrak{G}$ is denoted $EG$ and $\n G$ is denoted $BG$. We will denote it so in the following.
\end{definition}
We come now to another construction:
\begin{construction}
Let $\pi:\p\longrightarrow\m$ be a $G$-principal bundle. Let furthermore $\{U_i\}$ be an open covering of $\m$ which is trivializing the $G$-principal bundle. Let also $\{\pi^{-1}(U_i)\}$ be the corresponding covering on $\p$. We obtain thus the respective \v{C}ech groupoids $\m_{U}$ and $\p_U$ as in Definition \ref{cech groupoid}, such as the corresponding \v{C}ech nerves $\n\m_U$ and $\n\p_U$ as in Construction \ref{cech nerve}. There is canonically a projection $\pi_n: (\n\p_U)_n\longrightarrow(\n\m_U)_n$, given by:
$$x\in \pi^{-1}(U_{i_0})\cap\dots \cap \pi^{-1}(U_{i_n})\subseteq \p \mapsto \pi(x)\in U_{i_0}\cap \dots\cap U_{i_n}\subseteq \m$$
There is also a $G$-right action on $(\n\p_U)_n$ which is free and transitive on the fibers $\pi^{-1}(y)$ for $y\in \m$, directly implied by the $G$-action on $\p$. We also easily check that the map $\{\pi_n\}_n$ commutes with the face and degeneracy maps of the \v{C}ech nerve, as described in Construction \ref{cech nerve}, and that the face and degeneracy maps are $G$-invariant. \\
Thus, we obtain a simplicial $G$-principal bundle $\{\pi_n\}: \n\p_U\longrightarrow\n\m_U$.
\end{construction}
\begin{lemma}\label{first pullback}
For a given $f\in I^k(\fg)$, we can deduce the characteristic class of $\n\p_U$ corresponding to $f$ by a canonical pullback of the characteristic class of $EG$. 
\end{lemma}
\begin{proof}
This lemma can be very easily proved by applying the lemma Lemma \ref{pullback of simplicial chern weil} if we give a well-defined morphism of simplicial $G$-principal bundles from $\n\p_U\longrightarrow\n\m_U$ to $EG_{\bullet}\longrightarrow BG_{\bullet}$. We define it in the following way: With the open covering $\{U_i\}$ on $\m$ which is trivializing $\pi:\p\longrightarrow\m$, we have the smooth transition maps $\Gamma_{i,j}: U_i\cap U_j\longrightarrow G$. Then, we set:
$$f_1^n: (\n\m_U)_{n}\longrightarrow BG_{n}, \qquad y\in U_{i_0}\cap \dots \cap U_{i_n}\mapsto ( \Gamma_{i_0, i_1}(y), \Gamma_{i_1, i_2}(y),\dots, \Gamma_{i_{n-1}, i_n}(y) )$$
On the other side, we define:
\begin{equation}\label{eq: transition maps}
\begin{aligned}
    &f_2^n: (\n\p_U)_n \cong (\n\m_U)_n \times G\longrightarrow EG_n,\\
    &(x,g)\in (U_{i_0}\cap \dots \cap U_{i_n})\times G \mapsto  (g, \Gamma_{i_0, i_1}(x)\cdot g, \Gamma_{i_1, i_2}(x)\cdot \Gamma_{i_0, i_1}(x)\cdot g, \dots,\Gamma_{i_{n-1}, i_n}(x) \cdots  \Gamma_{i_0, i_1}(x)\cdot g  )
\end{aligned}
\end{equation}
\\
The identification $(\n\p_U)_n \cong (\n\m_U)_n \times G$ is given in the following way: On $\pi^{-1}(U_{i_0}\cap \dots \cap U_{i_n})$ we use the trivialization corresponding to $U_{i_0}$ to have the identification $\pi^{-1}(U_{i_0}\cap \dots \cap U_{i_n})\cong(U_{i_0}\cap \dots \cap U_{i_n})\times G  $\\
Recall that the properties of the transition maps give us the equation: $$\Gamma_{i_{n-1}, i_n}(x) \cdots  \Gamma_{i_0, i_1}(x)= \Gamma_{i_0, i_n}(x)$$ We can thus rewrite the map (\ref{eq: transition maps}) as:
$$(x,g)\in (U_{i_0}\cap \dots \cap U_{i_n})\times G \mapsto ( \underbrace{\Gamma_{i_0, i_0}(x)}_{=\text{id}}\cdot g, \Gamma_{i_0, i_1}(x)\cdot g, \Gamma_{i_0, i_2}(x)\cdot g, \dots, \Gamma_{i_0, i_n}(x)\cdot g)$$
We should now prove that the map $(f_1, f_2)=\{(f_1^n, f_2^n)\}$ is a well-defined morphism of simplicial $G$-principal bundles: First of all, we should check that $f_2^n$ is $G$-invariant:
\begin{align*}
    f_2^n((x,g)\cdot h)= f_2^n(x,g\cdot h)&= (\Gamma_{i_0, i_0}(x)\cdot g\cdot h, \dots, \Gamma_{i_0, i_n}(x)\cdot g\cdot h) \\
    &=  (\Gamma_{i_0, i_0}(x)\cdot g, \dots, \Gamma_{i_0, i_n}(x)\cdot g)\cdot h =  f_2^n(x,g)\cdot h
\end{align*}

We should now check that $(f_1^n, f_2^n)$ forms a $G$-principal bundle morphism for each $n$. For $(x,g)\in EG_n$, we have:
\begin{align*}
    f_1^n\circ \pi_n (x,g)= f_1^n(x)&= (\Gamma_{i_0, i_1}(x), \Gamma_{i_1, i_2}(x), \dots, \Gamma_{i_{n-1}, i_n}(x))\\&= \gamma_n (\Gamma_{i_0, i_0}(x)\cdot g, \dots, \Gamma_{i_0, i_n}(x)\cdot g) = \gamma_n\circ f^n_2(x,g) 
\end{align*}

Thus $(f_1^n, f_2^n)$ is a $G$-principal bundle morphism. The last thing to check is that $f_1=\{f_1^n\}$ and $f_2=\{f_2^n\}$ are well-defined simplicial maps. Let $x\in BG_n$, $(x, g)\in EG_n$.
\begin{align*}
    \text{for}\; j\neq 0, n\qquad d_j\circ f_1^n (x)=& (\Gamma_{i_{0}, i_1}(x), \dots,\underbrace{\Gamma_{i_j,i_{j+1}}(x)\cdot \Gamma_{i_{j-1}, i_j}(x)}_{=\Gamma_{i_{j-1}, i_{j+1}}(x)}, \dots , \Gamma_{i_{n-1}, i_n}(x)) = f_1^{n-1}\circ d_j(x) \\
    d_0\circ f_1^n (x)=& (\Gamma_{i_{1}, i_2}(x), \dots,\Gamma_{i_{n-1}, i_n}(x))= f_1^{n-1}\circ d_0(x)\qquad \text{(analogous for $d_n$)}\\
    s_j\circ f_1^n (x)=&(\Gamma_{i_{0}, i_1}(x), \dots,\underbrace{\Gamma_{i_j,i_j}(x)}_{=\text{id}}, \dots , \Gamma_{i_{n-1}, i_n}(x))= f_1^{n+1}\circ s_j(x)\\
    \text{for}\; j\neq 0\qquad d_j\circ f_2^n (x, g)=& (\Gamma_{i_0, i_0}(x)\cdot g,\dots  ,(\widehat{\Gamma_{i_0, i_j}(x)\cdot g}), \dots, \Gamma_{i_0, i_n}(x)\cdot g)= f_2^{n-1}\circ d_j(x, g)\\
    s_j\circ f_2^n (x, g)=& (\Gamma_{i_0, i_0}(x)\cdot g,\dots, \Gamma_{i_0, i_j}(x)\cdot g, \underbrace{\Gamma_{i_0, i_j}(x)}_{=\Gamma_{i_j, i_j}(x)\cdot \Gamma_{i_0, i_j}(x)}\cdot g , \dots, \Gamma_{i_0, i_n}(x)\cdot g)\\=& f_2^{n+1}\circ s_j(x, g)
\end{align*}
Let us finally consider the case of $f_2^{n-1}\circ d_0$. The difference is here that the $g$ is well-defined using the trivialization of the first element of the intersection $U_{i_0}\cap \dots \cap U_{i_n}$. But after using $d_0$, the first element changes, and we should use the trivialization of $U_{i_1}$. Thus:
$$d_0(x, g)= (x, \Gamma_{i_0, i_1}(x)\cdot g)\in U_{i_1}\cap \dots \cap U_{i_n}$$ 
Applying $f_2^{n-1}$ on it, we obtain:
$$(\Gamma_{i_0, i_1}(x)\cdot g, \Gamma_{i_1, i_2}(x)\cdot\Gamma_{i_0, i_1}(x)\cdot g, \dots, \Gamma_{i_1, i_n}(x)\cdot\Gamma_{i_0, i_1}(x)\cdot g) = (\Gamma_{i_0, i_1}(x)\cdot g, \Gamma_{i_0, i_2}(x)\cdot g, \dots \Gamma_{i_0, i_n}(x)\cdot g) $$
which is exactly $d_0\circ f_2^n (x, g)$.\\
So we proved that the map $(f_1, f_2)$ is a well-defined morphism of simplicial $G$-principal bundles, and it gives us a canonical way to pull back the characteristic class of $\gamma_n: EG_{\bullet}\longrightarrow BG_{\bullet}$.
\end{proof}

\begin{construction}\label{rho}
Having the (non-simplicial) $G$-principal bundle $\pi: \p\longrightarrow\m$, we can construct another simplicial $G$-principal bundle, which is $p: \underline{\p}_{\bullet}\longrightarrow\underline{\m}_{\bullet}$ as in Remark \ref{simp and non simp char class}. We obtain then a morphism $(\{\rho_1^n\}, \{\rho_2^n\})$ from $\pi:\n\p_U \longrightarrow \n\m_U$ to $p:\underline{\p}_{\bullet}\longrightarrow\underline{\m}_{\bullet}$ given by the projection $(\rho_1, \rho_2)$ from the \v{C}ech nerves of $\m$ and $\p$ to the manifolds $\m$ and $\p$ themselves. We can easily check that $\rho_i^n\circ d_j= \rho_i^{n+1}$ and $\rho_i^n\circ s_j= \rho_i^{n-1}$, thus both $\rho_1$ and $\rho_2$ are simplicial maps ($d_j$ and $s_j$ are the identities for the constant simplicial manifolds). Also, $(\rho_1^n, \rho_2^n)$ is furthermore a $G$-principal bundles morphism. Thus, $(\rho_1, \rho_2)$ is well-defined as simplicial $G$-principal bundles morphism.
\end{construction}
We can summarize it: 
$$\xymatrix{
EG_{\bullet} \ar[d]_{\gamma} & \n\p_U \ar[d]_{\pi} \ar[l]_{f_2}\ar[r]^{\rho_2} & \underline{\p}_{\bullet}\ar[d]^{p}\\
BG_{\bullet} & \n\m_U \ar[l]^{f_1}\ar[r]_{\rho_1}& \underline{\m}_{\bullet}
}$$
\begin{remark}\label{second pullback}
    Recall from Remark \ref{simplicial form on usual manifolds}, that $\tot{}^*(\underline{\m}_{\bullet})= \Omega^*(\m)$, since $C^{p,q}(\underline{\m}_{\bullet})=0$ as soon as $p\neq 0$. On the other side, as explained in Remark \ref{extension of cech complex}, the complex of forms $\tot{}^*(\n\m_U)$ of $\n\m_U$ is the \v{C}ech total complex of $\m$ and thus the cohomology of $\n\m_U$ is the \v{C}ech cohomology of $\m$. Thus, if we take a form $\phi\in \Omega^k(\m) = \tot{}^k(\underline{\m}_{\bullet})$, and if we pull it back through $\rho_0$, it gives a form on $\Omega^k(B G_0)\times \{0\}\times \dots\times \{0\}$. More precisely, it will be the form $\bigsqcup_i \phi_{|U_i}$. It is well-known that this map implies an isomorphism of cohomology between the \v{C}ech cohomology and the De Rahm cohomology. The proof of it can be found in \cite{botttu}. The inverse map (we denote it $\hat{\rho}$) is also described there as an algorithm. Then, since pulling back the characteristic class of $\p$ for $f$ on $\m$ gives the characteristic class of $\n_U \p$ (Lemma \ref{pullback of simplicial chern weil}), then the image through $\hat{\rho}$ of this characteristic class is the characteristic class of $\p$ ($\hat{\rho}$ is the inverse of the pullback $\rho^*$). Remark that the use of the pullback property of the simplicial Chern-Weil is still correct, although we change from forms on $|BG_{\bullet}|$ to forms on $\tot{}^*(BG_{\bullet})$ because the pullbacks agree on their common cohomology (Lemma \ref{switch the pullback}).
\end{remark}
\begin{theorem}
Let $\pi: \p\longrightarrow\m$ be a $G$-principal bundle.
Given the characteristic class $\Phi^{EG_{\bullet}}(f)$ of $EG_{\bullet}$ for a certain $f\in I^k$, there is a canonical way to calculate the characteristic class of $\p$ relatively to $f$. 
\end{theorem}
\begin{proof}
This theorem is already proved with the preceding remarks and constructions: indeed, we found a canonical way to calculate the characteristic class of $\n_U\p$ on $\n_U\m$ using the one on $BG_{\bullet}$ in lemma \ref{first pullback}, and then we found a canonical way to calculate the characteristic class of $\p$ on $\m$ using the one on $\n_U\m$ in remark \ref{second pullback}.
\end{proof}

\subsection{Shulman's construction}
Since we now proved that there is a general relation between the characteristic classes of any $G$-principal bundle and the classifying space of $G$, it is useful to find a canonical way of calculating the characteristic classes of the last. We will see in this part how to do it, by following a construction made by Shulman in his thesis \cite{shulman}.\\
Before we start, we need several small lemmas:
\begin{lemma}\label{addition of connections}
Let $\theta_1, \dots , \theta_n$ be connections on the $G$-principal bundle $\pi: \p\longrightarrow\m$, and $\lambda_1(x), \dots, \lambda_n(x)$ be real-valued functions on $\m$ such that $\sum^n_{i=1}\lambda_i(x)=1\; \forall x\in \m$. Then, $\sum^n_{i=1}\lambda_i\cdot \theta_i$ is again a connection.
\end{lemma}
\begin{proof}
Clearly, it is again a $\fg$-valued $1$-form. Furthermore:
$$\left(\sum^n_{i=1}\lambda_i\cdot \theta_i\right)\circ \varphi_x= \sum^n_{i=1}\lambda_i\cdot \text{id}= \text{id}$$
Secondly, since the derivatives of maps are linear:
$$Ad_g\left((R_g)^*\left(\sum^n_{i=1}\lambda_i\cdot \theta_i\right)\right)=\sum^n_{i=1}\lambda_i\cdot Ad_g((R_g)^*(\theta_i))= \sum^n_{i=1}\lambda_i\cdot \theta_i$$
So we proved that it is a connection. 
\end{proof}
\begin{lemma}
Let $G$ be a Lie group. Then, the Maurer-Cartan form $\omega$ (see Definition \ref{def of maurer cartan}) fulfills the Maurer-Cartan equation, that is:
$$d\omega+\frac{1}{2}[\omega\wedge\omega]=0$$
\end{lemma}
\begin{proof}
Let us consider first $d\omega$. Let $v$ and $w$ be two tangent vectors on the point $g\in G$, and $\fX$ and $\fY$ be the corresponding left-invariant vector fields of $v$ and $w$. Let furthermore $\{e_1, \dots, e_n\}$ be a basis of $\fg$ (each $e_i$ is thus a left-invariant vector field). We can then write $\omega=\sum_{i=1}^n \omega_i\cdot e_i$. Furthermore, denote $\fX=\sum_{i=1}^n \fX_i\cdot e_i$ and $\fY=\sum_{i=1}^n \fY_i\cdot e_i$ with $\fX_i$ and $\fY_i$ being real numbers. By definition, $\omega_i(\fX)=\fX_i$. Then:
\begin{align*}
    d\omega(v, w)=& \sum_{i=1}^{n}d\omega_i(v, w)\cdot e_i
    = \sum_{i=1}^{n}d(\omega_i)_g(\fX, \fY)\cdot e_i\\
    =& \sum_{i=1}^{n}(\fX_g(\omega_i(\fY))-\fY_g(\omega_i(\fX))-(\omega_i)_g([\fX, \fY]))\cdot e_i
\end{align*}
Since $\fY$ is left-invariant, $\omega_i(\fY_g)=\fY_i\;\; \forall g\in G$, thus $\omega_i(\fY)$ is constant and $\fX(\omega_i(\fY))=0$. Analogously, we get $\fY(\omega_i(\fX))=0$. Finally:
\begin{align*}
    d\omega(v,w)=& \sum_{i=1}^{n} -\omega_i([\fX, \fY]_g)\cdot e_i
    =  \sum_{i=1}^{n} -\omega_i ([\fX, \fY]_e)\cdot e_i \\=& -\omega([\fX, \fY]_e)
    = -[\fX, \fY]= -[\omega(v), \omega(w)]
\end{align*}
Let us now consider the part $[\omega\wedge\omega]$:
\begin{align*}
    [\omega\wedge\omega]=& \left([\omega(v), \omega(w )]- [\omega(w), \omega(v)]\right)\\
    =& \left( \omega(v)\circ \omega(w) -\omega(w)\circ \omega(v) - \omega(w)\circ \omega(v) +\omega(v)\circ \omega(w)\right)\\
    =& 2\cdot \omega(v)\circ \omega(w) -2\cdot \omega(w)\circ \omega(v) = 2\cdot [\omega(v), \omega(w)]
\end{align*}
So we obtain $d\omega+\frac{1}{2}[\omega\wedge\omega]=0$. 
\end{proof}
\begin{corollary}\label{mc equation pullback}
Let $\m$ be a manifold and $f: \m\longrightarrow G$ a smooth map, $G$ being a Lie group equipped with the Maurer-Cartan form $\omega$. Then, $f^{\ast}\omega$ also fulfils the Maurer-Cartan equation.
\end{corollary}
\begin{proof}
We already know that $f^{\ast}d\omega= d f^{\ast}\omega$. Furthermore, for $v$ and $w$ tangent vectors at $p\in \m$:
$$f^{\ast}[\omega\wedge\omega](v, w)= \omega(f_{\ast}v)\circ \omega(f_{\ast}w)-\omega(f_{\ast}w)\circ \omega(f_{\ast}v)= [f^{\ast}\omega\wedge f^{\ast}\omega]$$
Thus $d f^{\ast}\omega +\frac{1}{2}[f^{\ast}\omega\wedge f^{\ast}\omega]= f^*(d\omega+\frac{1}{2}[\omega\wedge\omega]) =0 $. 
\end{proof}

Now, we can introduce a canonical connection $1$-form on $EG_{\bullet}$. 
\begin{definition}\label{pullback omega} 
Let $G$ be a Lie group and $\text{id}\times \gamma: \foo{}^{\ast}\times EG_{\bullet}\longrightarrow \foo{}^{\ast}\times BG_{\bullet}$ be the pseudo-realization of the construction \ref{simplicial principal bundle}. Then, we define on each $\foo{}^n\times EG_n$ a $\fg$-valued 1-form $\hat{\omega}_n$, defined as:
\begin{equation}\label{eq: 1form simp}
\hat{\omega}_n=\sum^n_{j=0}t_j\cdot \omega_j
\end{equation}
with $t_j$ being the $j$-th barycentric coordinate of $\foo{}^n$ and $\omega_j$ the pullback of the Maurer-Cartan form on $G$ along the $j$-th projection $\pi_j: EG_n\cong G^{n+1}\longrightarrow G$. 
\end{definition}
\newpage
\begin{lemma}
The form (\ref{eq: 1form simp}) is a well-defined form in the sense of the definition \ref{pseudo-real form}, i.e. that $\hat{\omega}=\{\hat{\omega}_n\}$ is a differential form on $|EG_{\bullet}|$. 
\end{lemma}
\begin{proof}
Let us first work on $\foo{}^{n-1}\times EG_n$. Let $v\times (w_0, \dots, w_n)$ be a tangent vector at the point $(t_0, \dots, t_{n-1})\times (g_0, \dots, g_n)$. Then:
\begin{align*}
    (\text{id}\times d_i)^*\hat{\omega}_{n-1} (v\times (w_0, \dots, w_n))&= \hat{\omega}_{n-1} (v\times (w_0, \dots,\widehat{w_i}, \dots, w_n))\\&= \sum_{j=0}^{i-1}t_j\cdot \omega(w_j) + \sum_{j=i}^{n-1}t_j\cdot \omega(w_{j+1})
\end{align*}
On the other side, we have:
\begin{align*}
    (\tau_i\times \text{id})^*\hat{\omega}_{n}(v\times (w_0, \dots, w_n)) &= \hat{\omega}_{n}((\tau_i)_*v\times (w_0, \dots, w_n))\\
    &= \sum_{j=0}^{i-1}t_j\cdot \omega(w_j) + \sum_{j=i}^{n-1}t_j\cdot \omega(w_{j+1})
\end{align*}
Thus : $(\tau_j\times \text{id})^{\ast}\hat{\omega}_i= (\text{id}\times d_j)^{\ast} \hat{\omega}_{i-1}$. Let us now work on $\foo{}^n\times EG_{n-1}$. Let $v\times (w_0, \dots, w_{n-1})$ be a tangent vector at the point $(t_0, \dots, t_n)\times (g_0, \dots, g_{n-1})$. Then:
\begin{align*}
    (\text{id}\times s_i)^* \hat{\omega}_n (v\times (w_0, \dots, w_{n-1})) &= \hat{\omega}_n (v\times (w_0,\dots, w_i, w_i \dots, w_{n-1}))\\ &= \sum_{j=0}^{i-1}t_j\cdot \omega(w_j) + (t_i+t_{i+1})\cdot \omega(w_i) + \sum_{j=i+2}^{n}t_j\cdot \omega(w_{j-1})
\end{align*}
On the other hand, we get: 
\begin{align*}
    (\eta_i\times \text{id})^* \hat{\omega}_{n-1} (v\times (w_0, \dots, w_{n-1})) &= \hat{\omega}_{n-1} ((\eta_i)_*v\times (w_0,\dots w_{n-1}))\\&= \sum_{j=0}^{i-1}t_j\cdot \omega(w_j) + (t_i+t_{i+1})\cdot \omega(w_i) + \sum_{j=i+2}^{n}t_j\cdot \omega(w_{j-1})
\end{align*}
So we have: $(\eta_j\times \text{id})^{\ast}\hat{\omega}_{i-1}= (\text{id}\times s_j)^{\ast} \hat{\omega}_{i}$.
So the well-definiteness is checked.
\end{proof}
\begin{lemma}
The differential form $\hat{\omega}_n$ is a connection on the $G$-principal bundle $EG_n\longrightarrow  BG_n$. 

\end{lemma}
\begin{proof}
Let us first consider $\omega_i$. It is a $\fg$-valued 1-form on $\foo{}^n\times EG_n$. Let us prove first that $\omega_i\circ \varphi_x=\text{id}$ as in Definition \ref{definition of connection}. Consider a smooth curve $g(t)\subset G$ such that $g(0)=e$, and thus representing a tangent vector $\bar{g}\in T_e G$. Then $\varphi_x(\bar{g})$, for $x=p\times (g_0, \dots , g_n)\in \foo{}^n\times EG_n$, will be the tangent vector represented by $x\cdot g(t)= p\times(g_0\cdot g(t), \dots, g_n\cdot g(t))$. Applying $\omega_i$ on it is the same as applying the Maurer-Cartan form $\omega$ of $G$ on $g_i\cdot g(t)$. The result will be represented by the curve $(g_i\cdot g(0))^{-1}\cdot g_i\cdot g(t)= g(t)$. Thus:
$$\omega_i\circ \varphi_x=\text{id}$$
Let us now prove that $Ad_g(R_g^{\ast}(\omega_i))=\omega_i$ for all $g\in G$. Let $p(t)\times (g_0(t), \dots, g_n(t))$ be a curve in $\foo{}^n\times EG_n$, representing the tangent vector $\bar{g}$. Then $\omega_i(\bar{g})= \omega((\pi_i)_*(\bar{g}))$ is the left-invariant vector field on $G$ represented at $e$ by the curve $(g_i(0))^{-1}\cdot g_i(t)$. Applying $R_g^*$, we obtain the left-invariant vector field represented at $g$ by the curve $(g_i(0))^{-1}\cdot g_i(t)\cdot g$ and thus at $e$ by the curve $g^{-1}\cdot (g_i(0))^{-1}\cdot g_i(t)\cdot g$. Finally, applying $Ad_g$ gives us the left-invariant vector field defined at $e$ by: 
$$g\cdot g^{-1}\cdot (g_i(0))^{-1}\cdot g_i(t)\cdot g\cdot g^{-1}= (g_i(0))^{-1}\cdot g_i(t) $$
Thus:
$$Ad_g(R_g^{\ast}(\omega_i))=\omega_i$$
It follows that $\omega_i$ is a connection for each $i$, thus applying Lemma \ref{addition of connections} and knowing that $\sum_{i=0}^n t_i=1$ by definition, it follows that $\hat{\omega}_n$ is a connection.
\end{proof}
The next step for obtaining the characteristic class is to calculate the curvature of this connection:
\begin{construction}\label{connection on ngn}
The curvature $\kappa_n$ of the differential form $\hat{\omega}_n$ can be calculated as following:
\begin{align*}
    \kappa_n =& d\omega_n +\frac{1}{2}[\omega_n\wedge\omega_n]\\
    =& d\left(\sum_{j=0}^{n} t_j\wedge \omega_j \right)+\frac{1}{2} \left[\left(\sum_{i=0}^{n} t_i\cdot \omega_i\right)\wedge\left(\sum_{j=0}^{n} t_j\cdot \omega_j\right)\right]\\
    =& \sum_{j=0}^{n} \left(dt_j\wedge \omega_j + t_j\wedge d\omega_j\right) + \frac{1}{2}\sum_{\substack{i=0\\ j=0}}^n t_i\cdot t_j\cdot[\omega_i\wedge \omega_j]\\
    =& \sum_{j=0}^n dt_j\wedge \omega_j + \sum_{j=0}^n\left(t_j\cdot d\omega_j +\frac{1}{2} t_j^2[\omega_j\wedge\omega_j]\right)+  \sum_{0\leq i<j \leq n} t_i\cdot t_j [\omega_i\wedge\omega_j]
\end{align*}
Where the last part of the last line is obtained using the fact that $[\omega_i\wedge \omega_j]= [\omega_j\wedge \omega_i]$ (see Lemma \ref{commutation of wedge product}).
Since $\omega_i$ fulfills the Maurer-Cartan equation (Corollary \ref{mc equation pullback}), it follows that $d\omega_i=-\frac{1}{2} [\omega_i\wedge\omega_i]$ for all $i$. Thus:
\begin{equation}\label{eq: curvature}
    \kappa^n= \sum_{j=0}^n dt_j\wedge \omega_j + \frac{1}{2}\sum_{j=0}^n (t_j^2-t_j)\cdot[\omega_j\wedge\omega_j]+ \sum_{0\leq i<j \leq n} t_i\cdot t_j [\omega_i\wedge\omega_j] 
\end{equation}

We further know that $t_0= 1 - \sum_{i=1}^n t_i$. Thus:
$$dt_0\wedge \omega_0= d\left(1 - \sum_{i=1}^n t_i\right)\wedge \omega_0 = -\left(\sum_{i=1}^n dt_i\right)\wedge \omega_0= -\sum_{i=1}^n dt_i\wedge \omega_0$$
For the first term of (\ref{eq: curvature}), it follows that:
$$\sum_{j=0}^n dt_j\wedge\omega_j= \sum_{j=1}^n dt_j\wedge (\omega_j - \omega_0)$$
If we consider the second term of (\ref{eq: curvature}):
\begin{align*}
     \frac{1}{2}(t_0^2-t_0)\cdot [\omega_0\wedge \omega_0]=& \frac{1}{2} \left(\left(1 - \sum_{i=1}^n t_i\right)^2-1 + \sum_{i=1}^n t_i\right)\cdot [\omega_0\wedge \omega_0]\\
    =& \frac{1}{2} \left(-\sum_{i=1}^n t_i + \left(\sum_{i=1}^n t_i\right)^2\right)\cdot [\omega_0\wedge \omega_0]\\
    =& \frac{1}{2} \left(-\sum_{i=1}^n t_i + \sum_{i=1}^n t_i^2 + 2\cdot \sum_{1\leq i<j \leq n} t_i\cdot t_j\right)\cdot [\omega_0\wedge \omega_0]\\
    =& \frac{1}{2}\underbrace{  \sum_{i=1}^n (t_i^2-t_i)\cdot [\omega_0\wedge \omega_0]}_{\text{\barroman{I}}}  +  \underbrace{\sum_{1\leq i<j \leq n} (t_i\cdot t_j) \cdot [\omega_0\wedge \omega_0]}_{\text{\barroman{II}}}
\end{align*}

We obtain for the third term of (\ref{eq: curvature}):
\begin{align*}
    &\sum_{0\leq i<j \leq n} t_i\cdot  t_j [\omega_i\wedge\omega_j] + \text{\barroman{II}}= \sum_{1\leq i<j \leq n} t_i\cdot t_j [\omega_i\wedge\omega_j] + \sum_{1\leq i<j \leq n} t_i\cdot t_j \cdot [\omega_0\wedge \omega_0] + \sum_{j=1}^n t_0 \cdot t_j [\omega_0\wedge \omega_j]\\
    =&\sum_{1\leq i<j \leq n} t_i\cdot t_j [\omega_i\wedge\omega_j] + \sum_{1\leq i<j \leq n} (t_i\cdot t_j) \cdot [\omega_0\wedge \omega_0] + \sum_{j=1}^n \left(t_j - \sum_{i=1}^n t_j\cdot t_i\right) [\omega_0\wedge \omega_j]\\
    =&\sum_{1\leq i<j \leq n} \left((t_i\cdot t_j)\cdot [\omega_i\wedge\omega_j] +  (t_i\cdot t_j) \cdot [\omega_0\wedge \omega_0] - (t_i\cdot t_j) \cdot [\omega_0\wedge \omega_i]- (t_i\cdot t_j) \cdot [\omega_0\wedge \omega_j]\right)+ \\&\qquad\qquad\qquad\qquad\qquad\qquad\qquad\qquad\qquad\qquad\qquad\qquad\qquad\qquad\qquad\qquad \left.\;\;\sum_{j=1}^n \left(t_j - t_j^2\right) [\omega_0\wedge \omega_j]\right. \\
\end{align*}
Knowing that:
\begin{equation}\label{eq: curv2}
    [\omega_i - \omega_0 \wedge \omega_j - \omega_0]= [\omega_i \wedge \omega_j ] - [\omega_0 \wedge \omega_j ] - \underbrace{[\omega_i \wedge \omega_0 ]}_{=  [\omega_0 \wedge \omega_i ]} + [\omega_0 \wedge \omega_0 ]
\end{equation}
we obtain: 
$$\sum_{0\leq i<j \leq n} t_i\cdot  t_j [\omega_i\wedge\omega_j] + \text{\barroman{II}}= \sum_{1\leq i<j \leq n} t_i\cdot  t_j [\omega_i- \omega_0\wedge\omega_j-\omega_0] - \underbrace{\sum_{j=1}^n \left(t_j^2 - t_j\right) [\omega_0\wedge \omega_j]}_{\barroman{III}}$$
Finally, we get, again because of (\ref{eq: curv2}):
$$\frac{1}{2}\sum_{j=1}^n (t_j^2-t_j)\cdot[\omega_j\wedge\omega_j]  +  \frac{1}{2}\text{\barroman{I}}-\text{\barroman{III}}=\frac{1}{2} \sum_{j=1}^n (t_j^2-t_j)\cdot[\omega_j- \omega_0\wedge\omega_j- \omega_0]$$
Thus the curvature of $\hat{\omega}_n$ is given by:
$$\kappa_n = \sum_{j=1}^n dt_j\wedge (\omega_j - \omega_0) +\frac{1}{2} \sum_{j=1}^n (t_j^2-t_j)\cdot[\omega_j- \omega_0\wedge\omega_j- \omega_0] + \sum_{1\leq i<j \leq n} t_i\cdot  t_j [\omega_i- \omega_0\wedge\omega_j-\omega_0]$$
\end{construction}
Now, the end of the construction is follows:
\begin{definition}[The Shulman's construction]\label{The Shulman's construction}
Let $G$ be a Lie group and let $f\in I^k(\fg)$. Then we can associate to $f$ a sequence of differential forms given by:
$$\chi_n(f)= \int_{\foo{}^n}f\circ \bigwedge^k\kappa_n$$
for each $n$, with $\kappa_n$ being equal:
$$\kappa_n = \sum_{j=1}^n dt_j\wedge (\omega_j - \omega_0) +\frac{1}{2} \sum_{j=1}^n (t_j^2-t_j)\cdot[\omega_j- \omega_0\wedge\omega_j- \omega_0] + \sum_{1\leq i<j \leq n} t_i\cdot  t_j [\omega_i- \omega_0\wedge\omega_j-\omega_0]$$
\end{definition}
\begin{remark}
The construction is in fact the following one: with $\hat{\omega}$, we defined a connection $1$-form for the simplicial $G$-principal bundle $\gamma:EG _\bullet\longrightarrow BG_{\bullet}$. As explained in Definition \ref{simplicial chern weil}, we can construct a characteristic class of $f\in I^k(\fg)$. For that, we need to calculate $f\circ \bigwedge^k\kappa_n$ on $EG_\bullet$. The idea of Shulman was to change the definition of differential form at this moment, that is: since the current definition follows Definition \ref{pseudo-real form}, we use now the map defined in Lemma \ref{chain complex map} to get a differential form in the definition given by Definition \ref{lalaoui}. Since the two definitions give rise to the same cohomology, and the aim of the Chern-Weil homomorphism is to get a cohomology class, this is still well-defined.
\end{remark}
\begin{theorem}[Shulman's Theorem]\label{Shulman theorem}
Let $f\in I^k(\fg)$. Then Shulman's construction fulfils the following equations:
\begin{enumerate}
    \item $\chi_n(f)=0$ if $n>k$, with $k$ the degree of $f$.
    \item $0= \delta\chi_{n-1}(f) + (-1)^n d\chi_n(f)$, with $\delta$ as in Definition \ref{definition of normalized complex}. 
    \item $s_i \chi_n(f)=0$ for all $i\in [ 0, n ]$. 
    \item $\chi_0(f)=0$.
\end{enumerate}
\end{theorem}
\begin{proof}
For the first point, let us do that inductively over $k$. Because of the integration along the fibers, $\chi_n(f)$ is given by:
$$\chi_n(f)(u_1, \dots, u_{2k-n})= \int_{\foo{}^n}f\circ \bigwedge^k \kappa_n (\underbrace{\square, \dots, \square}_{\text{n-times}}, \underbrace{\Tilde{u}_1, \dots, \Tilde{u}_{2k-n}}_{\text{lifts of the $u_i$'s}})$$
where the squares $\square$ will only contain vertical vectors of $\foo{}^n\times EG_n$, with $\foo{}^n$ being the fiber. \\
In the case of $k=1$, we have $f\circ \kappa_n$. Take $v$ as a vertical vector. Then, with $\pi_i: EG_n \longrightarrow G$, we have $(\pi_i)_*(v)=0$ and thus $\omega_i(v)=0$, since $\omega_i$ is a pullback through $\pi_i$. It follows then easily from the formula of the curvature $\kappa_n$ that $\kappa_n(v_1, v_2)= 0$ for two vertical vectors $v_1, v_2$. Since $n$ vertical vectors are necessary for computing $\chi_n$, we have $\chi_n(f)=0$ for $n>1$.\\
Suppose that $\bigwedge^{k-1}\kappa_n (u_1, \dots, u_{2k-2}) = 0$ if $k$ of the $u_i$'s are vertical. Then, $\bigwedge^{k}\kappa_n= \kappa_n \wedge \left(\bigwedge^{k-1}\kappa_n\right)$ is given by a sum over the set of permutations, where summands have the following form, up to the choice of the permutation and a coefficient:
$$\kappa_n(u_{\sigma(1)}, u_{\sigma(2)})\otimes \bigwedge^{k-1}\kappa_n (u_{\sigma(3)}, \dots, u_{\sigma(2k)})$$
If there is $k+1$ vertical vectors, either both $u_{\sigma(1)}$ and $u_{\sigma(2)}$ are vertical, giving $\kappa_n(u_{\sigma(1)}, u_{\sigma(2)})=0$ or there is at least $k$ vertical vectors in $\{u_{\sigma(3)}, \dots, u_{\sigma(2k)}\}$ implying $\bigwedge^{k-1}\kappa_n (u_{\sigma(3)}, \dots, u_{\sigma(2k)})=0$. Thus $\chi_n(f)=0$ for $n>k$. And so we proved the first point by induction. \\
For the second point: The statement is the same as stating that $D(\chi(f))=0$. Because of Lemma \ref{chain complex map}, we know that:
$$\left\{\int_{\fooo{}^n}d\left(f\circ \bigwedge^k\kappa_n\right)\right\}=D(\chi(f))$$
By the Chern-Weil theorem, $\phi^n_{CW}(f, \hat{\omega}_n)$, the descent of $f\circ \bigwedge^k\kappa_n$, is closed. Since $f\circ \bigwedge^k\kappa_n= \gamma_n^*(\phi^n_{CW}(f, \hat{\omega}_n))$ and since the exterior differential commutes with the pullback, we get that $f\circ \bigwedge^k\kappa_n$ is also closed, and thus $D(\chi(f))=0$. So the second statement follows. \\
For the third point, it follows directly from Lemma \ref{well defined integral}.\\
For the fourth point, it comes directly from the fact that the integration of any differential form of $\foo{}^0$, which is a point, is equal to $0$. \\
So the theorem is proved. 
\end{proof}
\begin{corollary}
It follows from this theorem that
$$\chi(f)= (0, \chi_1(f), \dots , \chi_k (f), 0, \dots, 0)\in \Omega^{2k}(EG_0)\times \dots\times \Omega^{2k-i}(EG_i)\times \dots\times \Omega^0(EG_{2k})$$
is a closed form of $F^k\tot{}^{2k}(EG_{\bullet})$.
\end{corollary}
\begin{remark}
    In Shulman's thesis \cite{shulman}, the theorem is stated as follows: The construction presented in Definition \ref{The Shulman's construction} is a well-defined cocycle (equivalent to the points 3 and 2 of the Theorem \ref{Shulman theorem}) which stops at the diagonal. A differential form $\phi\in \tot{}^k(\m_{\bullet})$ stopping at the diagonal is a form such that $\phi_n\in \Omega^n(\m_{k-n})=0$ if $n>\frac{k}{2}$. Since the differential form of Definition \ref{The Shulman's construction} is of degree $2k$, it follows that this condition is equivalent to the first point of Theorem \ref{Shulman theorem}.
\end{remark}
\section{Application on the matrix groups}
As a matter of example, we would like to present in this part the calculation of the first Pontryagin class of matrix Lie groups such as the special unitary group. We will at the end relate our result to the 2-shifted symplectic form defined in \cite{zhu} on $BG_{\bullet}$.
\subsection{The matrix groups}
To do that correctly, it might be useful to recall some facts about the matrix Lie groups. 
We set here a short recall regarding the Lie algebras of matrix Lie groups. For more details or the proofs of the statements, see \cite{matrix}. 
\begin{definition}
A \underline{matrix Lie group} is a Lie subgroup $G$ of $GL(n)$, the Lie group of invertible matrices. We allow us to consider both real and complex-valued matrices. 
\end{definition}
The main particularity of the matrix Lie groups is the possibility to represent both the elements of the group and their tangent vectors with matrices. Indeed let $A_t= (a_{i,j}(t))$ be a curve in $G$, then the corresponding tangent vector will be $A= \left(\frac{d}{dt}a_{i,j}(t)\middle|_{t=0}\right)$.

It might be interesting to know which kind of matrices are these tangent vectors, and how to describe them easily. To do that, we introduce the following operation: 
\begin{definition}
We call the \underline{matrix exponential} the following map:
$$exp: gl(n)\longrightarrow gl(n)\qquad \qquad A\mapsto \sum_{i=0}^{\infty}\frac{A^i}{i!}$$
where $gl(n)$ is the space of all the matrices of dimension $n\times n$. 
\end{definition}
So we come to the following property:
\begin{lemma}
Let $G$ be any matrix Lie group. Then its canonical Lie algebra is isomorphic to the Lie algebra given by:
$$\left\{ A\in gl(n)\middle| e^{t\cdot A}\in G \;\; \forall \; t\in \R\right\}$$
with the Lie brackets given by the commutator $[A, B]= A\cdot B - B\cdot A$. Here, $gl(n)$ denotes the set of all $n\times n $ matrices. 
\end{lemma}
\begin{proof}
See \cite{matrix}. 
\end{proof}
Let us see the following example to understand it:
\begin{example}\label{lie algebra of so(n)}
We want to obtain the Lie algebra of $\text{SO}(n)$. It is the set of all matrices with $e^{tA}\in \text{SO}(n)$ for all $t$, i.e. $(e^{tA})^T\cdot e^{tA}=e^{tA^T}\cdot e^{tA}=\text{id}$, and $\text{det}(e^{tA})=e^{\text{trace}(tA)}=1$. The last one is equivalent to $\text{trace}(tA)= t\cdot \text{trace}(A)=0 \;\forall t\in \R$ and thus $\text{trace}(A)=0$. The first one implies:
$$\frac{d}{dt}e^{tA^T}\cdot e^{tA}= (A^T+A)\cdot e^{tA^T}\cdot e^{tA}= 0= \frac{d}{dt}\text{id}$$
For $t=0$, we have $A+A^T=0$, thus $A$ is skew-symmetric. In the other direction, if $A$ is skew-symmetric, $A+A^T=0$, thus $(A^T+A)\cdot e^{tA^T}\cdot e^{tA}= 0$ and thus $e^{tA^T}\cdot e^{tA}$ is constant. Since $e^{0}\cdot e^{0}=\text{id}$, we have $e^{tA^T}\cdot e^{tA}=\text{id}$ for all $t$; furthermore, the skew-symmetric matrices have the trace equal to $0$, implying $\text{det}(e^{tA})=1$ . It follows that the Lie algebra of $\text{SO}(n)$ is the set of skew-symmetric matrices. \\
Analogously, we can see that the Lie algebra of $O(n)$ is also the set of skew-symmetric matrices, or that the Lie algebra of $\text{SU}(n)$ is the set of all anti-Hermitian matrices with trace equal to $0$.  
\end{example}

Finally, we get the two following maps for the matrix Lie groups. 
\begin{lemma}
Let $G$ be a matrix Lie group and $\fg$ its Lie algebra. Then, for $S\in G$ and $A\in \fg$, we have:
$$Ad(S)(A)= S\cdot A\cdot S^{-1}$$
\end{lemma}
\begin{proof}
Since $\left.\frac{d}{dt}e^{tA}\right|_{t=0}=A$, it follows that $e^{tA}$ is a smooth curve representing the tangent vector $A$. Thus:
$$Ad(S)(A)= \left.\frac{d}{dt}S\cdot e^{tA} \cdot S^{-1}\right|_{t=0}= \left.\frac{d}{dt} \sum_{i=0}^{\infty}\frac{S\cdot t^i\cdot A^i\cdot S^{-1}}{i!}\right|_{t=0}=\left. \frac{d}{dt} \sum_{i=0}^{\infty}\frac{(S\cdot t\cdot A\cdot S^{-1})^{i}}{i!}\right|_{t=0}=S\cdot A\cdot S^{-1}$$
\end{proof}
\begin{lemma}\label{matrix maurer cartan form}
Let $G$ be any matrix Lie group. Then, for $A$ a tangent vector at the point $S\in G$, the Maurer-Cartan form of $G$ is given by:
$$\omega_S(A)= S^{-1}\cdot A$$
\end{lemma}
\begin{proof}
Let $A_t$ be a curve representing $A$. Then, the Maurer-Cartan form sends $A$ to the tangent vector represented by the curve $S^{-1}\cdot A_t$, that is $\frac{d}{dt}S^{-1}\cdot A_t= S^{-1}\cdot A$. 
\end{proof}
\begin{remark}
The map of the lemma Lemma \ref{matrix maurer cartan form} can be denoted $S^{-1}\cdot d_S$. This notation is useful to recall that it is a differential form. 
\end{remark}

Let us now introduce the Spin group. This group is strongly related to $\text{SO}(n)$ and we shall see that we can almost consider it as a matrix Lie group since most of the properties on these last ones can be adapted to the Spin group. \\
To present the definition of the Spin group, it is necessary to present shortly the concept of Clifford algebra:
\begin{definition}
Let $V$ be a vector space. A quadratic form $Q$ on $V$ is a map from $V$ to $\R$, with:
$$Q(a\cdot v)= a^2 \cdot Q(v)\qquad \text{and}\qquad Q(v+w)- Q(v)-Q(w)\; \text{is bilinear}$$
where $v, w\in V$ and $a\in \R$.
\end{definition}
\begin{definition}
Let $V$ be a vector space with $Q$ a given quadratic form on it. Then, we define the Clifford algebra corresponding to $(V, Q)$ to be the largest associative algebra generated by each element of $V$ and $1_{Cl}$, the last one being the neutral element of the multiplication, and we require the multiplication to fulfil:
$$v\cdot v= -Q(v)\cdot 1_{Cl}$$
\end{definition}
We come now to the definition of the Spin group itself. 

\begin{definition}
We define the Spin group of a Clifford algebra as the following subset:
$$\{v_1\cdots v_{2k}| k\in \N,\;v_1,\dots , v_{2k}\in V,\; Q(v_i)=1\, \forall i \}\cup \{-1_{Cl}, 1_{Cl}\}$$
A special case of the Spin group, called the real Spin group and denoted $\text{Sp}(n)$, is the Spin group that we obtain if we set $V=\R^n$ and:
$$Q\left(\sum_{i=1}^n \lambda_i\cdot e_i\right)= \sum_{i=1}^n \lambda_i^2$$
\end{definition}
This definition can be seen more conveniently and geometrically using the following lemma:
\begin{lemma}
We have the following exact sequence:
$$0\longrightarrow \Z /2\Z \longrightarrow \text{\normalfont Sp}(n)\longrightarrow \text{\normalfont SO}(n)\longrightarrow 0$$
where the first map sends $\pm 1 \mapsto \pm 1_{Cl}$ and the second one sends:
$$v_1\cdots v_{2k}\mapsto R(v_1)\circ\dots \circ R(v_{2k})$$
where $R(v_i)$ is the matrix representing the map given by the reflection through the hyperplane perpendicular to $v_i$, while $\pm 1_{Cl}$ are sent to the identity. 
\end{lemma}
\begin{proof}
The proof of this can be found in \cite{spin}
\end{proof}
This geometric consideration is particularly useful. Indeed, we can remark that some of the differential structures of the Spin group are analogous to the ones of $\text{SO}(n)$. 
\begin{remark}\label{spin double cover}
The group $\text{Sp}(n)$ is a double covering of $\text{SO}(n)$. It implies that for an enough small neighbourhood $U$ of $\text{SO}(n)$, the preimage of $U$ through the projection $\pi: \text{Sp}(n)\longrightarrow \text{SO}(n)$ is $U\sqcup U$. It means that for $p\in \text{Sp}(n)$, $T_p \text{Sp}(n) \cong T_{\pi(p)}\text{SO}(n)$, and the last one is isomorphic to the set of skew-symmetric matrices, as seen in Example \ref{lie algebra of so(n)}. Furthermore, let $v$ be a tangent vector at the point $p$ in $\text{Sp}(n)$, corresponding to the tangent vector $A$ in $\text{SO}(n)$ at the point $\pi(p)$. Let $v(t)$ be the curve representing $v$. Then, the Maurer-Cartan form of $v$ will be the tangent vector at $e_{\text{Sp}(n)}$ given by the curve $p^{-1}\cdot v(t)$. It corresponds on $\text{SO}(n)$ to the tangent vector represented by the curve $\pi(p^{-1}\cdot v(t))=\pi(p^{-1})\cdot \pi(v(t))$, that is $\pi(p)^{-1}\cdot A_t$, representing the tangent $\pi(p)^{-1}\cdot A$. Thus the Maurer-Cartan form on $\text{Sp}(n)$ is very analogous to the one of the matrix Lie groups, since we can associate its Lie algebra and its Maurer-Cartan form to the one of $\text{SO}(n)$.
\end{remark}
\subsection{The Pontryagin polynomial}
In this subsection, we will introduce a polynomial on $\fg$, called the Pontryagin polynomial, which allows us to calculate the Pontryagin class using the Chern-Weil homomorphism. 
\begin{definition}
Let $G$ be a real Lie matrix group, with $n\times n$ matrices. Then, we can define the \underline{Pontryagin (homogeneous) invariant polynomial} on $\fg$ of order $k/2$, denoted $p_{\frac{k}{2}}$, as the only polynomial fulfilling:
$$\text{det}\left(t\cdot \text{id}-\frac{1}{2\pi}A\right)= \sum_{k=0}^{n} p_{\frac{k}{2}} (\underbrace{A, \dots, A}_{\text{k-times}}) t^{n-k} $$
\end{definition}
As seen in Remark \ref{polynomials}, $p_{\frac{k}{2}}(A, \dots, A)$ should be in $\R[x_1, \dots, x_{n^2}]^k$, where the $x_i$ are the entries of the matrix, and then $p_{\frac{k}{2}}$ can be obtained by polarization. We can check that:
\begin{lemma}\label{matrix of p1}
This definition is correct, that is, we have an invariant homogeneous polynomial on $\fg$. 
\end{lemma}
\begin{proof}
Because of the remark \ref{polynomials}, it is enough to prove that $A\mapsto p_{\frac{k}{2}}(A,\dots, A)$ is a homogeneous polynomial in $n^2$ variables of degree $k$. If we write the matrix $t\cdot \text{id}-\frac{1}{2\pi}A$, we get:
$$
\begin{pmatrix}
 t-\frac{1}{2\pi}a_{1,1}& -\frac{1}{2\pi}a_{1,2} &\cdots & \cdots & -\frac{1}{2\pi}a_{ 1,n}\\
-\frac{1}{2\pi}a_{2,1} &  t-\frac{1}{2\pi}a_{2,2} &&&\vdots\\
\vdots & & \ddots & &\vdots\\
\vdots & & & \ddots & -\frac{1}{2\pi}a_{n-1, n}\\
-\frac{1}{2\pi}a_{1, n} & \cdots & \cdots & -\frac{1}{2\pi}a_{n, n-1}&  t-\frac{1}{2\pi}a_{n,n}
\end{pmatrix}
$$
Thus the determinant will be a sum of summands of the form $\pm\prod (t-\frac{1}{2\pi} a_{i,i})\cdot \prod -\frac{1}{2\pi}a_{i,j}$, with a total of $n$ factors for each summand. It follows then easily that restricting to the summands being factors of $t^{n-k}$, we obtain a sum of the form $t^{n-k} \cdot \sum \pm \left(\frac{1}{2\pi}\right)^k(a_{i_1, j_1}\cdots a_{i_k, j_k})$. This is a well-defined element of $\R[x_1, \dots, x_{n^2}]^k$, and we can thus apply the polarization on it as in Remark \ref{polynomials}. Thus we proved that $p_{\frac{k}{2}}$ is well-defined as a homogeneous polynomial on $\fg$. Furthermore, this polynomial is invariant, since:
$$\text{det}(t\cdot \text{id}-\frac{1}{2\pi}Ad_S(A))= \text{det}(t\cdot \text{id}-\frac{1}{2\pi}S\cdot A\cdot S^{-1})=\text{det}(S\cdot(t\cdot \text{id}-\frac{1}{2\pi} A)\cdot S^{-1})= \text{det}(t\cdot \text{id}-\frac{1}{2\pi}A)$$
Thus, $p_{\frac{k}{2}}(A, \dots, A)= p_{\frac{k}{2}}(Ad_S A, \dots,Ad_S A)$. Let $f(A)= p_{\frac{k}{2}}(A, \dots, A)$ Finally, using the polarization such as the linearity of the adjoint representation, we have:

\begin{align*}
    p_{\frac{k}{2}}(Ad_S A_1, \dots,Ad_S A_k)&= \frac{1}{k!}\frac{\partial}{\partial \lambda_1}\dots \frac{\partial}{\partial \lambda_k}f( Ad_S(\lambda_1\cdot A_1)+\dots +  Ad_S(\lambda_k\cdot A_k))\\&= \frac{1}{k!}\frac{\partial}{\partial \lambda_1}\dots \frac{\partial}{\partial \lambda_k}f(Ad_S(\lambda_1\cdot A_1+\dots + \lambda_k\cdot A_k))\\
    &== \frac{1}{k!}\frac{\partial}{\partial \lambda_1}\dots \frac{\partial}{\partial \lambda_k}f(\lambda_1\cdot A_1+\dots + \lambda_k\cdot A_k)
    \\&=p_{\frac{k}{2}}(A_1, \dots, A_k)
\end{align*}
Thus $p_{\frac{k}{2}}\in I^k(\fg)$.
\end{proof}
\begin{remark}
Very often, we only consider $p_{k/2}$ for $k$ even. Indeed, if we choose $k$ odd, it can be proved that the Chern-Weil class generated by $p_{k/2}$ is 0, for any matrix Lie group and any $G$-principal bundle. For more information, see \cite{dupont}. 
\end{remark}
The definition we gave corresponds generally to the real matrix Lie groups. We also have:
\begin{remark}\label{chern class}
It is also possible to define an analogous homogeneous polynomial, which is specific to complex matrix Lie groups. It is called the \underline{Chern polynomial} and is defined as: 
$$\text{det}\left(t\cdot \text{id}-\frac{1}{2\pi i}A\right)= \sum_{k=0}^{n} c_{k} (\underbrace{A, \dots, A}_{\text{k-times}}) t^{n-k}$$
Remark that the images of the polynomial are in $\C$. The Chern polynomial of degree $2k$ does only differ from the Pontryagin one of degree $k= 2k/2$ by a coefficient $i^{2k}= (-1)^k$. See also that if we work with $\text{U}(n)$ or $\text{SU}(n)$, whose Lie algebras are respectively the Hermitian matrices ($A= -\bar{A}^T$) and the Hermitian matrices with the trace equal to zero, the values of $c_k$ are again real. Indeed:
$$\text{det}\left(t\cdot \text{id}-\frac{1}{2\pi i}A\right)= \text{det}\left(t\cdot \text{id}+\frac{1}{2\pi i}\bar{A}^T\right)=\text{det}\left(\left(\overline{t\cdot \text{id}-\frac{1}{2\pi i}A}\right)^T\right) =\overline{\text{det}\left(t\cdot \text{id}-\frac{1}{2\pi i}A\right)}$$
implying that the determinant is real for all $t$. Thus, the coefficients $c_k(A, \dots, A)$ are also real. Since the Pontryagin polynomial differs only by a factor $\pm 1$ of the Chern one, it follows that the Pontryagin polynomial gives also real results for $\text{U}(n)$ or $\text{SU}(n)$, and can be applied on them. 
\end{remark}
\begin{construction}\label{first pontryagin}
Let us see as an example, which will be useful later, the formula of $p_1$, the first Pontryagin class. We need thus to calculate the coefficient for $t^{n-2}$ of the polynomial $\text{det}(t\cdot \text{id}-\frac{1}{2\pi}A)$.  Let us denote $e_{i,j}$ the entries of the matrix $t\cdot \text{id}-\frac{1}{2\pi}A$. Then, the determinant is given by $\sum_{\sigma}\text{sign}(\sigma) \prod_{i=1}^n e_{i, \sigma(i)}$. Since we want to have the summands such that we obtain $t^{n-2}$, it implies that we should multiply at least $n-2$ elements of the diagonal. It follows thus that $\sigma$ is only a transposition of two elements or the identity. If we restrict the sum to these cases, we have:
$$\underbrace{\prod_{i=1}^n\left(t-\frac{1}{2\pi}a_{i,i}\right)}_{=\frac{1}{4\pi^2} \left(\sum_{i<j}  a_{i,i}\cdot a_{j,j}\right)\cdot t^{n-2} + \ast}- \underbrace{\sum_{i< j}\frac{1}{2\pi}a_{i,j}\cdot \frac{1}{2\pi}a_{j,i}\cdot \prod_{k\neq i, j} \left(t-\frac{1}{2\pi}a_{k,k}\right)}_{=\frac{1}{4\pi^2}\sum_{i<j} a_{i,j}\cdot a_{j,i}\cdot t^{n-2} + \ast  }$$
Where the $\ast$ means the summands that are multiplied by another power of $t$. Thus, we deduce that the polynomial $A\mapsto p_1(A, A)$ is given by:
$$\frac{1}{4\pi^2} \sum_{i<j}  a_{i,i}\cdot a_{j,j} - a_{i,j}\cdot a_{j,i}$$
Now, applying the polarization on it, we have:
\begin{align*}
    p_1(A, B)&= \frac{1}{2}\frac{\partial}{\partial \lambda_1}\frac{\partial}{\partial \lambda_2 }\frac{1}{4\pi^2} \sum_{i<j}  (\lambda_1 a_{i,i}+\lambda_2 b_{i,i})\cdot (\lambda_1 a_{j,j}+\lambda_2 b_{j,j}) - (\lambda_1 a_{i,j}+\lambda_2 b_{i,j})\cdot (\lambda_1 a_{j,i}+\lambda_2 b_{j,i})\\
    &= \frac{1}{8\pi^2} \sum_{i<j} a_{i,i}\cdot b_{j, j}+ b_{i, i}\cdot a_{j, j} - a_{i, j}\cdot b_{j,i} - a_{j, i}\cdot b_{i, j}\\
    &=\frac{1}{8\pi^2}\left( \underbrace{\sum_{i<j} a_{i,i}\cdot b_{j, j}+ b_{i, i}\cdot a_{j, j} + \sum_{i=1}^n a_{i, i}\cdot b_{i, i}}_{\text{Tr}(A)\cdot \text{Tr}(B)}-  \underbrace{\left(\sum_{i<j}a_{i, j}\cdot b_{j,i} + a_{j, i}\cdot b_{i, j}+\sum_{i=1}^n a_{i, i}\cdot b_{i, i}\right)}_{\text{Tr}(A\cdot B)}\right)\\
    &= \frac{1}{8\pi^2} (\text{Tr}(A)\cdot \text{Tr}(B) - \text{Tr}(A\cdot B)))
\end{align*}

\end{construction}

\begin{remark}
The second Chern polynomial is thus the opposite of the first Pontryagin polynomial, that is:
$$\frac{1}{8\pi^2} (\text{Tr}(A\cdot B) -\text{Tr}(A)\cdot \text{Tr}(B)  ))$$
\end{remark}
\subsection{The first Pontryagin class of the classifying space of a matrix Lie group}
We will now consider the main point of this section: 
\begin{lemma}\label{main calcul}
Let $G$ be any real matrix Lie group. Applying the Shulman's construction for the first Pontryagin class on $G$ will give:
\begin{align*}
    \chi_1=& \frac{1}{24\pi^2} \text{Tr} ( \bar{\omega}_1 \cdot \bar{\omega}_1 \cdot  \bar{\omega}_1 )\\
    \chi_2=& \frac{1}{8\pi^2}\left(\tr(\bar{\omega}_1\cdot \bar{\omega}_2) - \tr(\bar{\omega}_1)\cdot \tr(\bar{\omega}_2) \right)\\
    \chi_n=&0 \qquad\qquad \forall n \neq 1, 2 
\end{align*}
where $\bar{\omega}_i= \omega_i- \omega_0$.
\end{lemma}
\begin{proof}
As seen in the construction \ref{first pontryagin}, the first Pontryagin class is defined by the polynomial
$$p_1(A, B)= \frac{1}{8\pi^2} (\text{Tr}(A)\cdot \text{Tr}(B) - \text{Tr}(A\cdot B)))$$
Let us furthermore denote $\kappa_n$ the curvature of the canonical connection on $EG_n$, as defined in Construction \ref{connection on ngn}. For simplicity, we decompose it into three parts:
$$\kappa_n = \underbrace{\sum_{j=1}^n dt_j\wedge (\omega_j - \omega_0)}_{=: \kappa_n^1} + \underbrace{\frac{1}{2}\sum_{j=1}^n (t_j^2-t_j)\cdot[\omega_j- \omega_0\wedge\omega_j- \omega_0]}_{=: \kappa_n^2} + \underbrace{\sum_{1\leq i<j \leq n} t_i\cdot  t_j [\omega_i- \omega_0\wedge\omega_j-\omega_0]}_{=:\kappa_n^3}$$
From now on, let us denote $\omega_i-\omega_0$ by $\bar{\omega}_i$.
We have to consider:
$$\int_{\fooo{}^n}p_1\circ \bigwedge^2 \kappa_n$$
Because of Shulman's theorem, it is enough to consider it for $n=1$ and $n=2$. \\
\newmoon For the case $n=1$:\\
Let us consider $p_1 \circ \kappa_1\wedge\kappa_1$, defined on $\foo{}^1\times EG_1$. Since we will integrate this form along $\foo{}^1$, it means that we have to consider the $1$-form:
$p_1 \circ \kappa_1\wedge\kappa_1(\square, w_2, w_3, w_4)$, defined on $\foo{}^1$, for given $w_i$ which are horizontal tangent vectors of $\foo{}^1\times EG_1$. Let thus $w_1$ be a vertical vector of $\foo{}^1\times EG_1$, and let us consider $p_1 \circ \kappa_1\wedge\kappa_1(w_1, w_2, w_3, w_4)$. The part $\kappa_1\wedge\kappa_1$ can be decomposed as:
$$\kappa_1\wedge\kappa_1 = \sum_{i, j=1}^3 \kappa_1^i\wedge \kappa_1^j$$ 
\newpage
First of all, $\kappa_1^3$ is simply $0$ because the condition of the sum cannot be fulfilled for $n=1$. Secondly:
$$\kappa_1^1\wedge\kappa_1^1( w_1, w_2, w_3, w_4)= \left(dt_1\wedge\bar{\omega}_1\wedge dt_1\wedge \bar{\omega}_1\right)(w_1, w_2, w_3, w_4) $$
Recall that $dt_1(w)=0$ for a horizontal tangent vector. Since there is twice $dt_1$ and only one non-horizontal vector, it follows that each permutation for the wedge product will give $0$ (refer to the definition \ref{wedge products}). Thus $\kappa_1^1\wedge \kappa_1^1= 0$. With the same reasoning, it follows that:
$$\kappa_1^2\wedge\kappa_1^2(w_1, w_2, w_3, w_4)=  \frac{1}{4}(t_1^2-t_1)^2\bigwedge^2\cdot[\bar{\omega}_1\wedge\bar{\omega}_1](w_1, w_2, w_3, w_4) =0$$
since $[\omega_1- \omega_0\wedge\omega_1- \omega_0](v, w)=0$ if $v$ or $w$ is a vertical tangent vector. It comes from the fact that $(\omega_1- \omega_0)(v)=0 $ for a vertical vector and from the definition \ref{wedge products}. Thus, we have:
$$\kappa_1\wedge\kappa_1(w_1, w_2, w_3, w_4)= (\kappa_1^1\wedge \kappa_1^2 + \kappa^2_1\wedge \kappa_1^1)(w_1, w_2, w_3, w_4)$$
Let us now do the composition $\widetilde{p_1^1}\circ\kappa_1^1\wedge \kappa_1^2$, where $p_1^1$ is $\tr(\bullet)\cdot \tr(\bullet)$, and $\widetilde{p_1^1}$ the map induced by $p_1^1$ following the universal property of the tensor product (see the notation of Remark \ref{notation tensor poly}). Then, using the definition of the wedge product and applying $\widetilde{p^1_1}$, it will be equal to: 
$$\frac{1}{4}\sum_{\sigma\in \mathfrak{S}_{4}}\text{sign}(\sigma)\tr\left(\kappa_1^1(w_{\sigma(1)}, w_{\sigma(2)})\right)\cdot \tr\left(\kappa_1^2(w_{\sigma(3)}, w_{\sigma(4)})\right)$$
The second part of the equation can be written as: 
\begin{equation}\label{eq: first part 0}
    \begin{aligned}
    \tr\left(\kappa_1^2(w_{\sigma(3)}, w_{\sigma(4)})\right)&= \frac{1}{2} (t_1^2-t_1)\cdot \tr\left([\bar{\omega}_1\wedge \bar{\omega}_1](w_{\sigma(3)}, w_{\sigma(4)})\right)\\
    &\underset{\ref{wedge products}}{=} (t^2_1 -t_1)\cdot \tr (\bar{\omega}_1(w_{\sigma(3)})\cdot \bar{\omega}_1(w_{\sigma(4)}) - \bar{\omega}_1(w_{\sigma(4)})\cdot \bar{\omega}_1(w_{\sigma(3)}))
\end{aligned}
\end{equation}
But it can be easily proved that $\tr(A\cdot B)= \tr( B\cdot A)$ for $A,B$ any square matrices. Since each $\bar{\omega}_i(w_{\sigma(j)})$ is a matrix, the equation (\ref{eq: first part 0}) gives $0$, and since we can do the same for $\widetilde{p_1^1}\circ\kappa_1^2\wedge \kappa_1^1$, it implies that $\widetilde{p_1^1}\circ \kappa_1\wedge \kappa_1= 0$. \\
Let us now consider $\widetilde{p_1^2}\circ\kappa_1^1\wedge \kappa_1^2$, where $p_1^2$ is the map $\tr(\bullet\cdot \bullet)$. We have thus the following, using the definition of the wedge product:
\begin{equation}\label{eq: random eq}
    \frac{1}{4}\sum_{\sigma\in \mathfrak{S}_4} \text{sign}(\sigma)\tr\left(\kappa_1^1(w_{\sigma(1)}, w_{\sigma(2)})\cdot \kappa_1^2(w_{\sigma(3)},w_{\sigma(4)})\right)
\end{equation}
\newpage
Let us recall that $[\omega_1- \omega_0\wedge\omega_1- \omega_0](v, w)=0$ if $v$ or $w$ is a vertical tangent vector. Thus, if the summand is non-zero, $w_1$ should be one of the entries of $\kappa_1^1$, and thus $\sigma(1)=1$ or $\sigma(2)=1$. We can thus re-write the formula (\ref{eq: random eq}) in:
\begin{align*}
    \frac{1}{4}\sum_{\sigma\in \mathfrak{S}_3} \text{sign}(\sigma)\left(\tr\left(\kappa_1^1(w_1, w_{\sigma(1)+1})\cdot \kappa_1^2(w_{\sigma(2)+1},w_{\sigma(3)+1})\right) -  \tr\left(\kappa_1^1( w_{\sigma(1)+1}, w_1)\cdot \kappa_1^2(w_{\sigma(2)+1},w_{\sigma(3)+1})\right)\right)\\
    = \frac{1}{2}\sum_{\sigma\in \mathfrak{S}_3} \text{sign}(\sigma)\tr\left(\kappa_1^1(w_1, w_{\sigma(1)+1})\cdot \kappa_1^2(w_{\sigma(2)+1},w_{\sigma(3)+1})\right)
\end{align*}
where the last equation is true because of $\kappa_1^1( w_{\sigma(1)+1}, w_1)=-\kappa_1^1(w_1, w_{\sigma(1)+1})$ and of the linearity of $\tr$. Furthermore, if we write $\kappa_1^1$ and then $\kappa^2_1$ in equation, we have:
\begin{align*}
    \frac{1}{2}\sum_{\sigma\in \mathfrak{S}_3} \text{sign}(\sigma)\tr\left(\left(dt_1(w_1)\cdot \bar{\omega}_1(w_{\sigma(1)+1}) - \underbrace{dt_1(w_{\sigma(1)+1})}_{=0 \; \text{since hor. vector}}\cdot \bar{\omega}_1(w_1))\right)\cdot \kappa_1^2(w_{\sigma(2)+1},w_{\sigma(3)+1})\right)\\ = \frac{1}{4}\sum_{\sigma\in \mathfrak{S}_3} \text{sign}(\sigma)(t_1^2-t_1)dt_1(w_1)\cdot\tr\left( \bar{\omega}_1(w_{\sigma(1)+1})  \cdot [\bar{\omega}_1\wedge \bar{\omega}_1](w_{\sigma(2)+1},w_{\sigma(3)+1})\right)
\end{align*}
Let us now integrate this formula along $\foo{}^1$. It gives:
\begin{align*}
    \left(\frac{1}{4}\sum_{\sigma\in \mathfrak{S}_3} \text{sign}(\sigma)\cdot\tr\left( \bar{\omega}_1(w_{\sigma(1)+1})  \cdot [\bar{\omega}_1\wedge \bar{\omega}_1](w_{\sigma(2)+1},w_{\sigma(3)+1})\right)\right) \int_{\fooo{}^1}(t_1^2-t_1) dt_1\\
    =\left(\frac{1}{4}\sum_{\sigma\in \mathfrak{S}_3} \text{sign}(\sigma)\cdot\tr\left( \bar{\omega}_1(w_{\sigma(1)+1})  \cdot [\bar{\omega}_1\wedge \bar{\omega}_1](w_{\sigma(2)+1},w_{\sigma(3)+1})\right)\right) \left[\frac{1}{3}t_1^3-\frac{1}{2}t^2\right]_0^1\\
    = -\frac{1}{24}\sum_{\sigma\in \mathfrak{S}_3} \text{sign}(\sigma)\cdot\tr\left( \bar{\omega}_1(w_{\sigma(1)+1})  \cdot [\bar{\omega}_1\wedge \bar{\omega}_1](w_{\sigma(2)+1},w_{\sigma(3)+1})\right)
\end{align*}
If we now write $[\bar{\omega}_1\wedge \bar{\omega}_1](v,w)$ in an equation, using the definition of the wedge product, we obtain $[\bar{\omega}_1(v), \bar{\omega}_1(w)] - [\bar{\omega}_1(w), \bar{\omega}_1(v)]$, and developing the commutator: $2\bar{\omega}_1(v)\cdot \bar{\omega}_1(w)- 2\bar{\omega}_1(w)\cdot \bar{\omega}_1(v)$. Thus, we get:
\begin{align*}
    -\frac{1}{12}\sum_{\sigma\in \mathfrak{S}_3}& \text{sign}(\sigma)\cdot\tr\left( \bar{\omega}_1(w_{\sigma(1)+1})  \cdot \left(\bar{\omega}_1 (w_{\sigma(2)+1})\cdot \bar{\omega}_1 (w_{\sigma(3)+1}) - \bar{\omega}_1 (w_{\sigma(3)+1})\cdot \bar{\omega}_1 (w_{\sigma(2)+1}\right)\right)\\
    &=-\frac{1}{12}\sum_{\sigma\in \mathfrak{S}_3} \text{sign}(\sigma)\cdot\tr\left( \bar{\omega}_1(w_{\sigma(1)+1})  \cdot \bar{\omega}_1 (w_{\sigma(2)+1})\cdot \bar{\omega}_1 (w_{\sigma(3)+1})\right) - \\& \qquad\qquad\qquad\left. \frac{1}{12}\cdot \underbrace{\sum_{\sigma\in \mathfrak{S}_3}  -\text{sign}(\sigma)\cdot\tr\left( \bar{\omega}_1(w_{\sigma(1)+1})  \cdot \bar{\omega}_1 (w_{\sigma(3)+1})\cdot \bar{\omega}_1 (w_{\sigma(2)+1})\right)}_{\substack{=\sum_{\sigma\circ \tau\in \mathfrak{S}_3}  \text{sign}(\sigma\circ \tau)\cdot\tr\left( \bar{\omega}_1(w_{\sigma\circ \tau(1)+1})  \cdot \bar{\omega}_1 (w_{\sigma\circ \tau(2)+1})\cdot \bar{\omega}_1 (w_{\sigma\circ \tau(3)+1})\right),\\ \;\text{for}\; \tau(1,2,3)=(1,3,2)}} \right.\\
    &=-\frac{1}{6}\sum_{\sigma\in \mathfrak{S}_3} \text{sign}(\sigma)\cdot\tr\left( \bar{\omega}_1(w_{\sigma(1)+1})  \cdot \bar{\omega}_1 (w_{\sigma(2)+1})\cdot \bar{\omega}_1 (w_{\sigma(3)+1})\right)\\
     &=-\frac{1}{6}\tr\circ\text{Mult} \left(\sum_{\sigma\in \mathfrak{S}_3} \text{sign}(\sigma)\cdot \bar{\omega}_1(w_{\sigma(1)+1})  \otimes \bar{\omega}_1 (w_{\sigma(2)+1})\otimes \bar{\omega}_1 (w_{\sigma(3)+1})\right)\\
     &=-\frac{1}{6}\tr\circ\text{Mult} \left(\sum_{\substack{\sigma\in \mathfrak{S}_3\\ \sigma(2)<\sigma(3)}} \text{sign}(\sigma)\cdot \bar{\omega}_1(w_{\sigma(1)+1})  \otimes  \Bigl(\bar{\omega}_1 (w_{\sigma(2)+1})\otimes \bar{\omega}_1 (w_{\sigma(3)+1})-\right.\\
      &\qquad\qquad\qquad\qquad\qquad\qquad\qquad\qquad\qquad\qquad\qquad\qquad\qquad \bar{\omega}_1 (w_{\sigma(3)+1})\otimes \bar{\omega}_1 (w_{\sigma(2)+1})\Bigr) \Biggr)\\
     &=-\frac{1}{6}\tr\circ\text{Mult} \left(\sum_{\substack{\sigma\in \mathfrak{S}_3\\ \sigma(2)<\sigma(3)}}\text{sign}(\sigma)\cdot \bar{\omega}_1(w_{\sigma(1)+1})  \otimes \bar{\omega}_1\wedge \bar{\omega}_1 (w_{\sigma(2)+1},w_{\sigma(3)+1} )\right) \\
     &=-\frac{1}{6}\tr\circ\text{Mult} \left(\frac{1}{2}\sum_{\sigma\in \mathfrak{S}_3}\text{sign}(\sigma)\cdot \bar{\omega}_1(w_{\sigma(1)+1})  \otimes \bar{\omega}_1\wedge \bar{\omega}_1 (w_{\sigma(2)+1},w_{\sigma(3)+1} )\right)\\
    &=-\frac{1}{6}\tr\left(\text{Mult}\bigwedge^3\bar{\omega}_1 (w_2, w_3, w_4)\right)
\end{align*}
Analogously, we can do the same calculation for $\Tilde{p^2_1}\circ \kappa_1^2\wedge\kappa_1^1$ and obtain the same result (this comes from the symmetry of $p_1^2$). We get thus:
$$\int_{\fooo{}^1}\frac{1}{8\pi^2}(\tr(\bullet_1)\cdot \tr(\bullet_2)-\tr(\bullet_1\cdot \bullet_2))\circ \bigwedge^2\kappa_1= \frac{1}{8\pi^2}\cdot 2\cdot \frac{1}{6}\tr\circ \text{Mult}\left(\bigwedge^3\bar{\omega}_1\right)$$
We can thus write, under the notation of Remark \ref{notation of polynomial}:

$$\frac{1}{24\pi^2} \text{Tr} ( \bar{\omega}_1 \cdot \bar{\omega}_1 \cdot  \bar{\omega}_1 ) $$
\newmoon For the case $n=2$:
As before, we consider $\kappa_2\wedge\kappa_2$, and it can be decomposed in:
$$\kappa_2\wedge\kappa_2 = \sum_{i, j=1}^3 \kappa_2^i\wedge \kappa_2^j$$
We will this time integrate along $\foo{}^2$. It means that we should consider $p_1\circ \kappa_2\wedge\kappa_2(w_1, w_2, w_3, w_4)$ with $w_1, w_2$ vertical tangent vectors and $w_3, w_4$ horizontal tangent vectors. Since $[\bar{\omega}_i\wedge\bar{\omega}_j](v,w)=0$ as soon as $v$ or $w$ is vertical, it follows that for the vectors $w_1, w_2, w_3, w_4$, the forms $\kappa_2^2\wedge\kappa_2^2$, $\kappa_2^3\wedge\kappa_2^2$, $\kappa_2^2\wedge\kappa_2^3$ and $\kappa_2^3\wedge\kappa_2^3$ will simply give $0$. Analogously, $\kappa_2^1\wedge\kappa_2^2$ will be a linear combination of: 
$$dt_i\wedge \bar{\omega}_i\wedge [\bar{\omega}_j\wedge \bar{\omega}_j]$$
With the definition of wedge product (Definition \ref{wedge products}), it can be easily seen that this form will give $0$ if it receives two vertical vectors among the four entries. The same reasoning can be held for $\kappa_2^1\wedge\kappa_2^3$, $\kappa_2^3\wedge\kappa_2^1$ and $\kappa_2^2\wedge\kappa_2^1$. Thus, we can reduce our study to the case of $\kappa_2^1\wedge \kappa_2^1$, which can be written as:
$$\sum_{i, j=1}^2(dt_i\wedge \bar{\omega}_i)\wedge (dt_j \wedge \bar{\omega}_j)$$
If we develop it, we get:
\begin{align*}
    \sum_{i, j=1}^2(dt_i\wedge \bar{\omega}_i)\wedge &(dt_j \wedge \bar{\omega}_j)(w_1, w_2, w_3, w_4)\\
    &= \frac{1}{4}\sum_{i, j=1}^2 \sum_{\sigma\in \mathfrak{S}_4}\text{sign}(\sigma) (dt_i\wedge \bar{\omega}_i)(w_{\sigma(1)}, w_{\sigma(2)}) \otimes (dt_j\wedge \bar{\omega}_j)(w_{\sigma(3)}, w_{\sigma(4)})\\
    &=\frac{1}{4}\sum_{i, j=1}^2 \sum_{\sigma\in \mathfrak{S}_4} \text{sign}(\sigma)\left(dt_i(w_{\sigma(1)}) \cdot \bar{\omega}_i(w_{\sigma(2)}) - dt_i(w_{\sigma(2)}) \cdot \bar{\omega}_i(w_{\sigma(1)})\right) \otimes\\ &\qquad\qquad\qquad\qquad\qquad\qquad\qquad\left.\left(dt_j(w_{\sigma(3)}) \cdot \bar{\omega}_j(w_{\sigma(4)}) - dt_j(w_{\sigma(4)}) \cdot \bar{\omega}_j(w_{\sigma(3)})\right)\right.
\end{align*}
Remark furthermore that each vertical vector $v$ can be written $(v^1, v^2)$, for $v=v^1\cdot \frac{\partial}{\partial t_1}+ v^2\cdot \frac{\partial}{\partial t_2}$. Let $v$ and $w$ be two vertical vectors. Then, either $v^1= 0$ (then, $dt_1\wedge \bar{\omega}_1(v, u)= 0$ for any $u$, since $dt_1(v)=0$, and $\bar{\omega}_1(v)=0$), or $w - \frac{w^1}{v^1}\cdot v$ is well-defined, and $dt_1\wedge \bar{\omega}_1(w - \frac{w^1}{v^1}\cdot v, u)= 0$ for the same reason as before. Thus, if $w_1^1\neq 0$:
$$(dt_1\wedge \bar{\omega}_1)\wedge (dt_1 \wedge \bar{\omega}_1)(w_1, -\frac{w_2^1}{w_1^1}w_1, w_3, w_4)=0$$
since the two first vectors are colinear, and:
$$(dt_1\wedge \bar{\omega}_1)\wedge (dt_1 \wedge \bar{\omega}_1)(w_1, w_2, w_3, w_4)=(dt_1\wedge \bar{\omega}_1)\wedge (dt_1 \wedge \bar{\omega}_1)(w_1, w_2-\frac{w_2^1}{w_1^1}w_1, w_3, w_4)=0$$
If $w_1^1=0$, we follow directly that $(dt_1\wedge \bar{\omega}_1)\wedge (dt_1 \wedge \bar{\omega}_1)(w_1, w_2, w_3, w_4)=0$. Thus, in general, $(dt_1\wedge \bar{\omega}_1)\wedge (dt_1 \wedge \bar{\omega}_1)(w_1, w_2, w_3, w_4)=0$ for $w_1, w_2$ vertical tangent vectors. Analogously, we show that $(dt_2\wedge \bar{\omega}_2)\wedge (dt_2 \wedge \bar{\omega}_2)(w_1, w_2, w_3, w_4)=0$. \\
So consider specifically $(dt_1\wedge\bar{\omega}_1)\wedge(dt_2\wedge\bar{\omega}_2)$: 
\begin{align*}
    &\frac{1}{4}\sum_{\sigma\in \mathfrak{S}_4}\text{sign}(\sigma) \left(\underbrace{dt_1(w_{\sigma(1)}) \cdot \bar{\omega}_1(w_{\sigma(2)})}_{\barroman{I}} - \underbrace{dt_1(w_{\sigma(2)}) \cdot \bar{\omega}_1(w_{\sigma(1)})}_{\barroman{II}}\right) \otimes\\
    &\qquad\qquad\qquad\qquad\qquad\qquad\qquad\left.\left(\underbrace{dt_2(w_{\sigma(3)}) \cdot \bar{\omega}_2(w_{\sigma(4)})}_{\barroman{III}} - \underbrace{dt_2(w_{\sigma(4)}) \cdot \bar{\omega}_2(w_{\sigma(3)})}_{\barroman{IV}}\right)\right.\\
    &=\frac{1}{4}\sum_{\sigma\in \mathfrak{S}_4}\text{sign}(\sigma) \barroman{I}\otimes \barroman{III} -\frac{1}{4}\sum_{\sigma\in \mathfrak{S}_4}\text{sign}(\sigma) \barroman{I}\otimes \barroman{IV} - \frac{1}{4}\sum_{\sigma\in \mathfrak{S}_4}\text{sign}(\sigma) \barroman{II}\otimes \barroman{III}+ \frac{1}{4}\sum_{\sigma\in \mathfrak{S}_4}\text{sign}(\sigma) \barroman{II}\otimes \barroman{IV}
\end{align*}
Remark that, for the permutation $\tau$: $(1,2,3,4)\mapsto(1,2,4,3)$, we have $-dt_2(w_{\sigma(4)}) \cdot \bar{\omega}_2(w_{\sigma(3)})= \text{sign}(\tau) dt_2(w_{\sigma\circ \tau(3)}) \cdot \bar{\omega}_2(w_{\sigma\circ \tau(4)})$, and thus:
$$-\frac{1}{4}\sum_{\sigma\in \mathfrak{S}_4}\text{sign}(\sigma) \barroman{I}\otimes \barroman{IV}= \frac{1}{4}\sum_{\sigma\circ \tau\in \mathfrak{S}_4}\text{sign}(\sigma\circ \tau) \barroman{I}\otimes \barroman{III}= \frac{1}{4}\sum_{\sigma\in \mathfrak{S}_4}\text{sign}(\sigma) \barroman{I}\otimes \barroman{III}$$
We can prove similarly that $\barroman{I}\otimes \barroman{III}$ is equal to $-\barroman{II}\otimes \barroman{III}$ and $\barroman{II}\otimes \barroman{IV}$. Thus, the whole sum is equal to:
$$\sum_{\sigma\in \mathfrak{S}_4}\text{sign}(\sigma) \barroman{I}\otimes \barroman{III}= \sum_{\sigma\in \mathfrak{S}_4}\text{sign}(\sigma)dt_1(w_{\sigma(1)}) \cdot \bar{\omega}_1(w_{\sigma(2)})\otimes dt_2(w_{\sigma(3)}) \cdot \bar{\omega}_2(w_{\sigma(4)})$$
If we apply the same path of calculation for $(dt_2\wedge\bar{\omega}_2)\wedge (dt_1\wedge\bar{\omega}_1)$, we obtain:
\begin{align*}
    \sum_{\sigma\in \mathfrak{S}_4}\text{sign}(\sigma)dt_2&(w_{\sigma(1)}) \cdot \bar{\omega}_2(w_{\sigma(2)})\otimes dt_1(w_{\sigma(3)}) \cdot \bar{\omega}_1(w_{\sigma(4)})\\
    &=\sum_{\sigma\in \mathfrak{S}_4}\text{sign}(\sigma)dt_2(w_{\sigma(3)}) \cdot \bar{\omega}_2(w_{\sigma(4)})\otimes dt_1(w_{\sigma(1)}) \cdot \bar{\omega}_1(w_{\sigma(2)})
\end{align*}
Since the both $p^1_1$ and $p^2_1$ are symmetric, it follows that $\tilde{p^i_1}\circ (dt_2\wedge\bar{\omega}_2)\wedge (dt_1\wedge\bar{\omega}_1) = \tilde{p^i_1}\circ (dt_1\wedge\bar{\omega}_1)\wedge (dt_2\wedge\bar{\omega}_2)$ for $i=1,2$. Let $A$ denote either $\tilde{p^1_1}$ and $\tilde{p^2_1}$. Then:

\begin{align*}
    A\circ \kappa_2\wedge\kappa_2(w_1, w_2, w_3, w_4)&=2\cdot \sum_{\sigma\in \mathfrak{S}_4}\text{sign}(\sigma)A \left(dt_1(w_{\sigma(1)}) \cdot \bar{\omega}_1(w_{\sigma(2)})\otimes dt_2(w_{\sigma(3)})\cdot \bar{\omega}_2(w_{\sigma(4)})\right)\\
   &= 2\cdot \sum_{\sigma\in \mathfrak{S}_4}\text{sign}(\sigma)dt_1(w_{\sigma(1)}) \cdot  dt_2(w_{\sigma(3)}) \cdot A\left(\bar{\omega}_1(w_{\sigma(2)})\otimes \bar{\omega}_2(w_{\sigma(4)})\right)
\end{align*}
Thus we get:
\begin{align*}
   &-2\cdot\sum_{\sigma\in \mathfrak{S}_4}\text{sign}(\sigma)dt_1(w_{\sigma(1)})\cdot dt_2(w_{\sigma(2)}) \cdot A(\Bar{\omega}_1 (w_{\sigma(3)})\otimes \Bar{\omega}_2 (w_{\sigma(4)}))\\
    &=-\sum_{\sigma\in \mathfrak{S}_4}\text{sign}(\sigma)\left(dt_1(w_{\sigma(1)})\cdot dt_2(w_{\sigma(2)}) \cdot A(\Bar{\omega}_1 (w_{\sigma(3)})\otimes \Bar{\omega}_2 (w_{\sigma(4)})) \right.\\ &\qquad \qquad \qquad\qquad \qquad \qquad\left.- dt_1(w_{\sigma(2)})\cdot dt_2(w_{\sigma(1)}) \cdot A(\Bar{\omega}_1 (w_{\sigma(3)})\otimes \Bar{\omega}_2 (w_{\sigma(4)}))\right)\\
    &= -\sum_{\sigma\in \mathfrak{S}_4}\text{sign}(\sigma) dt_1\wedge dt_2(w_{\sigma(1)}, w_{\sigma(2)}) \cdot A(\Bar{\omega}_1 (w_{\sigma(3)})\otimes \Bar{\omega}_2 (w_{\sigma(4)}))
\end{align*}

Remark that $dt_1\wedge dt_2(v, w)=0$ if $v$ or $w$ is a horizontal tangent vector. Thus, the non-zero summands are those such that $\sigma(1)=1$ and $\sigma(2)=2$ or those such that $\sigma(1)=2$ and $\sigma(2)=1$. It gives us:
\begin{align*}  
-\sum_{\sigma\in \mathfrak{S}_2}\text{sign}(\sigma) dt_1\wedge dt_2(w_1, w_2) \cdot &A(\Bar{\omega}_1 (w_{\sigma(1)+2}), \Bar{\omega}_2 (w_{\sigma(2)+2}))+\\& \sum_{\sigma\in \mathfrak{S}_2}\text{sign}(\sigma) dt_1\wedge dt_2(w_2, w_1) \cdot A(\Bar{\omega}_1 (w_{\sigma(1)+2}), \Bar{\omega}_2 (w_{\sigma(2)+2}))
\end{align*}
which is equal to:
$$-2\cdot \sum_{\sigma\in \mathfrak{S}_2}\text{sign}(\sigma) dt_1\wedge dt_2(w_1, w_2) \cdot A(\Bar{\omega}_1 (w_{\sigma(1)+2}), \Bar{\omega}_2 (w_{\sigma(2)+2}))$$
We can now integrate it along $\foo{}^2$. We have then:
$$-2\cdot\sum_{\sigma\in \mathfrak{S}_2}\text{sign}(\sigma) \cdot A(\Bar{\omega}_1 (w_{\sigma(1)+2}), \Bar{\omega}_2 (w_{\sigma(2)+2}))\cdot \int_{\fooo{}^2}dt_1\wedge dt_2 $$
with:
$$\int_{\fooo{}^2}dt_1\;\wedge dt_2= \int_0^1\left(\int_0^{1-t_1} 1 dt_2\right)dt_1= \int_0^1 [t]^{1-t_1}_0 dt_1= \int_0^1 1-t_1 dt_1= \left[t_1- \frac{1}{2}\cdot t_1^2\right]_0^1= \frac{1}{2}$$
We get finally:
\begin{align*}
    -\sum_{\sigma\in \mathfrak{S}_2}\text{sign}(\sigma) \cdot A(\Bar{\omega}_1 (w_{\sigma(1)+2}), \Bar{\omega}_2 (w_{\sigma(2)+2})) &=-A\left(\sum_{\sigma\in \mathfrak{S}_2}\text{sign}(\sigma) \cdot \bar{\omega}_1(w_{\sigma(1)+2}) \otimes \bar{\omega}_2(w_{\sigma(2)+2})\right)\\
    &= -A \circ (\bar{\omega}_1\wedge \bar{\omega}_2) (w_3, w_4)\\
    &= -A(\bar{\omega}_1,\bar{\omega}_2)(w_3, w_4)
\end{align*}
Thus:
$$\int_{\fooo{}^2}\widetilde{p_1^i}\circ \kappa_2\wedge \kappa_2= - p_1^i (\bar{\omega}_1,\bar{\omega}_2) $$
And finally:
$$\int_{\fooo{}^2}{p}_1\circ \kappa_2\wedge \kappa_2= \frac{1}{8\pi^2}\left((\tr(\bar{\omega}_1\cdot \bar{\omega}_2)) - (\tr(\bar{\omega}_1)\cdot \tr(\bar{\omega}_2)) \right)$$

\end{proof}
We can extend this result to some other groups: 
\begin{corollary}
The same result can be applied to $\text{U}(n)$ and $\text{SU}(n)$.
\end{corollary}
\begin{proof}
Follows directly from Remark \ref{chern class}. 
\end{proof}

\begin{corollary}
The same result can be applied to $\text{Sp}(n)$. 
\end{corollary}
\begin{proof}
Follows directly from Remark \ref{spin double cover}. 
\end{proof}

Let us recall that the last calculation gives a differential form on $EG_{\bullet}$. We need descend it to $BG_{\bullet}$.
\begin{lemma}
The characteristic class on $BG_{\bullet}$ given by $f\in I^k(\fg)$ can be described as:

$$\left[\left(0,\frac{1}{24\pi^2}\text{Tr}(\omega^l \cdot\omega^l  \cdot\omega^l ),  \frac{1}{8\pi^2}\left(\tr(d_2^*(\omega^l\circ\text{inv}))\cdot \tr(d_0^*\omega^l)- \tr(d_2^*(\omega^l\circ\text{inv})\cdot d_0^*\omega^l )\right), 0, \dots \right)\right]$$
where the brackets means the cohomology class, $\omega^l$ the Maurer-Cartan form and where $inv$ is the map sending the tangent vector $X$ over the point $A\in G$ to
$$\left.\frac{d}{dt}exp(t\cdot X)^{-1}\right|_{t=0}$$
with $exp$ denoting the matrix exponential. This is a tangent vector over the point $A^{-1}$.
\end{lemma}
\begin{proof}
In the following let $exp(t\cdot X)$ denotes the representating curve obtained with the matrix exponential.\\
For the first part: We have $\frac{1}{24\pi^2}\text{Tr}(\bar{\omega}_1\cdot\bar{\omega}_1\cdot\bar{\omega}_1 )$ on $\Omega^3(EG_1)$ and we want to find the unique preimage of it through $\gamma^*$ in $\Omega^3(BG_1)$ ($\gamma^*$ is injective since $\gamma$ is a submersion). Let $A\in BG_1\cong G$ and $X\in T_A G$ in matrix form. Then, a preimage of $X$ through $\gamma_*$ is given by $(0, X)$. Then, $\bar{\omega}_1(0,X)$ is given by:
$$\omega_1(0,X)-\omega_0(0,X)=A^{-1}\cdot X$$
where the right side is the left Maurer-Cartan form, as the one defined in Definition \ref{def of maurer cartan}. Thus, the whole form is given by:
$$\frac{1}{24\pi^2}\text{Tr}(\omega^l \cdot\omega^l  \cdot\omega^l )$$
For the second point: We have $\frac{1}{8\pi^2}\left((\tr(\bar{\omega}_1\cdot \bar{\omega}_2)) - (\tr(\bar{\omega}_1)\cdot \tr(\bar{\omega}_2)) \right)$. Let $(X, Y)$ be a tangent vector of $BG_2$ at the point $(A, B)$, which is represented by the curve $(exp(t\cdot X), exp(t\cdot Y))$. Then, one possible preimage of this curve is $((exp(t\cdot X)^{-1}, \text{id}, (exp(t\cdot Y))$, which represents the tangent vector $(X', 0, Y)$ over the point $(A^{-1}, \text{id}, B)$, where $X'\in T_{A^{-1}}G$ corresponds to $exp(t\cdot X)^{-1}$. Then, $\omega_1-\omega_0 (X', 0, Y)$ is equal on $B G_2$ to $\omega^l(0)- \omega^l(X')= -d_2^*(\omega^l\circ\text{inv})(X, Y)$. Furthermore, $\omega_2-\omega_0 (X', 0, Y)$ is equal on $BG_2$ to $\omega^l(Y)- \omega^l(X')= d_0^*\omega^l-d_2^*(\omega^l\circ\text{inv})(X, Y)$ We obtain thus:  
\begin{align*}
   \frac{1}{8\pi^2}&\left(\tr(-d_2^*(\omega^l\circ\text{inv})\cdot (d_0^*\omega^l- d_2^*(\omega^l\circ\text{inv}))) - \tr(-d_2^*(\omega^l\circ\text{inv}))\cdot \tr(d_0^*\omega^l- d_2^*(\omega^l\circ\text{inv})) \right)\\
   &=\frac{1}{8\pi^2}\left(\tr(d_2^*(\omega^l\circ\text{inv})\cdot d_2^*(\omega^l\circ\text{inv})) - \tr(d_2^*(\omega^l\circ\text{inv})\cdot d_0^*\omega^l) \right)\\
   &\qquad\qquad \qquad+ \frac{1}{8\pi^2} \left(\tr(d_2^*(\omega^l\circ\text{inv}))\cdot \tr(d_0^*\omega^l)-\tr(d_2^*(\omega^l\circ\text{inv}))\cdot \tr( d_2^*(\omega^l\circ\text{inv}))  \right)
\end{align*}
Recall that by definition of the notation, $\tr(d_2^*(\omega^l\circ\text{inv})\cdot d_2^*(\omega^l\circ\text{inv}))= \tr(*\cdot *)\circ (d_2^*(\omega^l\circ\text{inv}))\wedge (d_2^*(\omega^l\circ\text{inv}))$. It gives us, for $v$ and $w$ tangent vectors:
\begin{align*}
    \tr(d_2^*(\omega^l&\circ\text{inv})\cdot d_2^*(\omega^l\circ\text{inv}))(v, w)\\&= \tr(d_2^*(\omega^l\circ\text{inv})(v)\cdot d_2^*(\omega^l\circ\text{inv})(w)) - \tr(d_2^*(\omega^l\circ\text{inv})(w)\cdot d_2^*(\omega^l\circ\text{inv})(v))=0
\end{align*}
because of the symmetry of $\tr(*\cdot *)$. By the same way, it follows that $\tr(d_2^*(\omega^l\circ\text{inv}))\cdot \tr( d_2^*(\omega^l\circ\text{inv}))=0$. So we proved the statement. 
\end{proof}
In the case of unitary matrices, we can adapt this formula, especially we can remove the inelegant map $\text{inv}$:
\begin{lemma}\label{descend for unitary}
    Let $G$ be a matrix Lie group which is a subgroup of the unitary matrices, such as $U(n)$, $SU(n)$, $O(n)$ or $SO(n)$. Then, the map $\omega^l\circ \text{inv}: TG \longrightarrow \fg$ is equal to $-\omega^r$, with $\omega^r$ being the right Maurer-Cartan form, that is, for $X\in T_A G$, $\omega^r(X)=(\mathcal{R}_{A^{-1}})_*X$, which is equal to $X\cdot A^{-1}$, with a proof analogous to Lemma \ref{matrix maurer cartan form}.  
\end{lemma}
\begin{proof}
    Let $X$ be a tangent vector of $G$ at the point $A$. Then, $X$ is represented by the curve $exp(t\cdot X)$. Then, $\omega^l\circ \text{inv}(X)$ has the curve $A\cdot exp(t\cdot X)^{-1}= (exp(t\cdot X)\cdot A^{-1})^{-1}$. Since we are working with unitary matrices, we can write the last as $(exp(t\cdot X)\cdot A^{-1})^{*}$ where $*$ means the conjugate transposition of the matrix. Then, we can see that:
    $$\left.\frac{d}{dt}(exp(t\cdot X)\cdot A^{-1})^{*}\right|_{t=0}= \left(\left.\frac{d}{dt}(exp(t\cdot X)\cdot A^{-1})\right|_{t=0}\right)^{*}= (\omega^r(X))^*$$
    since the derivation along $t$ commutes with the transposition (the transposition only changes the place of the coordinates and $\frac{d}{dt}$ is applied coordinate-wise), and with the conjugation ($\text{cong}(\frac{d}{dt}(a(t)+i\cdot b(t)))= \frac{d}{dt}(a(t)-i\cdot b(t))$). Since $\omega^r(X)$ is a skew-hermitian matrix, we can write:
    $$(\omega^r(X))^*= - \omega^r(X)$$   
    So we obtain the statement.
\end{proof}
\begin{corollary}\label{pontryagine for comparison}
    In the case of a matrix Lie group $G$ being a subset of $U(n)$, we can re-write the characteristic class corresponding to the first Pontryagin polynomial as:
    $$\left[\left(0,\frac{1}{24\pi^2}\text{Tr}(\omega^l \cdot\omega^l  \cdot\omega^l ),  \frac{1}{8\pi^2}\left( \tr(d_2^*\omega^r\cdot d_0^*\omega^l) - \tr(d_2^*\omega^r)\cdot \tr(d_0^*\omega^l) \right), 0, \dots \right)\right]$$
    In the specific case of $O(n)$, $SO(n)$ or $SU(n)$, we can even reduce it to:
    $$\left[\left(0,\frac{1}{24\pi^2}\text{Tr}(\omega^l \cdot\omega^l  \cdot\omega^l ), \frac{1}{8\pi^2}\tr(d_2^*\omega^r\cdot d_0^*\omega^l), 0, \dots \right)\right]$$
\end{corollary}
\begin{proof}
    The first part follows directly from the lemma \ref{descend for unitary}. The second part follows directly from the first one, and from the fact that the Lie algebras of these Lie groups contain only matrices whose trace is equal to $0$.
\end{proof}
\begin{corollary}
    Using the association made in Remark \ref{spin double cover} between $\text{Sp}(n)$ and $SO(n)$, we obtain the same formula of the characteristic class for $\text{Sp}(n)$. 
\end{corollary}
\begin{proof}
    Follows from Remark \ref{spin double cover}. The fact that the matrix exponential is not defined on $\text{Spin}(n)$ is not a problem, since we simply used it in the previous proofs to obtain a curve representing the tangent vector. Any other (non-canonical) choice of curve allows the same constructions.
\end{proof}
\subsection{Symplectic structure on $BG_{\bullet}$}
In the paper \cite{zhu}, M. Cueca and C. Zhu presented a canonical $2$-shifted symplectic structure on $ BG_{\bullet}$. We can resume their work as: 
\begin{definition}
    Let $\fg$ be a Lie algebra. If we equip $\fg$ with a symmetric and non-degenerate bilinear map $\langle*, *\rangle : \fg^2\longrightarrow \R$, such that:
    $$\langle a, [b,c]
    \rangle=\langle c, [a,b]
    \rangle$$
    then we say that $(\fg, \langle*, *\rangle)$ is a \underline{quadratic Lie algebra}. 
\end{definition}
\begin{theorem}
    Let $G$ be a Lie group, such that its corresponding Lie algebra $\fg$ is quadratic with a certain $\langle*, *\rangle$. Then, the following differential form on $ BG_{\bullet}$ gives us a $2$-shifted symplectic structure, i.e. a well-defined closed and non-degenerate form:
    \begin{equation}\label{eq: form final}
        \left(0, -\frac{1}{6}\cdot \langle \omega^l , [\omega^l \wedge \omega^l ]\rangle, \langle d_0^*\omega^l, d^*_2 \omega^r\rangle, 0, \dots \right)
    \end{equation}
\end{theorem}
\begin{proof}
    See \cite{zhu}, \cite{weinstein}. 
\end{proof}
\begin{remark}
    In \cite{zhu}, the second non-zero part of the form is given by $\langle d_2^*\omega^l, d^*_0 \omega^r\rangle$. This difference simply comes from an alternative definition of the nerve of a category. Following our definition, the right notation is the one given in (\ref{eq: form final}). 
\end{remark}
We want to show in the following that in some cases, such as for $G=SU(n)$, the representative of the first Pontryagin class on $BG_n$ we calculated is equal up to a coefficient to this canonical symplectic form, if we equip $\fg$ with the Killing form (see the next definitions) to obtain a quadratic Lie algebra.
\\ To define the Killing form, we first need the following recall:
\begin{definition}
    We define the adjoint representation of $\fg$ as the following map:
    $$\mathfrak{ad}: \fg \longrightarrow \text{Der}(\fg)\qquad \mathfrak{ad}(x)= [x, *] =: \mathfrak{ad}_x$$
\end{definition}
\begin{remark}
    It can be proved that this map is the derivation of $Ad$ (as defined in \ref{adjoint representation}) the adjoint representation of $G$. 
\end{remark}
\begin{definition}
    The Killing form is a bilinear map on $\fg$ given as:
    $$\mathcal{K}(X, Y):= \tr_{\R}(\mathfrak{ad}_X\circ \mathfrak{ad}_Y)$$
\end{definition}
Recall that the Lie algebra corresponding to $SU(n)$ is the space of skew-hermitian matrices with the trace equal to $0$. This space is denoted $\mathfrak{su}(n)$. 
\begin{lemma}\label{killing form}
   On $\mathfrak{su}(n)$, for $n\geq 2$, we have:
   $$\mathcal{K}(X, Y)= 2 n \cdot \tr(X\cdot Y)$$
\end{lemma}
\begin{proof}
    Let $X$, $Y$ and $Z$ be three $n\times n$ matrices, with respective entries $x_{ij}$, $y_{ij}$ and $z_{ij}$. Then, we calculate easily that:
    $$(Y\cdot Z)_{ij}= \sum_{r=1}^n y_{ir}\cdot z_{rj}\qquad (X\cdot Y\cdot Z)_{ij}=\sum_{k=1}^n\sum_{r=1}^n x_{ik}\cdot y_{kr}\cdot z_{rj}$$
    Thus, we obtain that:
    \begin{equation}\label{eq: developp of Killing}
        \begin{aligned}
            ([X, [Y, Z]])_{ij}=& (X\cdot Y\cdot Z - X\cdot Z\cdot Y - Y\cdot Z\cdot X + Z\cdot Y\cdot X)_{ij}\\
        =& \sum_{k=1}^n\sum_{r=1}^n x_{ik}\cdot y_{kr}\cdot z_{rj} - x_{ik}\cdot z_{kr}\cdot y_{rj} - y_{ik}\cdot z_{kr}\cdot x_{rj} + z_{ik}\cdot y_{kr}\cdot x_{rj}
        \end{aligned}
    \end{equation}
    We want here the trace of the map $Z\mapsto [X, [Y, Z]]$ for given $X$ and $Y$. To do that, we have to find a real basis of $\mathfrak{su}(n)$. Remark that for $X\in \mathfrak{su}(n)$, $X$ is uniquely defined by the values of $x_{ij}$ for $i\leq j$ as it is skew-hermitian. Since $x_{ij}$ for $i<j$ can be any complex number, independent from each other, it follows that we need two real coordinates $x_{ij}^R$ and $x_{ij}^I$ for each $x_{ij}$ corresponding to the real and the imaginary part. In the case of $x_{ii}$, we know for a matrix in $\mathfrak{su}(n)$ that they are all imaginary (and thus described by a unique real coordinate) and that their sum is equal to zero, implying that the diagonal of $X$ is uniquely defined by the imaginary parts of $x_{1,1}, \dots, x_{n-1, n-1}$. Counting all these parameters, we deduce that the real dimension of $\mathfrak{su}(n)$ is $n^2-1$. \\
    Let $A$ be any matrix over any vector space $V$ with the basis $(e_1, \dots, e_n)$. Then the element $a_{ii}$ of the matrix is given by the $i$-th element of $A\cdot e_i$.  In the same way, to calculate the elements of the diagonal of the matrix representing the map $Z\mapsto [X, [Y, Z]]$, we need to consider the real parts of the $i, j$-th entries of $[X, [Y, \mathbb{1}_{i, j}]]$ such as the imaginary parts of the $i, j$-th entries of $[X, [Y, i\cdot \mathbb{1}_{i, j}]]$ for all $i<j$ such as the imaginary parts of the $i, i$-th entries of $[X, [Y, i\cdot \mathbb{1}_{i, i} - i\cdot \mathbb{1}_{n, n}]]$ for $i<n$. Here $\mathbb{1}_{i, j}$ is the matrix such that $a_{k, l}=1$ if $k=i, l=j$ and $a_{k,l}=0$ else. \\
    Let us start with the case $i<j$. If we consider the real part, and thus calculate the $i,j$-th entry of $[X, [Y, \mathbb{1}_{i, j}]]$, we obtain using the equation (\ref{eq: developp of Killing}):
    \begin{equation}\label{eq: first of Killing}
        \sum_{k=1}^n (x_{ik}\cdot y_{ki}) - x_{ii}\cdot  y_{jj} - y_{ii}\cdot x_{jj} + \sum_{k=1}^n ( y_{jk}\cdot x_{kj})
    \end{equation}
    For $a, b$ two complex numbers, the real part of $a\cdot b$ is $a^R\cdot b^R- a^I\cdot b^I$. If we now restrict to the real part of the formula (\ref{eq: first of Killing}), we have:
    \begin{equation}\label{eq: result real part}
        \sum_{k=1}^n (x_{ik}^R\cdot y_{ki}^R- x_{ik}^I\cdot y_{ki}^I) - \underbrace{x_{ii}^R\cdot  y_{jj}^R}_{=0, \text{ since in $\mathfrak{su}(n)$}} + x_{ii}^I\cdot  y_{jj}^I - \underbrace{y_{ii}^R\cdot x_{jj}^R}_{=0, \text{ since in $\mathfrak{su}(n)$}} + y_{ii}^I\cdot x_{jj}^I + \sum_{k=1}^n ( y_{jk}^R\cdot x_{kj}^R- y_{jk}^I\cdot x_{kj}^I)
    \end{equation}
    If we consider the imaginary part, and thus calculate the $i,j$-th entry of $[X, [Y, i\cdot \mathbb{1}_{i, j}]]$, the equation (\ref{eq: developp of Killing}) gives us: 
    $$\sum_{k=1}^n (i\cdot x_{ik}\cdot y_{ki}) - i\cdot x_{ii}\cdot  y_{jj} - i\cdot y_{ii}\cdot x_{jj} + \sum_{k=1}^n ( i\cdot y_{jk}\cdot x_{kj})$$
    Considering the imaginary part of it, we obtain again (\ref{eq: result real part}) (we are multiplying the formula (\ref{eq: first of Killing}) with $i$, thus the imaginary part is equal to the real part of (\ref{eq: first of Killing})), that is:
    $$\sum_{k=1}^n (x_{ik}^R\cdot y_{ki}^R- x_{ik}^I\cdot y_{ki}^I)  + x_{ii}^I\cdot  y_{jj}^I  + y_{ii}^I\cdot x_{jj}^I + \sum_{k=1}^n ( y_{jk}^R\cdot x_{kj}^R- y_{jk}^I\cdot x_{kj}^I)$$
    Let us now consider the case of $i=j$. 
    $$\sum_{k=1}^n (i\cdot  x_{ik}\cdot y_{ki}) -2i\cdot x_{ii}\cdot y_{ii} +i\cdot x_{in}\cdot y_{ni}+i\cdot y_{in}\cdot x_{ni}+ \sum_{k=1}^n (i\cdot y_{ik}\cdot x_{ki})$$
    If we consider the imaginary part:
    \begin{align*}
        \sum_{k=1}^n (x_{ik}^R\cdot y_{ki}^R - x_{ik}^I\cdot y_{ki}^I ) +& \sum_{k=1}^n (x_{ki}^R\cdot y_{ik}^R - x_{ki}^I\cdot y_{ik}^I)\\
        &+2 x_{ii}^I\cdot y_{ii}^I - x_{in}^I\cdot y_{ni}^I - y_{in}^I\cdot x_{ni}^I +x_{in}^R\cdot y_{ni}^R + y_{in}^R\cdot x_{ni}^R
    \end{align*}
    And this whole formula is equal to:
    $$2\cdot \sum_{k=1}^n (x_{ik}^R\cdot y_{ki}^R - x_{ik}^I\cdot y_{ki}^I )+2 x_{ii}^I\cdot y_{ii}^I - 2\cdot x_{in}^I\cdot y_{ni}^I+2\cdot x_{in}^R\cdot y_{ni}^R$$
    since in $\mathfrak{su}(n)$, $x_{ij}^R= -x_{ji}^R$ and $x_{ij}^I=x_{ji}^I$. If we now sum all the elements of the diagonal we have obtained, we get:
    \begin{align*}
        \sum_{1\leq i<j\leq n} &\left(2\cdot \sum_{k=1}^n (x_{ik}^R\cdot y_{ki}^R- x_{ik}^I\cdot y_{ki}^I)  + 2\cdot x_{ii}^I\cdot  y_{jj}^I  + 2\cdot y_{ii}^I\cdot x_{jj}^I + 2\cdot\sum_{k=1}^n ( y_{jk}^R\cdot x_{kj}^R- y_{jk}^I\cdot x_{kj}^I)\right) \\
        &\qquad\qquad\qquad\qquad+ \sum_{i=1}^{n-1} \left(2\cdot \sum_{k=1}^n (x_{ik}^R\cdot y_{ki}^R - x_{ik}^I\cdot y_{ki}^I )+2 x_{ii}^I\cdot y_{ii}^I - 2\cdot x_{in}^I\cdot y_{ni}^I+2\cdot x_{in}^R\cdot y_{ni}^R\right)
    \end{align*}
    Remark that:
    \begin{align*}
        \sum_{1\leq i<j\leq n}(2\cdot x_{ii}^I\cdot  y_{jj}^I  + 2\cdot y_{ii}^I\cdot x_{jj}^I ) +\sum_{i=1}^{n-1} 2 x_{ii}^I\cdot y_{ii}^I  &= 2\cdot \sum_{i=1}^n\sum_{j=1}^n x_{ii}^I\cdot y_{jj}^I - 2\cdot x_{nn}^I\cdot y_{nn}^I\\
        &=  2\cdot \sum_{i=1}^n x_{ii}^I\cdot \underbrace{\sum_{j=1}^n y_{jj}^I}_{=0 \text{ since }Y\in \mathfrak{su}(n)}- 2\cdot x_{nn}^I\cdot y_{nn}^I\\
        &= - 2\cdot x_{nn}^I\cdot y_{nn}^I
    \end{align*}
    The formula becomes now:
    \begin{equation}\label{eq: sum of both}
        \begin{aligned}
            \sum_{1\leq i<j\leq n} &\left(2\cdot \sum_{k=1}^n (x_{ik}^R\cdot y_{ki}^R- x_{ik}^I\cdot y_{ki}^I)+ 2\cdot\sum_{k=1}^n ( y_{jk}^R\cdot x_{kj}^R- y_{jk}^I\cdot x_{kj}^I)\right)\\&\qquad\qquad+\sum_{i=1}^{n-1} \left(2\cdot \sum_{k=1}^n (x_{ik}^R\cdot y_{ki}^R - x_{ik}^I\cdot y_{ki}^I ) - 2\cdot x_{in}^I\cdot y_{ni}^I+2\cdot x_{in}^R\cdot y_{ni}^R\right)- 2\cdot x_{nn}^I\cdot y_{nn}^I
        \end{aligned}
    \end{equation}
    Since we are in $\mathfrak{su}(n)$, $x_{ij}^R= -x_{ji}^R$ and $x_{ij}^I=x_{ji}^I$ and thus, the second part of (\ref{eq: sum of both}) is equal to:
    $$\sum_{i=1}^{n-1}\left( 2\cdot \sum_{k=1}^n (x_{ik}^R\cdot y_{ki}^R - x_{ik}^I\cdot y_{ki}^I )\right) -\sum_{k=1}^{n-1}( 2\cdot x_{nk}^I\cdot y_{kn}^I)+\sum_{k=1}^{n-1}(2\cdot x_{nk}^R\cdot y_{kn}^R)- 2\cdot x_{nn}^I\cdot y_{nn}^I +\underbrace{2\cdot x_{nn}^R\cdot y_{nn}^R}_{=0}$$
    which is simply equal to:
    \begin{equation}\label{first of final}
       2\cdot\sum_{i=1}^{n}  \sum_{k=1}^n (x_{ik}^R\cdot y_{ki}^R - x_{ik}^I\cdot y_{ki}^I )= 2\cdot \tr(X^R\cdot Y^R)- 2\cdot \tr(X^I\cdot Y^I) 
    \end{equation}
     where $X^R+i\cdot X^I=X$ are respectively the real and the imaginary part of $X$. If we come back to the first part of the formula (\ref{eq: sum of both})), using that $y_{j,k}^R\cdot x_{k,j}^R- y_{j,k}^I\cdot x_{k,j}^I = x_{j,k}^R\cdot y_{k,j}^R - x_{j,k}^I\cdot y_{k,j}^I$ by the properties of $\mathfrak{su}(n)$, we obtain:
     \begin{equation}\label{eq: count sum}
         \sum_{1\leq i< j\leq n} \left(2\cdot \sum_{k=1}^n (x_{ik}^R\cdot y_{ki}^R- x_{ik}^I\cdot y_{ki}^I)+ 2\cdot\sum_{k=1}^n (  x_{j,k}^R\cdot y_{k,j}^R - x_{j,k}^I\cdot y_{k,j}^I)\right)
     \end{equation}
    For a given $i$, $(x_{ik}^R\cdot y_{ki}^R- x_{ik}^I\cdot y_{ki}^I)$ comes in the first sum once for each $j>i$, so $n-i$ times. For a given $j$, $x_{j,k}^R\cdot y_{k,j}^R - x_{j,k}^I\cdot y_{k,j}^I$ comes in the second sum once for each $i<j$ so $j-1$ times. Thus, for a given $t$, $(x_{tk}^R\cdot y_{kt}^R- x_{tk}^I\cdot y_{kt}^I)$ comes $n-t+t-1=n-1$ times as summand in (\ref{eq: count sum}). Thus we can write it as:
    \begin{equation}\label{second of final}
        2\cdot (n-1)\sum_{t=1}^n \sum_{k=1}^n(x_{tk}^R\cdot y_{kt}^R- x_{tk}^I\cdot y_{kt}^I)= 2\cdot(n-1)\cdot \left(\tr(X^R\cdot Y^R)- \tr(X^I\cdot Y^I)\right)
    \end{equation}
    Thus the addition of the two parts (\ref{first of final}) and (\ref{second of final}) gives us:
    $$\mathcal{K}(X,Y)= 2\cdot n \cdot\left( \tr(X^R\cdot Y^R)- 2\cdot n \cdot  \tr(X^I\cdot Y^I)\right) $$
    Finally, let us see that:
    \begin{align*}
        \tr(X^R\cdot Y^I)&= \sum_{i=1}^n\sum_{k=1}^n x_{ik}^R\cdot y_{ki}^I\\
        &= \sum_{1\leq i<k\leq n}\left(x_{ik}^R\cdot y_{ki}^I +\underbrace{x_{ki}^R}_{= - x_{ik}^R}\cdot \underbrace{y_{ik}^I}_{=y_{ki}^I}\right) +\sum_{i=1}^n \underbrace{x_{ii}^R}_{=0}\cdot y_{ii}^I\\
        &=0
    \end{align*}
    and analogous for $\tr(X^I\cdot Y^R)$. Thus: 
    \begin{align*}
      \mathcal{K}(X,Y)&= 2\cdot n \cdot\left( \tr(X^R\cdot Y^R) +i\cdot\underbrace{\tr(X^R\cdot Y^I)}_{=0} + i\cdot  \underbrace{\tr(X^I\cdot Y^R)}_{=0}+ (i)^2 \cdot 2\cdot n  \tr(X^I\cdot Y^I)\right)\\&= 2\cdot n \cdot\tr(X\cdot Y)  
    \end{align*}
    So the statement is proved. 
\end{proof}
This formula is not true for all the matrix spaces, since the base and the dimension of the space change. For example, for $\mathfrak{so}(n)$, the space of skew-symmetric matrices, which is the Lie algebra of $SO(n)$, $O(n)$ and $\text{Spin}(n)$, we have:
\begin{lemma}
    On $\mathfrak{so}(n)$, for $n\geq 3$, the Killing form is given by:
    $$\mathcal{K}(X, Y)= (n-2)\cdot \tr(X\cdot Y)$$
\end{lemma}
\begin{proof}
   The main idea of the proof is analogous to the lemma \ref{killing form}, where we do ignore the imaginary parts. 
\end{proof}
 
\begin{remark}
    It is easy to check that the Killing form on $\mathfrak{su}(n)$ is symmetric, non-degenerate and fulfils the equation:
    $$2n \cdot \tr(A\cdot [B,  C])= 2n\cdot \tr(A\cdot B \cdot C) - 2n \cdot \tr(A\cdot C\cdot B) = 2n\cdot \tr(C\cdot A\cdot B ) - 2n \cdot \tr( C\cdot B\cdot A) = 2n\cdot  \tr(C\cdot [A, B])$$
    implying that $(\mathfrak{su}(n), c\cdot \mathcal{K})$ is a quadratic Lie algebra for any $c\in \R\setminus \{0\}$. \\
    The Killing form on $\mathfrak{so}(n)$ has the same properties. 
\end{remark}
Let us apply it to the symplectic form presented in \cite{zhu} to prove the following:
\begin{theorem}\label{final theorem}
    Let $G=SU(n)$ with $n\geq 2$. Let respectively $\Psi$ be the symplectic form as in (\ref{eq: form final}) (using $c\cdot \mathcal{K}$ for $\langle*, *\rangle$, with $c\in \R\setminus \{0\}$) and let $\Phi$ be the representative of the first Pontryagin class for $BG_{\bullet}$ as in Corollary \ref{pontryagine for comparison}. Then, we have: 
    $$-\frac{1}{16 n c \pi^2} \Psi= \Phi$$
\end{theorem}
\begin{proof}
    It is enough to prove the two following points:
    \begin{itemize}
        \item $\frac{1}{96 nc \pi^2}\cdot \langle \omega^l , [\omega^l \wedge \omega^l ]\rangle = \frac{1}{24 \pi^2}\tr(\omega^l\cdot \omega^l\cdot \omega^l)$
        \item $-\frac{1}{16 nc \pi^2}\langle d_0^*\omega^l, d^*_2 \omega^r\rangle= \frac{1}{8\pi^2}\tr(d_2^*\omega^r\cdot d_0^*\omega^l)$
    \end{itemize}
    For the first point: Recall by definition:
    \begin{align*}
        \langle \omega^l , [\omega^l \wedge \omega^l ]\rangle(v_1, v_2, v_3)&= \langle *, * \rangle\circ \omega^l \wedge [\omega^l \wedge \omega^l](v_1, v_2, v_3)\\
        &= \langle *, * \rangle\circ \frac{1}{2}\sum_{\sigma\in \mathfrak{S}_3}\text{sign}(\sigma) \omega^l(v_{\sigma(1)})\otimes [\omega^l\wedge\omega^l ](v_{\sigma(2)}, v_{\sigma(3)})\\
       &= \langle *, * \rangle\circ \sum_{\sigma\in \mathfrak{S}_3}\text{sign}(\sigma) \omega^l(v_{\sigma(1)})\otimes [\omega^l(v_{\sigma(2)})\wedge\omega^l(v_{\sigma(3)})]\\
    \end{align*}
    where the last equation comes from the definition of the $[*\wedge*]$-product. Now, we see, using the property of the quadratic map $\langle*, * \rangle$, that $\langle a,[b,c]\rangle = \langle c,[a,b]\rangle = -\langle c, [b, a]\rangle$. Thus:
    \begin{align*}
        \langle a_1, [a_2, a_3]\rangle = -\langle a_1, [a_3, a_2]\rangle = \langle a_2, [a_3, a_1]\rangle = -\langle a_2, [a_1, a_3]\rangle= \langle a_3, [a_1, a_2]\rangle= - \langle a_3, [a_2, a_1]\rangle
    \end{align*}
    And thus, for any $\sigma\in \mathfrak{S}_3$, we have:
    $$\langle a_1, [a_2, a_3]\rangle= \text{sign}(\sigma)\langle a_{\sigma(1)}, [a_{\sigma(2)}, a_{\sigma(3)}]\rangle$$
    And then we finish with:
    $$\langle \omega^l , [\omega^l \wedge \omega^l ]\rangle(v_1, v_2, v_3)= 6\cdot \langle \omega^l(v_1) , [\omega^l(v_2), \omega^l(v_3) ]\rangle$$
    Then, applying the Killing form for $\mathfrak{su}(n)$, we obtain: 
    \begin{align*}
        12n c\cdot \tr(\omega^l(v_1)\cdot [\omega^l(v_2), \omega^l(v_3)])&= 12n c\cdot \tr(\omega^l(v_1)\cdot \omega^l(v_2)\cdot  \omega^l(v_3)) - 12nc \cdot \tr(\omega^l(v_1)\cdot \omega^l(v_3)\cdot  \omega^l(v_2))
    \end{align*}
    Because of the symmetry of $\tr(*\cdot *)$, we have for any square matrices $A, B$ and $C$:
    $$\tr(ABC)= \tr(BCA)= \tr(CAB)\qquad \text{and} \qquad \tr(BAC)=\tr(ACB)=\tr(CBA)$$
    Thus $\tr(A_1A_2A_3)=\tr(A_{\sigma(1)}A_{\sigma(2)}A_{\sigma(3)})$ if $\text{sign}(\sigma)=1$. We obtain:
    \begin{align*}
        12nc \cdot \tr(\omega^l(v_1)\cdot [\omega^l(v_2), \omega^l(v_3)])&= 4nc \cdot \sum_{\sigma\in \mathfrak{S}_3}\text{sign}(\sigma)\tr(\omega^l(v_{\sigma(1)})\cdot \omega^l(v_{\sigma(2)}) \cdot \omega^l(v_{\sigma(3)}))\\&= 4nc\cdot \tr(\omega^l\cdot \omega^l\cdot \omega^l)(v_1, v_2, v_3)
    \end{align*}
    Finally, we multiply by $\frac{1}{96n c\pi^2}$:
    $$\frac{1}{96 nc \pi^2}\cdot \langle \omega^l , [\omega^l \wedge \omega^l ]\rangle = \frac{1}{24 \pi^2}\tr(\omega^l\cdot \omega^l\cdot \omega^l)$$
    So we proved the first point. \\
    Let us come to the second point:
    \begin{align*}
        \langle d_0^*\omega^l, d^*_2 \omega^r\rangle(v_1, v_2)&= \langle*,*\rangle\circ d_0^*\omega^l\wedge  d^*_2 \omega^r (v_1, v_2)\\
        &= 2nc \cdot \tr(*\cdot *)\circ \left(d_0^*\omega^l (v_1)\otimes  d^*_2 \omega^r(v_2) - d_0^*\omega^l(v_2)\otimes  d^*_2 \omega^r(v_1)\right)\\
        &= 2n c\cdot \tr(d_0^*\omega^l (v_1)\cdot   d^*_2 \omega^r(v_2)) - 2n c\cdot\tr(d_0^*\omega^l(v_2)\cdot  d^*_2 \omega^r(v_1) ) \\
        &= 2n c\cdot \tr( d^*_2 \omega^r(v_2)\cdot d_0^*\omega^l (v_1)) - 2n c\cdot\tr(  d^*_2 \omega^r(v_1) \cdot d_0^*\omega^l(v_2) ) \\
        &=  -2nc \cdot \tr(*\cdot *) \circ  d^*_2 \omega^r \wedge d_0^* \omega^l (v_1, v_2)\\
        &= -2nc\cdot \tr( d^*_2 \omega^r \cdot d_0^* \omega^l)(v_1, v_2)
    \end{align*}
    Now, we multiply by the coefficient $-\frac{1}{16nc\pi^2}$:
    $$-\frac{1}{16 nc \pi^2}\langle d_0^*\omega^l, d^*_2 \omega^r\rangle= \frac{1}{8\pi^2}\tr(d_2^*\omega^r\cdot d_0^*\omega^l)$$
    And so we prove the statement.
\end{proof}
Remark that with $c= - \frac{1}{16n \pi^2}$, we get the identity $\Psi= \Phi$. 
\begin{corollary}
    Let $G=O(n)$, $SO(n)$ or $\text{Spin}(n)$ for $n\geq 3$. Let respectively $\Psi$ be the symplectic form as in (\ref{eq: form final}) (using $c\cdot \mathcal{K}$ for $\langle*, *\rangle$, with $c\in \R\setminus \{0\}$) and let $\Phi$ be the first Pontryagin class for $BG_{\bullet}$. Then, we have: 
    $$-\frac{1}{8 (n-1)c \pi^2} \Psi= \Phi$$
\end{corollary}
\begin{proof}
    The proof is the same as for the theorem \ref{final theorem}, but one should take care that the formula of the Killing form is different
\end{proof}
Remark that with $c= - \frac{1}{8(n-1) \pi^2}$, we get the identity $\Psi= \Phi$.
\begin{remark}
    If we use the definition $\langle*, *\rangle:= \tr(*\cdot *)$ instead of the Killing form $c\cdot \mathcal{K}$, we obtain the equality
$$-\frac{1}{8 \pi^2} \Psi= \Phi$$
for $Su(n)$ such as $O(n)$, $SO(n)$ or $\text{Spin}(n)$. Remark that $\frac{1}{8 \pi^2}$ is simply the coefficient arising naturally from the Pontryagin polynomial.
\end{remark}

\newpage~
\thispagestyle{empty}
\vspace*{1cm}
\vfill  
~\\
\begin{large} 
\textbf{Selbstst\"andigkeitserkl\"arung}
\end{large}
~\\
Hiermit erkl\"are ich, dass ich die vorliegende Arbeit
selbstst\"andig und nur unter Verwendung der angegebenen
Quellen und Hilfsmittel verfasst habe.\\
~\\
G\"ottingen, den 15.12.2022 \hspace*{4cm} UNTERSCHRIFT\\\\

\hspace*{8cm}Milor, Abel Henri Guillaume\\

\hspace*{8cm} \includegraphics[scale=0.4]{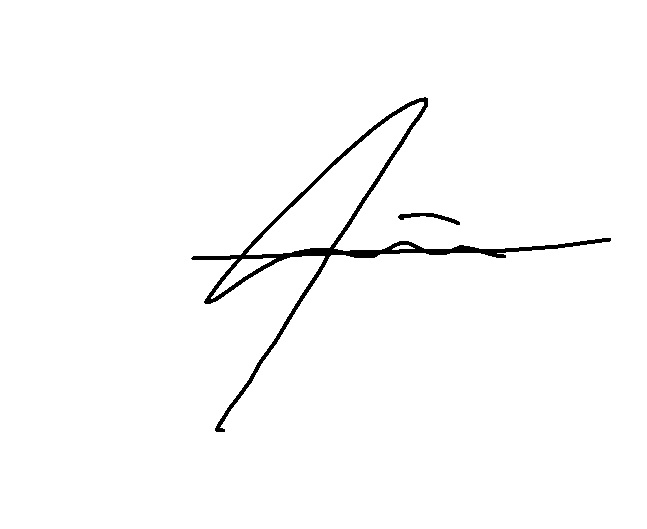}

\end{document}